\documentclass[12pt]{article}
\usepackage{amsthm}
\usepackage{amssymb}
\usepackage{amsxtra}     
\usepackage{epsfig}
\usepackage{verbatim}
\usepackage[lmargin=1.2in, rmargin=1.2in, tmargin=1.2in, bmargin=1.2in]{geometry}
\usepackage[all]{xypic}

\numberwithin{equation}{section}
\theoremstyle{plain}

\newtheorem{thm}{Theorem}
\newtheorem{prop}[thm]{Proposition}
\newtheorem{cor}[thm]{Corollary}
\newtheorem{lem}[thm]{Lemma}

\theoremstyle{definition}
\newtheorem{defi}[thm]{Definition}

\theoremstyle{remark}
\newtheorem{rmk}[thm]{Remark}
\numberwithin{thm}{section}

\newcommand{\adeles}{\mathbb{A}}

\newcommand{\moduli}{\mathcal{M}}
\newcommand{\hn}{\mathcal{H}_n}
\newcommand{\A}{\univav}
\newcommand{\OK}{\mathcal{O}_{\cmfield}}
\newcommand{\Hom}{\mbox{Hom}}
\newcommand{\dual}{^\vee}
\newcommand{\isomto}{\overset{\sim}{\rightarrow}}
\newcommand{\lp}{\left(}
\newcommand{\rp}{\right)}
\newcommand{\ci}{C^{\infty}}
\newcommand{\omci}{\underline{\omega}(\ci)}

\newcommand{\dr}{{H^1_{DR}}}
\newcommand{\drci}{\dr(\ci)}
\newcommand{\Split}{\mathrm{Split}}
\newcommand{\splci}{\Split(\ci)}
\newcommand{\splpa}{\Split(\padic)}
\newcommand{\fci}{\mathcal{F}(\ci)}
\newcommand{\etan}{\eta_n}
\newcommand{\zj}{z_j}
\newcommand{\ej}{e_j}
\newcommand{\zjp}{z_j'}
\newcommand{\ejp}{e_j'}
\newcommand{\s}{^*}
\newcommand{\aj}{\alpha_j}
\newcommand{\bj}{\beta_j}
\newcommand{\ajp}{\alpha_j'}
\newcommand{\bjp}{\beta_j'}
\newcommand{\ai}{\alpha_i}
\newcommand{\bi}{\beta_i}
\newcommand{\aip}{\alpha_i'}
\newcommand{\bip}{\beta_i'}

\newcommand{\lz}{L_z}
\newcommand{\az}{A_z}
\newcommand{\rhost}{\rho_{st}}
\newcommand{\uo}{\underline{\omega}}
\newcommand{\IKS}{I_{KS}}
\newcommand{\level}{\alpha}
\newcommand{\peltup}{(A, \lambda, \iota, \level)}

\newcommand{\padic}{p\mbox{-adic}}

\newcommand{\agauto}{\mbox{\bf M}}
\newcommand{\padicauto}{\agauto^{\padic}}

\newcommand{\JKS}{J_{KS}}
\newcommand{\ciop}{\partial(\rho, \ci, d)}
\newcommand{\ciopo}{\partial(\rho, \ci, 1)}
\newcommand{\paop}{\partial(\rho, \padic, d)}
\newcommand{\paopo}{\partial(\rho, \padic, 1)}
\newcommand{\ompa}{\uo(\padic)}
\newcommand{\x}{(X, \lambda, \iota, \alpha)}
\newcommand{\splxr}{\Split(X/R)}
\newcommand{\ir}{\iota_R}
\newcommand{\omxr}{\uo_{X/R}}
\newcommand{\drxr}{\dr(X/R)}
\newcommand{\strict}{\dag}
\newcommand{\pstrict}{$\mathrm(\dag)$}
\newcommand{\pstrictp}{$\mathrm({\ddag})$}

\newcommand{\uroot}{\underline{U}}
\newcommand{\Oig}{\mathcal{O}_{T_{\infty, \infty}}}
\newcommand{\Oigg}{\Oig^n}

\newcommand{\omcan}{\uo_{can}}

\newcommand{\hermpair}{\langle \bullet, \bullet\rangle}
\newcommand{\cpctopen}{K}

\newcommand{\cmfield}{\mathcal{K}}
\newcommand{\qexpn}{\sum_{h\in H\dual} \lp c(h)q^h\rp}
\newcommand{\qexpnp}{\sum_{h\in H\dual} \lp c(h)q^{ph}\rp} 
\newcommand{\qexpringb}{\left(\OK\right)_{(p)}[[q, H\dual_{\geq 0}]]}
\newcommand{\qexpringp}{\left(\OK\right)_{(p)}((q, H\dual_{\geq 0}))}
\newcommand{\tateav}{\mbox{Mum}_{L}(q)}
\newcommand{\tate}{\mbox{Tate}}
\newcommand{\omm}{\{\omega_i^-\}_{i=1}^n}
\newcommand{\omp}{\{\omega_i^+\}_{i=1}^n}

\newcommand{\LL}{\mathcal{L}}
\newcommand{\EE}{\mathcal{E}}
\newcommand{\uA}{\underline{A}}
\newcommand{\glnc}{GL_n(\IC)}
\newcommand{\gln}{GL_n}

\newcommand{\EAR}{\EE_{\uA/R}}
\newcommand{\EARp}{\EAR^+}
\newcommand{\EARm}{\EAR^-}
\newcommand{\EARpm}{\EAR^{\pm}}
\newcommand{\EARVr}{\EE_{(\uA, V, \rho)/R}}
\newcommand{\OM}{\mathcal{O}_{\moduli}}
\newcommand{\EEVr}{\EE_{V, \rho}}
\newcommand{\Vlam}{V_{\Lambda}}
\newcommand{\rholam}{\rho_{\Lambda}}
\newcommand{\e}{\bf{e}}

\newcommand{\mformsxrd}{(\uo^-_{X/R})^{\rho^-\otimes\rho_{st}^d}\otimes(\uo^+_{X/R})^{\rho^+\otimes\rho_{st}^d}}
\newcommand{\mformsnost}{(\uo^-)^{\rho_-}\otimes(\uo^+)^{\rho_+}}
\newcommand{\rmrp}{\rho_-\otimes\rho_+}
\newcommand{\mformspa}{\mformsnost(\padic)}
\newcommand{\mformsst}{(\uo^-)^{\rho_-\otimes\rhost}\otimes(\uo^+)^{\rho_+\otimes\rhost}}
\newcommand{\ev}{\epsilon_\nu}
\newcommand{\KSh}{_{\cpctopen}\mathbb{S}}
\newcommand{\pmr}{P_{m,r}}
\newcommand{\mr}{_{m,r}}
\newcommand{\ig}{T_{\infty, \infty}}
\newcommand{\ptors}{\mu_{p^{\infty}}^g}

\newcommand{\me}{m_E}
\newcommand{\mep}{m_{E'}}
\newcommand{\cmtype}{\Sigma}
\newcommand{\gothsig}{\mathfrak{S}}
\newcommand{\scripto}{\mathcal{O}}
\newcommand{\gm}{\mathbb{G}_m}
\newcommand{\uq}{\underline{q}}
\newcommand{\lie}{\mbox{Lie}}
\newcommand{\Lie}{\lie}
\newcommand{\IR}{\mathbb{R}}
\newcommand{\ZZ}{\mathbb{Z}}
\newcommand{\IC}{\mathbb{C}}
\newcommand{\IQ}{\mathbb{Q}}
\newcommand{\univav}{A_{\mathrm{univ}}}
\newcommand{\image}{\mbox{Image}}
\newcommand{\id}{\mbox{Id}}
\newcommand{\sym}{\mbox{Sym}}
\newcommand{\End}{\mbox{End}}
\newcommand{\shfunctor}{_{\cpctopen}\mathbb{A}_V}
\newcommand{\incl}{\mbox{incl}}
\newcommand{\smargin}[1]{ }
\newcommand{\an}{\mathrm{an}}
\newcommand{\hol}{\mathrm{hol}}
\newcommand{\eval}{\mathrm{eval}}

\newcommand{\Isom}{\mathrm{Isom}}
\newcommand{\Spec}{\mathrm{Spec}}
\newcommand{\freealgrt}{R\langle T_1, \ldots, T_n\rangle}
\newcommand{\noncomalg}{\mathfrak{R}}
\newcommand{\rvaluedpt}{\underline{A}}
\newcommand{\ptopmoduli}{_{\cpctopen}\mathbb{S}(G, X)}
\newcommand{\et}{{\mbox{\'et}}}

\def\tr{{\rm tr}\,}

\begin{document}
\bibliographystyle{amsalpha} 

\title{$p$-adic Differential Operators on Automorphic Forms on Unitary Groups
 \\
{\it Op\'erateurs diff\'erentiels $p$-adiques sur formes automorphes pour groupes unitaires}}
\author{Ellen E. Eischen\thanks{This research was partially supported by a fellowship from the Lucent Foundation.}}
\date{April 25, 2011}

\maketitle

\begin{abstract}
The goal of this paper is to study certain $p$-adic differential operators on automorphic forms on $U(n,n)$.  These operators are a generalization to the higher-dimensional, vector-valued situation of the $p$-adic differential operators constructed for Hilbert modular forms by N. Katz.  They are a generalization to the $p$-adic case of the $\ci$-differential operators first studied by H. Maass and later studied extensively by M. Harris and G. Shimura.  The operators should be useful in the construction of certain $p$-adic $L$-functions attached to $p$-adic families of automorphic forms on the unitary groups $U(n)\times U(n)$.

 \begin{center}{\bf R\'esum\'e}\end{center}
 \hspace{.4cm} Nous construisons certains op\'erateurs diff\'erentiels $C^{\infty}$ et leurs analogues $p$-adiques, qui agissent sur des formes automorphes (\`a valeurs vectorielles ou scalaires) pour les groupes unitaires $U (n, n)$.  Nous \'etudions des propri\'et\'es de ces op\'erateurs, et nous les utilisons \`a prouver quelques th\'eor\`emes arithmetiques.  Ces op\'erateurs diff\'erentiels sont une g\'en\'eralisation au cas $p$-adique des op\'erateurs diff\'erentiels $C^{\infty}$ \'etudi\'es d'abord par H. Maass et \'etudi\'es ensuite en d\'etail par M. Harris et G. Shimura.  Ils sont une g\'en\'eralisation au cas \`a valeurs vectorielles des op\'erateurs diff\'erentiels $p$-adiques construits dans le cas des formes modulaires par N. Katz.  Ils devraient \^etre utiles dans la construction de certaines fonctions $L$ $p$-adiques, en particulier les fonctions $L$ $p$-adiques attach\'ees aux familles p-adiques de formes automorphes pour les groupes unitaires $U (n)\times U (n)$.
\end{abstract}

\tableofcontents

\section{Introduction}
\label{theintro} 

%

The goal of this paper is to study certain $p$-adic differential operators on automorphic forms on $U(n,n)$.  This is one step in an ongoing project to construct certain $p$-adic $L$-functions attached to $p$-adic families of automorphic forms on $U(n)\times U(n)$.  For example, in analogue with \cite{kaCM}, these differential operators will be used in \cite{apptoSHL} to generalize the construction of the Eisenstein measure in \cite{SHL}.  This, in turn, gives a more general construction of the $L$-functions in \cite{EEHLS} than the one proposed in \cite{SHL}. 

The differential operators in this paper eliminate some of the restrictions on the extent to which the construction of $p$-adic $L$-functions proposed in \cite{SHL} can be generalized to construct more general $2$- or $3$-variable $p$-adic $L$-functions attached to families of automorphic forms:
\begin{enumerate}
\item{As the introduction of \cite{SHL} notes, M. Harris, J.-S. Li, and C. Skinner interpolate the $L$-function at a fixed point $s_0$; removal of the restriction that $s_0$ be fixed requires the differential operators that are the topic of this paper.  This issue will be addressed in \cite{apptoSHL}.}
\item{No one has constructed $p$-adic $L$-functions attached to vector-valued automorphic forms, only scalar-valued automorphic forms.  The differential operators are expected to make this generalization possible.}
\item{In particular, no one has constructed $p$-adic $L$-functions attached to vector-valued families of {\it non}-ordinary automorphic forms other than modular forms.  This project originated from my attempt to construct (two- and three-variable) $p$-adic $L$-functions attached to certain families of overconvergent automorphic forms on $U(n,n)$.  This paper can be viewed as a step in that project, which has turned out to be more widely applicable.}
\end{enumerate}

The length of this paper is due to the fact that much of the ``background" material  provided here is not recorded elsewhere but is necessary for the discussion in this paper.  The details in the background sections should also serve as a reference for others.  For example, our discussion of the Kodaira-Spencer morphism is more explicit than elsewhere in the literature and includes a discussion at the level of coordinates.  To date, this does not appear elsewhere in the literature, but it is important for understanding the action of the differential operators at the level of coordinates (rather than just as abstract maps).  Also, this paper provides a user's guide to $q$-expansions and the ``Mumford object," a generalization of the Tate curve to the higher dimensional setting.  The prior literature \cite{la} on Mumford objects and algebraic $q$-expansions is at the level of existence statements.  Since our intended applications require a more explicit description of the Mumford object, we provide one here.  

\subsection{Motivation}

As mentioned above, our motivation for studying the differential operators comes from $L$-functions.  Through the doubling method, one can express special values of $L$-functions in terms of a finite sum of special values of Eisenstein series.   So if each of the finitely many terms in the sum is algebraic (or $p$-integral, or lies in a desired ring) up to a period, then the same holds for the special values of the $L$-function.  This is the approach taken in \cite{kaCM, SHL}.  Many constructions of $p$-adic $L$-functions rely on $p$-adic interpolation of special values of Eisenstein series \cite{serre, ka3, kaCM, pa, SHL}.  In the case of holomorphic Eisenstein series, the $p$-adic interpolation often takes place through $p$-adic interpolation of Fourier coefficients.  This approach is in general not sufficient, though, because of the following issue: most special values of $L$-functions come from non-holomorphic Eisenstein series, which do not have Fourier expansions.  This is closely related to the reason that the $L$-functions $L(s, f)$ in \cite{SHL} are only $p$-adically varied at a fixed point $s = s_0$.

%

For the case of the unitary groups $U(n,n)$, the $p$-adic differential operators in this paper can be used to solve this issue.\footnote{This paper generalizes almost immediately to the case of the symplectic groups $Sp(n)$, and it should be relatively straightforward to generalize to $U(m,n)$.  In the symplectic case, similar operators were also discussed in \cite{padiffops1, padiffops2}.}  These differential operators are a $p$-adic analogue of a class of $\ci$-differential operators first studied by H. Maass \cite{maass, mas} and later studied extensively by G. Shimura \cite{sh94b, shido, sh84, sh81, sh81, sh84, shar} and M. Harris \cite{ha86, hasv}.  (In the case of modular forms on the upper half plane, these $\ci$-differential operators are the widely used operators $g\mapsto y^{-k}(\frac{\partial}{\partial z})(y^k g)$ that map a weight $k$ modular form $g$ to a weight $k+2$ modular form.  More generally, they map a vector- or scalar-valued automorphic function to an automorphic function of a different weight.)  These $\ci$-differential operators play an important role in Shimura's proofs of algebraicity properties of Eisenstein series and $L$-functions.  Shimura's proofs, however, do not provide insight into $p$-adic properties.  

For Hilbert modular forms, N. Katz \cite{kaCM} reformulates Shimura's $\ci$-differential operators in terms of the Gauss-Manin connection and the Kodaira-Spencer isomorphism.  This algebraic-geometric approach is useful because it allows Katz to construct a $p$-adic analogue of the $\ci$-differential operators for Hilbert modular forms.  This paper generalizes \cite{kaCM} to the setting of automorphic forms on $U(n,n)$, including the more general case of {\it vector}-valued forms.  The author intends to apply these operators to construct some of the more general $p$-adic $L$-functions mentioned above.  The differential operators allow one to show that the values of a certain {\it $p$-adic} -- in general non-algebraic -- function at CM points over $\mathcal{O}_{\IC_p}$ are in fact not only algebraic but also the {\it same} as the values of a closely related {\it $\ci$} -- in general non-holomorphic -- function at CM points over $\mathcal{O}_{\IC_p}$.  So special values of pairs of seemingly unrelated functions (namely, a $\ci$ non-holomorphic function and $p$-adic function) are meaningfully compared and shown to be equal.  The possibility of constructing the $p$-adic $L$-functions relies upon this remarkable feature.

The starting point for my construction and proofs is \cite{kaCM}; however, the higher-dimensional, vector-valued situation is more complicated and involves several obstacles not encountered in the one-dimensional case considered in \cite{kaCM}.  My generalization involves a more delicate use of the Kodaira-Spencer morphism (which, for the unitary case, is no longer an isomorphism) than in Katz's situation.  Also, unlike in \cite{kaCM}, the action of the operator on $q$-expansions is no longer in terms of a derivation on a commutative ring, but rather a map (in general, not a derivation) on a ring that is in general non-commutative; I formulate the precise action of this map on coefficients of vector-valued $q$-expansions so that the description of the resulting coefficients can be used to study $p$-adic interpolation.  (Similarly, there are other instances in which a commutative ring in \cite{kaCM} is replaced with a non-commutative one in my situation.)  

\subsection{Organization of the paper}

Sections \ref{review} through \ref{padicandIg} cover background information, much of which is not available elsewhere in the literature.  In addition to being necessary for our discussion, these sections should be a useful reference for other researchers.  In Section \ref{review}, we review automorphic forms from several perspectives.  Section \ref{GMKS} discusses the Gauss-Manin connection and Kodaira-Spencer morphism in more detail than one finds elsewhere in the literature.  In particular, we give an explicit example, which is useful for the explicit description of the differential operators given later.  No similarly explicit examples are currently found in the literature.  In Section \ref{qexpnsection}, we discuss Fourier expansions and the algebraic theory of $q$-expansions.  We describe the ``Mumford object," the higher dimensional analogue of the Tate curve; this should also serve as a ``users' guide" for others seeking an explicit description of the Mumford Object.  In Section \ref{padicandIg}, we discuss $p$-adic automorphic forms on unitary groups, in a format analogous to the one in \cite{kaCM}.

Our discussion of the differential operators occurs in Sections \ref{ShimurasDiffOps} through \ref{cmsection}.  In Section \ref{ShimurasDiffOps}, we briefly review Shimura's definition of $\ci$-differential operators.  The material in Sections \ref{Kapuralg} through \ref{cmsection} builds upon the material in \cite{kaCM}.  In Section \ref{kapa}, we give an explicit formula for the action of the differential operators on vector-valued $p$-adic automorphic forms.  This differential operator, unlike the one in Katz, is not a derivation; indeed, it acts on a non-commutative ring.  We obtain a higher-dimensional, vector-valued analogue of Ramanujan's operator $q\frac{d}{dq}$.  

\subsection{Notation}

\label{notation}

Our setup is exactly the same as the setup in Sections 0 and 1 of \cite{SHL}, though our notation is not always the same.  When appropriate, we adopt the notation of \cite{shar}, \cite{hasv}, \cite{kaCM}, \cite{SHL}, and \cite{hida}.

Fix a quadratic imaginary extension $\cmfield$\index{$\cmfield$, cmfield} of $\IQ$, and let $\OK$\index{$\OK$} denote the ring of integers in $\cmfield$.  Fix a CM type $\cmtype$\index{$\cmtype$, cmtype} of $\cmfield$, i.e. an embedding
\begin{align}\label{kintoq}
\cmfield\hookrightarrow\IC. 
\end{align}
We associate $\cmfield$ with its image in $\IC$ under the embedding \eqref{kintoq}.  Let $\delta_{\cmfield}$ denote the discriminant of $\cmfield$.  Unless otherwise noted, we will always use $R_0$\index{$R_0$} to denote an $\OK$-algebra.  Let $\alpha$\index{$\alpha$} be a generator of $\OK$ over $\ZZ$. 

Let $\bar{\IQ}$ denote the algebraic closure of $\IQ$ in $\IC$, and write $\incl_{\infty}: \bar{\IQ}\hookrightarrow\IC$ to denote the given embedding of $\bar{\IQ}$ in $\IC$.  Fix a prime ideal $(p)$ in $\ZZ$ that splits completely in $K$.  Denote the finite adeles of $\IQ$ by $\adeles_f$, and write $\adeles_f^p$ to denote the restricted product $\prod'\IQ_l$ over finite primes $l\neq p$.  Let $\IC_p$ denote the completion of an algebraic closure of $\IQ_p$, and fix an embedding
$\incl_p: \bar{\IQ}\hookrightarrow\IC_p.$
Given an embedding
$\sigma: \cmfield\hookrightarrow\bar{\IQ},$
we write $\sigma_{\infty}$ to denote the embedding
$\incl_{\infty}\circ\sigma: \cmfield\hookrightarrow\IC,$
and we write $\sigma_p$ to denote the embedding
$\incl_p\circ\sigma: \cmfield\hookrightarrow\IC_p.$
We write $\bar{\sigma}$ to denote the composition of the embedding $\sigma$ with complex conjugation.

We now set some notation for modules.  For any module $M$, we denote $\sum_{e=0}^\infty M^{\otimes e}$ by $T(M)$.  From now on, for any module $M$, we associate $\sym(M) = \sum_e \sym^eM$
 with its image in 
$T(M) = \sum_e T^e(M)$
via
\begin{align}
\sym^e(M)&\hookrightarrow M^{\otimes e}\label{syminc}\\
x_1\cdot \cdots \cdot x_e &\mapsto \sum_{s\in S_e} x_{s(1)}\otimes\cdots\otimes x_{s(e)}\nonumber,
\end{align}
where $S_e$ is the group of permutations of $1, \ldots, e$.  Let $r$ be a positive integer, and let $V$ be a vector space containing vectors $v_i$ indexed by a subscript $i$.  For any $r$-tuple $\lambda = (\lambda_1, \ldots, \lambda_r)$ of integers, let $v_{\lambda}$ denote the tensor product $v_{\lambda_1}\otimes\cdots \otimes v_{\lambda_r}$ of $r$ vectors $v_{\lambda_1}, \ldots, v_{\lambda_r}$.  Denote the symmetric product $v_1^{\lambda_1}\cdots v_n^{\lambda_n}$  by $v^{\lambda}.$  We write $\rhost$ or simply $st$ to denote the standard representation of $GL(V)$ on a vector space $V$.

In parts of the paper dealing with a complex analytic approach, we primarily use Shimura's notation (used throughout his papers, e.g. as discussed in the Notation and Terminology section of \cite{shar}).  We review some of it here.  For a ring $R$ and positive integers $r$ and $c$, we write $R^r_c$\index{$R^r_c$} to denote the $R$-module of $r\times c$-matrices with entries in $R$, i.e. a matrix with $r$ rows and $c$ columns.  When necessary to distinguish between column and row vectors, we shall take advantage of this notation.  For a matrix $z$ with entries in $\IC$, we write ${^tz}$\index{${^tz}$} to denote the transpose of $z$ and $z\s$\index{$z\s$, z$\backslash$s} to denote the complex conjugation ${^t\bar{z}}$ of ${^tz}$.  

Throughout the paper, fix a positive integer $n$\index{$n$} and set 
$g\index{$g$} = 2n.$
We write $1_n$\index{$1_n$} to mean the $n\times n$ identity matrix.  As in \cite{shar}, let $\hn\index{$\hn$, hn} =  \{z\in \IC^n_n\mid i(z^* - z)>0\}$
and
$\eta_n\index{$\etan$, etan}  =  \left( \begin{array}{cc}
 0 &  -1_n\\
1_n & 0 \end{array}\right).$
The following notation will also be helpful.  If $A$ is a matrix, let $A^+$ denote $A$, and let $A^-$ denote $^tA$.  Given a subgroup $G$ of $GL_n(\IR)$, let $G^+$ denote the subgroup of $G$ consisting of elements of positive determinant.

We now set some conventions for schemes.  For any morphism of schemes $\pi: Y\rightarrow Z$, let $(\Omega^\bullet_{Y/Z}, d)$ denote the complex of sheaves of relative differentials on $Y/Z$ (where $d$ is the usual differentiation map).  The de Rham cohomology $H^i_{DR}(Y/Z)$ is defined to be the hypercohomology $\mathbb{R}^i\pi_*(\Omega_{Y/Z}^\bullet).$  Given a scheme $X$ over a scheme $S$ and a scheme $T$ over $S$, we denote by $X_T$ the scheme $X\times_S T$.  When working with a separated scheme $S$ of finite type over $\IC$, we write $S^{\an}$ to denote the associated complex analytic space.  We then write $\mathcal{O}_S^{\hol}$ or $\mathcal{O}_S(\hol)$ (resp. $\mathcal{O}_S^{\ci}$ or $\mathcal{O}_S(\ci)$) to denote the sheaf of holomorphic (resp. $\ci$) functions on $S^{\an}$.

%


%

\begin{center}{\textsc{Acknowledgments}}\end{center}
I would like to thank Christopher Skinner for many useful conversations and insightful suggestions.  I am also grateful to Michael Harris for helpful advice.  Nicholas Ramsey provided many stylistic comments on earlier drafts of this paper.   This research was supported by a Bell Labs Graduate Research Fellowship through the Lucent Foundation.

\section{Certain abelian varieties of PEL type and automorphic forms}
\label{review}



In this section, following the perspectives of  \cite{shar}, \cite{kaCM}, and \cite{hida}, we discuss automorphic forms on the unitary groups $U(n,n)$ and certain abelian varieties with PEL structure.

\subsection{Unitary groups}
In order to discuss automorphic forms on unitary groups and abelian varieties of PEL type, we need first to establish conventions for unitary groups.  In this subsection, we recall the notation and conventions concerning unitary groups given in Section (0.1) of \cite{SHL}; all the material in Section (0.1) of \cite{SHL} applies to our situation.  Let $V$ be an $n$-dimensional vector space over $\cmfield$, and let $\hermpair_V$ be a non-degenerate hermitian pairing on $V$ relative to the extension $\cmfield/\IQ$.   We write $-V$ to denote the vector space $V$ over $\cmfield$ with hermitian pairing $\hermpair_{-V}$ defined by
$\hermpair_{-V} = -\hermpair_V.$
We write $2V$ to denote the $\cmfield$-vector space $V\oplus V$ with the hermitian pairing
$\hermpair_{2V}$
defined by
\begin{align*}
\langle(v_1, v_2), (w_1, w_2)\rangle_{2V} = \langle v_1, w_1\rangle_V + \langle v_2, w_2\rangle_{-V}\\
( = \langle v_1, w_1\rangle_V - \langle v_2, w_2\rangle_{V})
\end{align*}                                                                                                                                                                                                                                                                                                                                                                                                                                                                                                                                                                                                                                                                                              
for all vectors $v_1, v_2, w_1, w_2$ in $V$.

The Hermitian pairing $\hermpair_{V}$ defines an involution $c$ on $\End(V)$ via
$\langle gv, v'\rangle = \langle v, c(g)v'\rangle$
for all $g$ in $\End(V)$ and $v, v'\in V$.  Note that for any $\IQ$-algebra $R$, the involution $c$ extends to an involution of $V\otimes_{\IQ}R$.

For any vector space $W$ with hermitian pairing $\hermpair_W$ and $\IQ$-algebra $R$, we define the following unitary groups over $\IQ$:
\begin{align*}
&U(W)(R) = U(W, \hermpair_W)(R)\\
&= \left\{g\in GL\left(W\otimes_{\IQ}R\right)\mid \langle gv, gv'\rangle = \langle v, v'\rangle, \mbox{ for all } v, v'\in W\right\}\\
&GU(W)(R)  = GU(W, \hermpair_W)(R)\\
& =  \left\{g\in GL\left(W\otimes_{\IQ}R\right)\mid \mbox{ for all } v, v'\in W, \langle gv, gv'\rangle = 
\nu(g)\langle v, v'\rangle \mbox{ with } \nu(g)\in R^{\times} \right\}.
\end{align*}
Then
\begin{align*}
U(2V)(R) &\cong U(n, n)(R) = \left\{g\in GL_{2n}(K\otimes_{\IQ} R)\mid g\eta_n g\s = \eta_n\right\}\\
GU(2V)(R) &\cong GU(n,n)(R) =\left\{g\in GL_{2n}(K\otimes_{\IQ} R)\mid g\eta_n g\s = \nu(g)\eta_n\mbox{ some } \nu(g)\in R^{\times} \right\}. 
\end{align*}

\subsection{Certain abelian varieties of PEL type}\label{seczero}

In this subsection, we review certain abelian varieties of PEL type.  

Our situation is similar to the setup in Sections 1.2 through 1.4 of \cite{SHL}.  We review here the most important features of the setup in \cite{SHL}, following \cite{SHL} closely.  Our notation in this section is not entirely the same as the notation in \cite{SHL}.  (For details on the material covered in this section, the reader may also find it helpful to look at chapters 1 and 2 of \cite{la} and at chapters 6 and 8 of \cite{milne}.)


Let $G = GU(V).$  Fix a compact open subgroup $\cpctopen = \cpctopen_p\times\cpctopen^p$\index{$\cpctopen$, cpctopen} of $GU(V)(\adeles_f)$, with $\cpctopen_p\subseteq G(\IQ_p)$ a hyperspecial maximal open compact subgroup of $G(\IQ_p)$ and $K^p$ in $G(\adeles_f^p)$.  (In applications, the maximal compacts of interest will be those used in \cite{SHL}.)  We recall the functor $\shfunctor$\index{$\shfunctor$, shfunctor} from schemes $S$ over $\IQ$ to the category of sets that is given in (1.3.1) of \cite{SHL}:
\begin{align}\label{peltuple}
S\mapsto \{\peltup\}
\end{align}
where
\begin{itemize}
\item{$A$ is an abelian scheme over $S$, up to isogeny}
\item{$\lambda: A\rightarrow A\dual$ is a polarization}
\item{$\iota: \cmfield\rightarrow \End_S(A)\otimes \IQ$ is an embedding of $\IQ$-algebras}
\item{$\alpha: V\otimes\adeles_f\isomto \prod_lT_l(A)\otimes\IQ$ is an isomorphism of $\cmfield$-spaces, modulo the action of $\cpctopen$.}
\end{itemize}
The above data are required to satisfy the Rosati condition, i.e. the following diagram commutes:
\begin{align*}
\xymatrix{
 A\ar[r]^{\lambda}\ar[d]_{\iota(\bar{a})}& A\dual\ar[d]^{\iota\dual(a)}\\
 A \ar[r]_{\lambda} & A\dual.
 }
\end{align*}
Furthermore, the isomorphism $\alpha$ must identify the Hermitian pairing on $V$ with a multiple of the one coming from the Weil pairing associated to $\lambda$.  We call the tuples $\peltup$ abelian varieties of PEL type.

Given a vector space $W$ over $\cmfield$ with a non-degenerate hermitian pairing on $W$, one can canonically associate to the group $G = GU(W)$ a Shimura datum $(G, X)$ and a Shimura variety $Sh(W)=Sh(G, X)$.  The complex-valued points of $Sh(G,X)$ are given by
$Sh(G, X)(\IC) = \lim_K G(\IQ)\backslash X\times G(\adeles_f)/\cpctopen,$
where the limit is over all open compact subgroups of $G(\adeles_f)$.  Let 
$_KSh = _{\cpctopen}Sh(V) = _{\cpctopen}Sh(G, X)$
 be the variety whose complex points are given by $G(\IQ)\backslash X\times G(\adeles_f)/\cpctopen$; the complex points of this variety classify complex abelian varieties satisfying the above moduli problem.  If $U$ is a compact closed subgroup of $GU(\adeles_f)$, we use the notation $_USh(V)$ to mean the tower of the varieties $_KSh = _KSh(V)$ with $K\subset U$ a compact open.

We remind the reader of Theorem (1.3.2) in \cite{SHL}, which is originally due to Shimura:
\begin{thm}Whenever $\cpctopen$ is sufficiently small (``neat", in the sense of \cite{la}, suffices), the functor $\shfunctor$ is representable by a quasi-projective scheme $_K\moduli$ over $\IQ$.  The scheme $_K\moduli$ is the canonical model for $_KSh(V)$.  As $\cpctopen$ varies, the natural maps between the above functors induce the natural maps between the varieties $_KSh(V)$.  The action of $GU(V)(\adeles_f)$ on $_KSh(V)$ preserves the $\IQ$-rational structure.
\end{thm}

Now we consider a similar but slightly different moduli problem that is discussed in Section (1.6) of \cite{SHL}; it will be useful for the $p$-adic theory, which we will discuss later.  Fix a sufficiently small compact open subgroup $\cpctopen = \cpctopen_p\times \cpctopen^p\subset G(\adeles_f)$ as above.
Consider the above moduli problem with $\iota$ replaced by an injection
$(\OK)_{(p)}\hookrightarrow \End_S(A)\otimes\ZZ_{(p)}.$
As explained in \cite{kottwitz}, this moduli problem is represented by a smooth integral scheme $\ptopmoduli$ over $\ZZ_{(p)}$, which is a smooth integral model for $_KSh(G,X)$.  This moduli problem is closely related to the moduli problem we now describe.  Let $_{\cpctopen^p}A^p$ be the functor $S\mapsto\{(A, \lambda, \iota, \alpha^p)\}$ with $A$ an abelian scheme over $S$ up to prime-to-$p$ isogeny, $\lambda$ a polarization of degree prime to $p$, $\iota: \left({\OK}\right)_{(p)}\rightarrow \End_S(A)\otimes\ZZ_{(p)}$ an embedding of $\ZZ_{(p)}$-algebras, and $\alpha^p: V(\adeles^p_f)\isomto V^{f, p}(A)$ a prime-to-$p$ ${\left(\OK\right)}_{(p)}$-linear level structure modulo $\cpctopen^p$.  This functor is representable over $\ZZ_{p}$ by a scheme also denoted $_K\mathbb{S}(G,X)$.  The forgetful map gives an isomorphism $\shfunctor\isomto_{K^p}A^p$.

We denote by $\moduli$\index{$\moduli$, moduli} a moduli space over an $\OK$-algebra in the situation where we want to remain ambiguous about the level structure or any other details about the moduli problem; we also use this notation when it is clear from context which moduli space we mean.  We write $\univav$\index{$\univav$, univav} to denote the universal abelian variety over $\moduli$:
\begin{align}\nonumber
\xymatrix{
A_{univ}\ar[d]^{\pi}\\
 \moduli.
 }
\end{align}

We define 
\begin{align}
\uo\index{$\uo$, uo}:=\pi_*(\Omega_{A_{univ}/\moduli})\nonumber\\
\dr\index{$\dr$, dr}:=\dr(A_{univ}/\moduli).\nonumber
\end{align}

Note that we will {\it always} take $\moduli$ to be over an $\OK$-algebra.
Working over $\OK$ affords us the following convenient splittings.  The embedding
$\iota: \cmfield\rightarrow \End_S(A)\otimes \IQ$
makes $\uo$ and $\dr$ into $\OK$-modules through the action defined by
$a\cdot v = \iota(a)\s(v)$
for each $a\in \OK$.
Similarly, since $A$ lies over an $\OK$-scheme $S$, the action on $\uo$ and $\dr$ induced by the composition of morphisms of structure sheaves
$\OK\rightarrow\mathcal{O}_S\rightarrow\mathcal{O}_A$
also makes $\uo$ and $\dr$ into $\OK$-modules.  The isomorphism
\begin{align*}
\OK\otimes_{\IQ}\OK&\isomto \OK\oplus\OK\\
a\otimes b &\mapsto (ab, \bar{a}b)
\end{align*}
extends to an isomorphism
\begin{align}\label{oksp}
\OK\otimes_{\IQ}\mathcal{O}_S&\isomto \mathcal{O}_S\oplus\mathcal{O}_S\\
a\otimes b&\mapsto (ab, \bar{a}b).\nonumber
\end{align}
So there is a splitting over $\mathcal{O}_S$
$\uo = \uo^+\oplus \uo^-,$
where $\uo^+$ is the $\mathcal{O}_S$-module defined by
$\uo^+ := \{w\in\uo \vert \iota(a)^*w = aw, a\in\OK\}$
and $\uo^-$ is given by
$\uo^- := \{w\in\uo \vert \iota(a)^*w = \bar{a}w, a\in\OK\}.$
There is a similarly defined splitting
$\dr = \dr^+\oplus \dr^-$.  We shall sometimes denote $\dr^\pm$ by $H^{\pm}$.


\subsection{The complex analytic viewpoint (Shimura's perspective)}\label{reviewsym}

In this section, for a fixed open compact $\cpctopen$, let $\Gamma$ be the congruence subgroup of $GU(\eta_n)$ defined by
$\Gamma = \cpctopen\cap G(\IQ).$

\subsubsection{Transcendental description of abelian varieties of PEL type}\label{trasec}

In this section, we follow the approach of Shimura in \cite{shar} and \cite{shab}.
Let 
$\underline{A} = (A, \lambda, \iota, \alpha)$
be a complex abelian variety of PEL type.  Then $A^{\an}$ is a complex torus $\IC^{2n}/\mathcal{L}$, for some $\ZZ$-lattice $\mathcal{L}$ in $\IC^{2n}$,
which can be obtained as follows.  Let $\omp$ be a basis for $\uo_{A/\IC}^+$, and let $\omm$ be a basis for $\uo_{A/\IC}^-$.  Then, we define the $\ZZ$-lattice $L(A, \omp, \omm)$ to be
\begin{align*}
L(A, \omp, \omm) = \left\{\left( \begin{array}{cc}
\int_{\gamma}\omega_1^- \\
\vdots\\
\int_{\gamma}\omega_n^-\\
\int_{\gamma}\omega_1^+\\
\vdots\\
\int_{\gamma}\omega_n^+ \end{array}\right) \middle\vert \gamma\in H_1(A, \ZZ) \right\}.
\end{align*}
The complex abelian variety $A^{\an}$ is isomorphic to $\IC^g/L(A, \omp, \omm)$.  The polarization $\lambda$ on $A$ corresponds to a Riemann form on $L(A, \omp, \omm)$.  The morphism
$\iota: \cmfield\hookrightarrow \End_{\IQ}(A)$
corresponds to the action $\iota$ of $\cmfield$ on $L(A, \omp, \omm)$ given by
$\iota(a) v = \Psi(a)\cdot v$
for all $a$ in $\cmfield$ and $v$ in $L(A, \omp, \omm)$, where
$\Psi: \cmfield\rightarrow\IC^{2n}_{2n}$
is given by
$a\mapsto diag[\bar{a}\cdot 1_n, a\cdot 1_n].$
The level structure
$\alpha: V(\adeles_f)\isomto V^f(A) \mod \cpctopen$
corresponds to a map
$V(\adeles_f)\isomto L(A, \omp, \omm)\otimes \adeles_f \mod \cpctopen,$
with a compatibility of pairings as above.

\subsubsection{Families of complex abelian varieties of PEL type}

We now recall Shimura's construction (Section 4 of \cite{shar}) of some families of complex abelian varieties of PEL type.  Throughout this section, fix a $\ZZ$-lattice $L$ in $\cmfield^1_{2n}$.

For each $z\in \hn$ and each row vector $x$ in $\cmfield^1_{2n}$, let $p_z(x)$ be the vector in $\IC^{2n}$ defined by
\begin{align*}
p_z(x) = \lp [z\hspace{0.4cm} 1_n]x\s,\hspace{0.3cm} [{^tz}\hspace{0.4cm} 1_n]\cdot {^tx} \rp.
\end{align*}
The function $p_z(x)$ is holomorphic in $z$.

We define
$\mathcal{L}_L(z) = p_z(L).$
Then $p_z(L)$ is a lattice in $\IC^g$.  Let $A_z$ be the complex torus defined by
$A_z = \IC^{2n}/p_z(L).$
Let $C_z$ be the polarization on $A_z$ given by the Riemann form $E_z$ defined by
$E_z\lp p_z(x), p_z(y)\rp = \tr_{(\cmfield\otimes_{\IQ}\IR)/\IR}\lp x\eta_n y\s\rp.$  Every complex-valued point $(A,\lambda, \iota, \alpha)$ of the Shimura variety $_KSh(V)$ is isomorphic for some $z$ to $A_z$ with polarization $C_z$ and action of $\cmfield$ given by
$\iota_z(a) \cdot v = diag[\bar{a}\cdot 1_n, a\cdot 1_n]\cdot v$.
To specify a finite ordered set of points $t_1(z)\ldots, t_s(z)$ of finite order on $A_z$, it is equivalent to specify a finite set of elements $u_1, \ldots, u_s$ in $L\otimes\IQ/L$ such that
$t_i(z) = p_z(u_i).$

We conclude this section by giving a more explicit classification of analytic families of complex abelian varieties of PEL type.  Fix a finite set of points $\{u_1, \ldots, u_s\}$ in $\cmfield^1_{2n}$.
Consider the quintuple
\begin{align}\label{omega}
\Omega = \left\{\cmfield, \Psi, L, \eta_n, \{u_i\}_{i=1}^s\right\}.
\end{align}
Such a quintuple is called a PEL-type.  Let
$P = (A, C, \iota; \{t_i\}_{i=1}^s)$
be a tuple consisting of a complex abelian variety $A$, a polarization $C$, a set of points $t_1, \ldots, t_s$ of finite order on $A$, and a ring injection
$\iota: \cmfield\hookrightarrow \End_{\IQ}(A)$
that is stable under the involution of $\End_{\IQ}(A)$ determined by $C$.

We say that $P$ is of type $\Omega$ if the following holds: there is a $\ZZ$-lattice $\Lambda$; a homomorphism $\xi: \IC^g\rightarrow A$; an $\IR$-linear isomorphism
$q: \IC^1_g\rightarrow\IC^g$
such that 
$q(ax) = \Psi(a)q(x)$ and
$\iota(a)\circ\xi = \xi\circ\Psi(a)$
for all $a\in \cmfield$ and $x\in \IC^1_g$ and such that $q(u_i) = t_i$ for each $i$; a Riemann form $E$ determined by $C$ that satisfies
$E(q(x), q(y)) = \tr_{K/\IQ}(x\eta_n y\s);$
and a commutative diagram (Figure (4.3) in \cite{shar})
\begin{align*}
\xymatrix{
0\ar[r]& L\ar[r]\ar[d]& \IC^1_g\ar[r]\ar[d]^q& \IC^1_g/L\ar[r]\ar[d]&0\\
0\ar[r]& \Lambda\ar[r]&\IC^g\ar[r]^{\xi}&A\ar[r]&0.
 }
\end{align*}
The classification of complex abelian varieties of type $\Omega$, which we now make precise, will be useful for understanding how the classical (analytic) definition of automorphic forms motivates the algebraic-geometric definition of automorphic forms.  For each $z$ in $\hn$, let
$P_z = (A_z, C_z, \iota_z; \{t_i(z)\}_{i=1}^s)$
with $A_z$, $C_z$, $\iota_z$, and $\{t_i(z)\}_{i=1}^s$ defined as above.  Then, $P_z$ is of type $\Omega$ for each $z$ in $\hn$.  Furthermore, Theorem \ref{thmfe} (Theorem 4.8 in \cite{shar}) classifies abelian varieties of type $\Omega$.
An element
$\alpha=\left( \begin{array}{cc}
 A &  B\\
C & D \end{array}\right)\in U(\eta_n)$
acts on $\hn$ by
$\alpha z = (Az+B)(Cz+D)^{-1}.$
\begin{thm}\label{thmfe}
The tuple $P_z$ is of type $\Omega$ for each $z$ in $\hn$, and every tuple of type $\Omega$ is isomorphic to $P_z$ for some $z$ in $\hn$.  Tuples $P_z$ and $P_w$ are isomorphic if and only if there is an element $\gamma$ in the group
$\Gamma = \{\alpha\in U(\eta_n)\vert L\alpha = L \mbox{ and } u_i\alpha - u_i\in L \mbox{ for each } i\}$
that satisfies
$w=\gamma z.$
\end{thm}

\begin{rmk}
Taking $L$ to be the lattice in $\cmfield^1_{2n}$ generated by the standard basis vectors $e_1, \ldots, e_{2n}$ and the vectors $\alpha\cdot e_1, \ldots, \alpha\cdot e_{2n}$ (with $\alpha$ a generator of $\cmfield$ over $\IQ$), we see that there is an analytic family of abelian varieties $\univav$\index{$\univav$, univav} over $\hn$ such that the fiber of $\A$ over each point $z = (z_{ij})$ in $\hn$ is the abelian variety 
$\az\index{$\az$, az}:=\IC^n/\lz,$
where $L_z$ is the $\ZZ$-lattice in $\IC^{2n}$ generated by:
\begin{align}\label{latticebasis1}
z_j\index{$\zj$, zj} & =  (z_{1j}, \ldots, z_{nj}, z_{j1}, \ldots, z_{jn}), \hspace{0.5cm}\\
e_j\index{$\ej$, ej} & = \mbox{vector with 1 in the $j$-th and $j+n$-th positions and zeroes everywhere else}\hspace{0.5cm} \\
z_j'\index{$\zjp$, zjp} & =  (\alpha^*z_{1j}, \ldots, \alpha^*z_{nj}, \alpha z_{j1}, \ldots, \alpha z_{jn})\hspace{0.5cm} \\
e_j'\index{$\ejp$, ejp} & =  \mbox{vector with $\alpha^*$ in $j$-th position, $\alpha$ in $j+n$-th position, and zeroes everywhere else},\label{latticebasis2}
\end{align}
with $j =1, \ldots,n$.
We will work with this family of abelian varieties in examples in future sections.
\end{rmk}

\subsubsection{Complex analytic automorphic forms}\label{cxanaut}
In this section, we recall the classical definition of automorphic forms over $\IC$, following \cite{shar}.

For $\alpha=\left( \begin{array}{cc}
 A &  B\\
C & D \end{array}\right)\in U(\eta_n)$ and $z\in\hn$, we define
$M_{\alpha}(z) = M(\alpha, z) = (\mu(\alpha, z), \lambda(\alpha, z)),$
where
$\mu(\alpha, z) = Cz+D$ and $\lambda(\alpha, z) = \bar{C}\cdot {^tz}+\bar{D}.$
Let $X$ be a finite-dimensional vector space, and fix a representation
$\omega: GL_n(\IC)\times GL_n(\IC)\rightarrow GL(X).$
For any map
$f: \hn\rightarrow X$
and $\alpha\in U(\eta_n)$, define
$f\Vert_{\omega}\alpha:\hn\rightarrow X$
by
$\left(f\middle\Vert_{\omega}\alpha\right)\left(z\right) = \omega\left(M_{\alpha}\left(z\right)\right)^{-1} f(\alpha z).$

Let $\Gamma$ be a congruence subgroup of $U(\eta_n)$.  A (holomorphic) automorphic form of weight $\omega$ with respect to $\Gamma$ is a holomorphic function $f:\hn\rightarrow X$
that satisfies
\begin{align}\label{mcond}
f\Vert_\gamma = f
\end{align}
for each $\gamma\in\Gamma$.
When $n=1$, we also require that $f$ is holomorphic at the cusps.

We denote the space of all (holomorphic) automorphic forms of weight $\omega$ with respect to $\Gamma$ by $\mathcal{M}_{\omega}(\Gamma)$, and we set
$\mathcal{M}_{\omega} =\cup_{\Gamma}\mathcal{M}_{\omega}(\Gamma).$
%
A $\ci$-automorphic form of weight $\omega$ with respect to a congruence subgroup $\Gamma$ is a $\ci$-function
$f:\hn\rightarrow X$
satisfying \eqref{mcond} for each $\gamma\in\Gamma$.

\subsection{Automorphic forms from another perspective}\label{earfirst}
Our discussion here is a generalization to abelian varieties (of type $\Omega$) of the situation for elliptic curves discussed in Appendix A1.1 of \cite{ka2}.  


Let $\Omega$ be as in \eqref{omega}.  Let $(\LL, E, \{t_i\}^{s}_{i=1})$ be a tuple consisting of a lattice $\LL$ in $\IC^g$ such that there is an $\IR$-linear isomorphism
\begin{align}\label{qlat1}
q:\IC^1_g\rightarrow\IC^g
\end{align} 
satisfying
$q(L) = \LL$
and
$q(ax) = \Psi(a)q(x)$
for all $a\in \cmfield$ and $x\in L$, a finite set of points $\{t_1, \ldots, t_s\}$ in $\LL\otimes\IQ$ such that $q(u_i) = t_i$ for all $i$, and a Riemann form $E = E_q$ on $\IC^g$ relative to $\LL$ such that
\begin{align}\label{qlat2}
E_q(q(x), q(y)) = \tr_{K/\IQ}(x\eta_n y\s)
\end{align}
for all $x, y\in \IC^g$.  We call such a tuple $(\LL, E, \{t_i\}_{i=1}^s)$ a {\it lattice of type} $\Omega$.    We define an action of $GL_n(\IC)\times GL_n(\IC)$ on the set of tuples $(\LL, E, \{t_i\}_{i=1}^s)$ of type $\Omega$ via
$\alpha\cdot (\LL, E, \{t_i\}_{i=1}^s) = (\alpha\LL, E^{\alpha}, \{\alpha t_i\}_{i=1}^s),$
where $E^{\alpha}$ is defined by
$E^\alpha(\alpha z,  \alpha w) = E(z, w).$
Observe that, modulo the action of $GL_n\times GL_n$ on lattices of type $\Omega$, there is a natural correspondence between lattices of type $\Omega$ and isomorphism classes of abelian varieties of type $\Omega$.
%

%

\begin{thm}
Fix a CM type $\Omega$.  Let $L$ be a $\ZZ$-lattice in $\IC^g$, and let $f$ be an automorphic form of weight $\omega$ with respect to a congruence subgroup $\Gamma$ of $U(\eta_n)$
containing the group 
\begin{align}\label{gammapr}
\Gamma' =  \{\alpha\in U(\eta_n)\vert L\alpha = L \mbox{ and } u_i\alpha - u_i\in L \mbox{ for each } i\}.
\end{align} There exists a unique function $f_L$ of lattices $(\LL\subset\IC^g, E, \{t_i\}_{i = 1}^s)$ of type $\Omega$ with $Nu_i\in \LL$ for all $i$, such that for each $\alpha\in GL_n(\IC)\times GL_n(\IC)$,
\begin{align}\label{flll}
f_L(\alpha \cdot (\LL, E, \{t_i\}_{i = 1}^s)) = \omega({^t\alpha})^{-1}f_L(\LL, E, \{t_i\}_{i = 1}^s)
\end{align}
and such that
\begin{align}\label{necpx}
f_L(p_z(L), E_z, \{p_z(u_i)\}_{i=1}^s) = f(z)
\end{align} for all $z$ in $\hn$.
\end{thm}

\begin{proof}
Let $(\LL\subset\IC^g, E, \{t_i\}_{i = 1}^s)$ be a lattice of type $\Omega$ such that $Nt_i\in \LL$ for all $i$, and let $q$ be as in \eqref{qlat1} through \eqref{qlat2}.  Then, as explained in the first paragraph of Theorem 4.8 of \cite{shar}, there is a diagonal matrix $S\in GL_{2n}(\IC)$ and an element $z\in\hn$ such that
\begin{align}\label{LLe}
q = S\cdot p_z,
\end{align}
i.e. such that $(\LL\subset\IC^g, E, \{t_i\}_{i = 1}^s) = S\cdot (p_z(L), E_{p_z}, p_z(u_i)).$
Therefore, if the function $f_L$ exists, then by \eqref{flll} and \eqref{necpx}, its value at $\LL$ {\it must} be $\omega({^tS})^{-1}f(z)$.  

Now we show that the function $f_L$ exists.  For this, it suffices to show that if there exist matrices $S$ and $T$ in $GL_n(\IC)\times GL_n(\IC)$ and elements $z$ and $w$ in $\hn$ such that
\begin{eqnarray}
(\LL\subset\IC^g, E, \{t_i\}_{i = 1}^s) = S\cdot (p_z(L), E_{p_z}, p_z(u_i))\label{psz}\\
\mbox{and}\nonumber\\
(\LL\subset\IC^g, E, \{t_i\}_{i = 1}^s) = T\cdot (p_w(L), E_{p_w}, p_w(u_i))\label{pwz},
\end{eqnarray}
then
\begin{align}\label{lrsh8}
\omega({^tT})^{-1}f(w) = \omega({^tS})^{-1}f(z).
\end{align}
For the remainder of the proof, suppose that both \eqref{psz} and \eqref{pwz} hold.  Then the abelian varieties of type $\Omega$ attached to $(p_z(L), E_{p_z}, p_z(u_i))$ and $(p_w(L), E_{p_w}, p_w(u_i))$ are both isomorphic to the abelian variety of type $\Omega$ attached to $(\LL\subset\IC^g, E, \{t_i\}_{i = 1}^s)$.  Therefore, by Theorem \ref{thmfe},
$w = \gamma z$
for some 
$\gamma\in \Gamma'\subset \Gamma\subset U(\eta_n).$
By line (4.31) of \cite{shar},
\begin{align}\label{pzL}
p_z(x\alpha) = {^tM(\alpha, z)}p_{\alpha z}(x)
\end{align}
for all $x\in \IC^g$ and $\alpha\in U(\eta_n)$.  Since $L = L\alpha$ for all $\alpha$ in $\Gamma'$, \eqref{pzL} shows that
$p_z(L) = {^tM(\gamma, z)}p_{w}(L).$
Since $u_i-\alpha u_i\in L$ for each $\alpha$ in $\Gamma'$,
$p_z(u_i\alpha) = p_z(u_i)$
for each $\alpha$ in $\Gamma'$.  Since $\gamma\in U(\eta_n)$,
$\gamma\eta_n\gamma\s = \eta_n,$
Hence, it immediately follows from the definition of $E_z$ that
$E_{p_z}(p_z(x\alpha), p_z(y\alpha)) = E_{p_z}(p_z(x), p_z(y)).$
So
\begin{align*}
(p_z(L), E_{p_z}, p_z(u_i)) = {^tM(\gamma, z)}\cdot (p_w(L), E_{p_w}, p_w(u_i))
\end{align*}
Therefore
$T = S \cdot  {^tM(\gamma, z)},$
so the right hand side of \eqref{lrsh8} is equal to
\begin{align}
\omega({^tT})^{-1}\omega(M(\gamma, z))f(z) = \omega({^tT})^{-1}f(\gamma z).\label{lhstoo}
\end{align}
Since $w = \gamma z$, the right hand side of \eqref{lhstoo} is equal to the left hand side of \eqref{lrsh8}.
So Equation \eqref{lrsh8} holds, which proves the existence of the function $f_L$.
\end{proof}

Having reformulated the definition of automorphic forms in terms of functions of lattices of type $\Omega$, we now reformulate it again in terms of functions of complex abelian varieties of type $\Omega$.  Observe that giving an ordered basis $\{\omega_i^{\pm}\}_{i=1}^n$ of $\uo^\pm$ is equivalent to giving an element of the module
$\EE_{\uA}^{\pm} = \Isom_{\IC}(\IC^n, \uo^{\pm}_{\uA}).$
(This equivalence is via $\langle e_i\mapsto\omega_i^{\pm}\rangle_{i=1}^{n}\leftrightarrow \{\omega_i^\pm\}_{i=1}^n$, with ${e_i}$ standard basis vectors in $\IC^n$.)  The group $GL_n(\IC)$ acts on $\EE_{\uA}\pm$ via
\begin{align}
(\alpha\cdot\lambda)(v) := \lambda({^t\alpha}\cdot v)\label{glact}
\end{align}
for all $v\in \IC^n$ and $\alpha\in GL_n(\IC)$.  We define $\EE_{\uA}$ by
$\EE_{\uA}: = \EE_{\uA}^-\oplus \EE_{\uA}^+.$
Then the action of $GL_n(\IC)$ on $\EE_{\uA}^{\pm}$ given in \eqref{glact} induces an action of $GL_n(\IC)\times GL_n(\IC)$ on $\EE_{\uA}$, and to give an element of $\EE_{\uA}$ is equivalent to specifying an ordered basis of $\uo^+$ and an ordered basis of $\uo^-$.

Let $\underline{\LL}(\uA)$ be the lattice of type $\Omega$ attached to $\uA$ as in Section \ref{trasec}.  Then we see that for each $\alpha\in GL_n(\IC)\times GL_n(\IC)$ and $\lambda\in \EE_{\uA}$,
\begin{align}\label{latab}
\alpha\cdot \underline\LL = \underline\LL(\uA, \alpha\cdot\lambda).
\end{align}

Let $(\omega, V)$ be a finite-dimensional representation of $\glnc\times\glnc$, and let 
$f:\hn\rightarrow V$
be an automorphic form of weight $\omega$ with respect to some congruence subgroup $\Gamma$ of $U(\eta_n)$ containing $\Gamma'$ (with $\Gamma'$ defined as in \eqref{gammapr}).  We define $F_L$ to be the unique function from pairs $(\uA, \lambda)$
(where $\lambda$ is an element of $\EE_{\uA}$) to $V$ satisfying both
$F_L(\uA, \alpha\lambda) = \omega({^t\alpha})^{-1}F_L(\uA, \lambda)$
and
$F_L(\uA, \lambda) = f
_L(\underline\LL(\uA, \lambda)).$
Thus, an automorphic form $f$ of weight $\omega$ on $\Gamma$ corresponds to a function $F$ from pairs $(\uA, \lambda)$ to $V$ satisfying
\begin{align}\label{dagmo}
F(\uA, \alpha\lambda) = \omega({^t\alpha})^{-1}F(\uA, \lambda).
\end{align}

Now we explain how to view functions $F$ satisfying \eqref{dagmo} as certain functions on abelian varieties $\uA$ of type $\Omega$.  Given a ring $R$ and a group $B$ that acts on $R$-modules $V_1$ and $V_2$, we define the contracted product $V_1\times^B V_2$ to be $V_1\oplus V_2$ modulo the relation $(v_1, v_2)\sim(bv_1, bv_2)$.  Let $H^{\omega} : = \glnc\times\glnc$ act on $\EE_{\uA}$ by the action induced by \eqref{glact}, and let $H^{\omega}$ act on $V$ via
$v\mapsto \omega({^th})^{-1}v.$
We define $\EE_{\uA, V, \omega}$ to be the contracted product
$\EE_{\uA, V, \omega} := \EE_{\uA}\times^{H^{\omega}} V.$
To give a function $F$ from pairs $(\uA, \lambda)$ to $V$ satisfying \eqref{dagmo} is equivalent to giving a function $\tilde{F}$ from abelian varieties $\uA$ of type $\Omega$ to $\EE_{\uA, V, \omega}$.  This equivalence is via
$\tilde{F}(A) = (\lambda, F(\uA, \lambda)).$

Letting $A^{\an}_{univ}$ be the universal family of abelian varieties of type $\Omega$ over $\Gamma\backslash\hn$,
we see that giving a holomorphic automorphic form $f$ is equivalent to giving a section of the $\mathcal{O}_{\hn}^{\hol}$-module
\begin{align*}
\EE_{V, \omega}^{\an} :=\EE_{A^{\an}_{univ}, V, \omega}\otimes \mathcal{O}_{\hn}^{\hol}.
\end{align*}
Similarly, giving a $\ci$-automorphic form $f$ is equivalent to giving a section of the $\mathcal{O}_{\hn}(\ci)$-module\begin{align*}
\EE_{V, \omega}(\ci) :=\EE_{A^{\an}_{univ}, V, \omega}\otimes \mathcal{O}_{\hn}(\ci).
\end{align*}

\subsection{Algebraic geometric approach to automorphic forms on unitary groups}\label{agafapproach}
The approach in this section is similar to the one in Section 1.2 of \cite{kaCM}.

Fix an $\OK$-algebra $R_0$.  Let $V$ be an $R_0$-module.  For any $R_0$-algebra $R$, we denote by $V_R$ the $R$-module $V\otimes_{R_0}R$ obtained by extension of scalars $R_0\rightarrow R$.  Let $(\rho, V)$ be an algebraic representation of $\gln\times\gln$ that is defined over $R_0$.  That is, for each $R_0$-algebra $R$, $\rho$ defines a homomorphism
\begin{align*}
\rho_R: \gln(R)\times\gln(R)\rightarrow GL(V_R)
\end{align*}
that commutes with extension of scalars $R\rightarrow R'$ of $R_0$-algebras.

Fix a compact open subgroup $\cpctopen$ of $G=GU(n,n)$.  We denote by $\Gamma$ the congruence subgroup $G(\IQ)\cap K$.  For each abelian variety $\uA = \peltup$ in $\shfunctor(R)$ over an $R_0$-algebra $R$, we now define modules $\EARp$, $\EARm$, and $\EAR$.  Like in Section \ref{earfirst}, we define
$\EARpm = \Isom_R(R^n, \uo^{\pm}_{A/R})$ and
$\EAR = \EARm\oplus \EARp.$
To give an element of $\lambda\in \EARpm$ is equivalent to specifying an ordered basis $\omega_1, \ldots, \omega_n$ of $\uo^{\pm}_{A/R}$; the equivalence is via
\begin{align*}
\lambda\in \EARpm \leftrightarrow \lambda(e_1), \ldots, \lambda(e_n) \in \uo^{\pm}_{A/R}.
\end{align*}
So giving element of $\EAR$ is equivalent to specifying an ordered basis of $\uo^-$ and an ordered basis of $\uo^+$.
The group $GL_n(R)$ acts on $\EARpm$ via
\begin{align}
(\alpha\cdot\lambda)(v) : = \lambda({^t\alpha}v).\label{glnact}
\end{align}
The action of $GL_n(R)$ given in \eqref{glnact} induces an action of $\gln(R)\times \gln(R)$ on $\EAR$.  Let $H^{\rho} = GL_n(R)\times GL_n(R)$ act on $V_R$ via $v\mapsto \rho({^t\alpha})^{-1}v$ and act on $\EAR$ through the action induced by \eqref{glnact}.  Similarly to in Section \ref{earfirst}, we denote by $\EARVr$ the $R$-module
$\EAR\times^{H^{\rho}}V.$
Observe that formation of $\EARVr$ commutes with extension of scalars $R\rightarrow R'$ of $R_0$-algebras.

\begin{defi}\label{autd1}An automorphic form of weight $\rho$, defined over $R_0$, is a function $f$ from the set of pairs $(\uA, \lambda)$, consisting of $\uA$ in $_K A_V(R)$ over an $R_0$-algebra $R$ and an element $\lambda$ in $\EAR$, to $V_R$ such that all of the following hold:
\begin{enumerate}
\item{The element $f(\uA, \lambda)$ depends only on the $R$-isomorphism class of $(\uA, \lambda)$.}
\item{The formation of $f(\uA, \lambda)\in V_R$ commutes with extension of scalars $R\rightarrow R'$ of $R_0$-algebras, i.e.
\begin{align*}
f(\uA\times_{\Spec R} R', \lambda\otimes_R R') = f(\uA, \lambda)\otimes_R 1 \in V\otimes_R R'.
\end{align*}}
\item{For each $(\uA, \lambda)$ over $R$ and $\alpha\in H^{\rho}(R)$,
$f(\uA, \alpha\lambda) = \rho({^t\alpha})^{-1}f(\uA, \lambda).$}
\end{enumerate}
We write $M_{\rho}(R_0)$ to denote the $R_0$-module of automorphic forms of weight $\rho$ defined over $R$.
\end{defi}

\begin{defi}\label{autd2}
An automorphic form of weight $\rho$ defined over $R_0$, is a rule $\tilde{f}$ that assigns to each $\uA$ in $_KA_V(R)$ over an $R_0$-algebra $R$ an element of $\EARVr$ such that both of the following conditions hold:
\begin{enumerate}
\item{The element$\tilde{f}(\uA)$ in $\EARVr$ depends only on the $R$-isomorphism class of $\uA$.}
\item{The formation of $\tilde{f}(\uA)$ commutes with extension of scalars $R\rightarrow R'$ of $R_0$-algebras, i.e.
\begin{align*}
\tilde{f}(\uA\times_{\Spec R}\Spec R') = \tilde{f}(\uA)\otimes_R 1\in \EARVr\otimes_R R'.
\end{align*}}
\end{enumerate}
\end{defi}
The equivalence of these Definition \ref{autd1} with Definition \ref{autd2} is through
$\tilde{f}(\uA) = (\lambda, f(\uA, \lambda)).$

The perspective of Definition \ref{autd2} leads us to another (equivalent) formulation of the definition of automorphic forms in the case where $_KA_V(R)$ is representable (i.e. in the case where $\cpctopen$ is sufficiently small).    In this case, consider the scheme 
\begin{align*}
\xymatrix{
\moduli : = \moduli_R(K) := _KSh(V)\times_{\OK} R\ar[d]\\
\Spec R
 }
\end{align*}  We denote by $\EE^{\pm}$ the locally free sheaf 
$\EE^{\pm} = \underline{\Isom}_{\OM}(\OM^n, \uo^{\pm})$
of $\OM$-modules on $\moduli$, and we denote by $\EE$ the locally free sheaf 
$\EE = \EE^-\oplus\EE^+$
 of $\OM$-modules on $\moduli$.  We denote by $\EEVr$, the locally free sheaf $\EE\times^{H^{\rho}}V$.  Then an automorphic form of weight $\rho$ is a global section of the sheaf $\EEVr$ on $\moduli_R(K)$.  Note that for any representation $(\rho, V)$ that can be decomposed as a direct sum $(\rho_1\oplus \rho_2, V_1\oplus V_2)$, the map
\begin{align}
\EEVr\rightarrow \EE_{V_1, \rho_1}\oplus\EE_{V_2, \rho_2}\label{edecom}\\
(\lambda, v)\mapsto ((\lambda, v_1), (\lambda, v_2))\nonumber
\end{align}
is an isomorphism.  
Therefore, to give an automorphic form of weight $\rho$ is equivalent to giving an automorphic form of weight $\rho_1$ and an automorphic form of weight $\rho_2$.

Since $GL_n$ is reductive, each finite dimensional representation $\rho$ can be written as a direct sum of irreducible representations 
$\rho = \rho_1\oplus\dots\oplus \rho_m$
for some $m$.  
Every irreducible representation of $GL_n$ can be realized as a subrepresentation of one of the representations constructed as follows.
For each set $\Lambda$ of ordered integers $\lambda_1\geq \ldots\geq \lambda_n$, there is a 
representation $(\rho_{\Lambda}, V_{\Lambda})$ of highest weight $\Lambda$.  The representation $(\rho_{\Lambda}, V_{\Lambda})$ can be realized explicitly by taking 
$V_{\Lambda} = \sym^{(\lambda_1-\lambda_2)}(R^n)\otimes \sym^{(\lambda_2-\lambda_3)}(\wedge^2 R^n)\otimes\cdots\otimes \sym^{\lambda_n}(\wedge^n R^n),$
and letting $\rho_{\Lambda}$ be the $GL_n$-action on $V_{\Lambda}$ induced by the standard representation of $GL_n(R)$ on $R^n$.  If $\lambda_n$ is negative, then by $\sym^{\lambda_n}(\wedge^n R^n)$, we mean the dual representation of $\sym^{-\lambda_n}(\wedge^n R^n)$, which is just the representation in which each $g\in \gln$ acts on each $v\in R$ by $v\mapsto \det g ^{\lambda_n}v$.  
(Note that the highest weight vector in $V_{\Lambda}$ is $(e_1)^{(\lambda_1-\lambda_2)}\otimes (e_1\wedge e_2)^{(\lambda_2-\lambda_3)}\otimes\cdots\otimes (e_1\wedge\cdots\wedge e_n)^{\lambda_n}$.)  Every irreducible representation of $\gln\times\gln$ is of the form $\rho^-\otimes \rho^+$ with $\rho^{\pm}$ irreducible representations of $\gln$.  

Let $W$ be a free rank $n$ $R$-module.  We write $W^{\rholam}$ to denote $\sym^{(\lambda_1-\lambda_2)}(W)\otimes \sym^{(\lambda_2-\lambda_3)}(\wedge^2 W)\otimes\cdots\otimes \sym^{\lambda_n}(\wedge^n W)$.  When $\rho$ is an arbitrary representation whose decomposition into irreducible representations is $\rho_{\Lambda_1}\oplus\cdots\oplus\rho_{\Lambda_m}$, we denote by $W^{\rho}$ the module $W^{\rho_{\Lambda_1}}\oplus\cdots\oplus W^{\rho_{\Lambda_m}}$.  Given another free-module $W_0$ of rank $n$ and a representation $\rho = \rho_{\Lambda_1}\otimes\rho_{\Lambda_2}$, we write $(W\otimes W_0)^{\rho_{\Lambda_1}\otimes\rho_{\Lambda_2}}$ to denote the module $W^{\rho_{\Lambda_1}}\otimes W_0^{\rho_{\Lambda_2}}$.  Given an arbitrary representation ${\rho}$ whose decomposition into irreducible representations is $\rho_{1}\oplus\cdots\oplus\rho_{m}$, we write $(W\otimes W_0)^{\rho}$ to denote the module $W^{\rho_{1}}\oplus\cdots\oplus W^{\rho_{m}}$.

Let $\Lambda$ be an ordered set of integers $\lambda_1\geq \ldots\geq \lambda_n$, corresponding to the representation of $\gln$ of highest weight $\Lambda$. Each $\lambda^\pm\in\EE^{\pm}$ induces an isomorphism
$\lambda^{\pm,\rholam}:\Vlam = (R^n)^{\rholam}\rightarrow (\uo^{\pm})^{\rholam}$
defined by 
$v_1\cdot v_2\cdot \cdots \cdot v_m\mapsto \lambda(v_1)\cdot\lambda(v_2)\cdot\cdots\cdot \lambda(v_m)$
where each $v_i$ is in $R^n$ and each $\cdot$ denotes the symmetric, tensor, or alterating product (according to $\Vlam$).  Observe that
$(\alpha\cdot\lambda)^{\pm\rholam} = \lambda^{\rholam}(\rholam({^t\alpha})v).$
So given a representation $\rho_{\Lambda_-}\otimes \rho_{\Lambda_+}$ of $\gln\times\gln$, each isomorphism $(\lambda^-, \lambda^+)\in\EE = \EE^-\oplus\EE^+$ induces an isomorphism
$\lambda^{\rho_{\Lambda_-}\otimes\rho_{\Lambda_+}}:V_{\rho_{\Lambda_-}}\otimes V_{\rho_{\Lambda_+}} = (R^n)^{\rho_{\Lambda_-}}\otimes (R^n)^{\rho_{\Lambda_+}}\isomto (\uo^-)^{\rho_{\Lambda_-}}\otimes(\uo^+)^{\rho_{\Lambda_+}}$
via
$v^-\otimes v^+\mapsto \lambda^{\rho_{\Lambda_-}}(v^-)\otimes \lambda^{\rho_{\Lambda_+}}(v^+).$
Observe that
$(\alpha\cdot\lambda)^{\rho_{\Lambda_-}\otimes\rho_{\Lambda_+}}(v) = \lambda^{\rho_{\Lambda_-}\otimes\rho_{\Lambda_+}}(\rho_{\Lambda_-}\otimes\rho_{\Lambda_+}({^t\alpha})v).$
Therefore, there is an isomorphism
\begin{align}\label{isototens}
\EE_{V_{\Lambda_-}\otimes V_{\Lambda_+}, \rho_{\Lambda_-}\otimes\rho_{\Lambda_+}}\isomto (\uo^-)^{\rho_{\Lambda_-}}\otimes(\uo^+)^{\rho_{\Lambda_+}},
\end{align}
defined by
$(\lambda, v)\mapsto \lambda^{\rho_{\Lambda_-}\otimes\rho_{\Lambda_+}}(v).$
%
Thus, at least in the case in which $\cpctopen$ is sufficiently small (i.e. when the moduli problem $_KA_V$ is representable), automorphic forms of weight $\tilde{\rho}$ (with $\tilde{\rho}$ a subrepresentation of $\rho = \rho_{\Lambda_-}\otimes\rho_{\Lambda_+}$) are sections of $(\uo^-\otimes\uo^+)^{\rho}.$  This last perspective (i.e. viewing automorphic forms as sections of $(\uo^-\otimes\uo^+)^{\rho}$ will be particularly useful to us when defining the differential operators.

When working over $\IC$, the following theorem, which relates algebraic automorphic forms to holomorphic ones, is useful.
\begin{thm}
If $n>1$, then $f\mapsto f^{\an}$ gives an isomorphism
$M^{alg}(\IC)(\rho, \Gamma) \rightarrow M^{\an}(\rho, \Gamma).$
\end{thm}
This fact involves the existence of toroidal compactifications and the analytic Koecher principle.  The reader may see \cite{CF} for additional details.

\subsection{$\ci$-automorphic forms}\label{cisec}

For any sheaf $\mathcal{F}$ on $\moduli$, we let $\fci$\index{$\fci$, fci} be the sheaf obtained by tensoring $\mathcal{F}$ with the $\ci$-structural sheaf of $\moduli^{\an}$.

Note that we have inclusions (analogous to (1.8.1) of \cite{kaCM})
\begin{align*}
M^{alg} = H^0(\moduli_{\IC}, \EEVr)\subset M^{\hol} = H^0(\moduli^{\an}, \EEVr\otimes \OM^{\hol})\subset M^{\ci} = H^0(\moduli, \EEVr\otimes\OM^{\ci})
\end{align*}



\section{The Gauss-Manin connection and the Kodaira-Spencer isomorphism}
\label{GMKS}
Throughout this section, let $\pi: X\rightarrow S$ and $\pi':Y\rightarrow S$ be smooth, proper morphisms of schemes, and suppose that $S$ is a smooth scheme over a scheme $T$.

\subsection{The Gauss-Manin connection}In this section, we review the construction of the Gauss-Manin connection (\cite{kaGM}, \cite{KO}, \cite{kan}, \cite{ke}).  

Consider the decreasing filtration of $(\Omega^\bullet_{X/T}, d)$ defined by 
\begin{eqnarray}\nonumber
F^i &= &Fil^i(\Omega^\bullet_{X/T})\nonumber\\
& = & \image(\pi^*\Omega^i_{S/T}\otimes_{\mathcal{O}_X}\Omega_{X/T}^{\bullet-i}\rightarrow\Omega_{X/T}^{\bullet}),\label{canmp}
\end{eqnarray}
where the morphism in (\ref{canmp}) is the canonical one.
The associated graded complex is $Gr(\Omega^\bullet)=\oplus_{p\geq 0}Gr^p$, with $Gr^p = F^p/F^{p+1}$.  As explained in \cite{KO} and \cite{kaGM}, there is a spectral sequence (which converges to $\mathbb{R}^{p+q}\pi_*(\Omega_{X/T}^\bullet) = H^{p+q}_{DR}(X/S)$) with $E_1$ term given by 
$E_1^{p,q} = \mathbb{R}^q\pi_*(Gr^p).$
The Gauss-Manin connection is the differential 
$d_1: E_1^{0, q}\rightarrow E_1^{1, q}.$
We denote the Gauss-Manin connection by $\nabla$\index{$\nabla$}.

Observe that 
$Gr^i \cong \Omega_{X/S}^{\bullet-i}\otimes_{\mathcal{O}_X}\pi^*\Omega_{S/T}^i$
for all $i$.  In particular, we see that 
$E_1^{0, q} \cong  H^q_{DR}(X/S)$ and 
$E_1^{1, q} \cong  H_{DR}^q(X/S)\otimes_{\mathcal{O}_S} \Omega_{S/T}^1.$
So the Gauss-Manin connection is the map 
\begin{align*}
\nabla = d_1: H^q_{DR}(X/S)\rightarrow H_{DR}^q(X/S)\otimes_{\mathcal{O}_S} \Omega_{S/T}^1.
\end{align*}
We will always take $q=1$ when applying the Gauss-Manin connection.

Observe that by construction of $\nabla$, if $f$ is an endomorphism of $X$ over $S$, then
\begin{align}\label{endnb}
\nabla(f\s(v)) = (f\s\otimes \id)(\nabla(v))
\end{align}
for each $v$ in $\dr(X/S)$.  As a consequence of \eqref{endnb}, we see that if $A$ is an abelian variety of type \eqref{peltuple} over an $\OK$-scheme $S$ and $v$ is in $\dr(A/S)^+$, then
$(\iota(a)\s\otimes \id)(\nabla(v)) = \nabla(\iota(a)\s v) = \nabla(a \cdot v) =a\nabla(v) = (a\otimes \id)\nabla(v),$
so
\begin{align}\label{gmonplus}
\nabla\left(\dr(A/S)^+\right)\subseteq \dr(A/S)^+\otimes \Omega
\end{align}
Similarly,
\begin{align}\label{gmonminus}
\nabla\left(\dr(A/S)^-\right)\subseteq \dr(A/S)^-\otimes \Omega.
\end{align}

\subsubsection{An important example}\label{gmex}
If we trace through the map given by the Gauss-Manin connection, we see that it involves lifting a relative form to an absolute form, differentiating the absolute form, and then projecting down to an element of $\dr(X/S)\otimes \Omega_{S/T}^1$.  This idea is best made clear in an example over $\IC$, which we will now provide.  This example not only explicitly illustrates how the Gauss-Manin connection acts in one of the main cases that interests us (the other example being over a $p$-adic base); it also will be useful later when we explicitly describe the action of our $C^\infty$-differential operators and when we relate our $C^\infty$-operators to the ones in \cite{shar}.  This example is strongly inspired by sections 4.0-4.2 of \cite{hasv}, and the construction here is directly analogous to and closely follows what Harris does for symplectic modular forms.  Our example is the $U(n,n)$ analogue of the example for symplectic groups in sections 4.0-4.2 of \cite{hasv}.

As noted in Section \ref{review}, we will work over the underlying $C^\infty$-manifold of our moduli space.  In the example, we will consider $\A$ over $\hn$ over $\IC$, as in Section \ref{reviewsym}.  The sheaf $H^1_{DR}(C^\infty)$ has a splitting 
\begin{align}\label{split}
\drci\cong\omci\oplus\splci,
\end{align} 
where $\omci$ is the space of holomorphic one-forms (which is the $C^\infty$ vector bundle corresponding to the sheaf of relative one-forms $\underline{\omega}=\pi_*\Omega_{\A/\hn}$) and $\splci$ is the space of anti-holomorphic one-forms.  Also, recall that the fiber $\drci_z$ of $\drci$ over a point $z\in \hn$ is the de Rham cohomology of that fiber (i.e. $A_z$, in the notation of Section \ref{reviewsym}) and that the splitting (\ref{split}) induces the Hodge decomposition of $\drci$ at each fiber.

Let $u_1, \ldots, u_{2n}$\index{$u_j$} denote standard coordinates in $\IC^{2n}$.  Then the global relative $1$-forms $du_1, \ldots, du_{2n}$\index{$du_j$} form a basis of the fiber of $\underline\omega$ over each point $z\in H$.

We now define some global relative $1$-forms that have constant periods across the fibers of $\A/\hn$.  We define them to be dual to the one-cycles (in homology) defined in terms of the basis for $L_z$ given in Equations (\ref{latticebasis1}-\ref{latticebasis2}).  

We consider the $\IR$-linear global relative one forms (for $i=1, \ldots , n$) which are given over
$z = \left(z_{ij}\right)\mbox{ in } \hn$
 by 
\begin{eqnarray}\index{$\alpha_i$, ai}\index{$\beta_i$, bi}\index{$\alpha_i'$, aip}\index{$\beta_i'$, bip}
\alpha_i\left(\sum_{j=1}^n a_je_j + \sum_{j=1}^n b_jz_j +\sum_{j=1}^n a'_je'_j +\sum_{j=1}^n b'_jz'_j\right) & = & a_i\label{ab1}\\
\beta_i\left(\sum_{j=1}^n a_je_j + \sum_{j=1}^n b_jz_j +\sum_{j=1}^n a'_je'_j +\sum_{j=1}^n b'_jz'_j\right) & = & b_i\\
\alpha'_i\left(\sum_{j=1}^n a_je_j + \sum_{j=1}^n b_jz_j +\sum_{j=1}^n a'_je'_j +\sum_{j=1}^n b'_jz'_j\right) & = & a'_i\\
\beta'_i\left(\sum_{j=1}^n a_je_j + \sum_{j=1}^n b_jz_j +\sum_{j=1}^n a'_je'_j +\sum_{j=1}^n b'_jz'_j\right) & = & b'_i,\label{ab2}
\end{eqnarray}
for each $a_i, b_i, a_i', b_i'$ in $\IR$.  (Here we are using the notation for the basis of $L_z$ given in Equations (\ref{latticebasis1}-\ref{latticebasis2}).)

Observe that these forms have constant periods along the fibers.  Therefore, 
\begin{eqnarray}\label{GM0}
\nabla\left(\alpha_i\right) = \nabla\left(\beta_i\right) = \nabla \left(\alpha'_i\right) = \nabla\left(\beta'_i\right) = 0 \hspace{0.5cm}\mbox{for } i=1,\ldots, n.
\end{eqnarray}
That is, the sections $\ai, \bi, \aip, \bip$ are {\it horizontal}.
Now we express $du_1, \ldots, du_{2n}$ and $\bar{du_1}, \ldots, \bar{du_{2n}}$ in terms of $\alpha_i, \beta_i, \alpha'_i, \beta'_i$.  Using the definitions of $e_i, z_i, e_i', z_i'$ (as in Equations (\ref{latticebasis1}-\ref{latticebasis2})) and $\alpha_i, \beta_i, \alpha_i', \beta_i'$ (as above), we see that for $i = 1, \ldots, n$,
\begin{eqnarray}
du_i & = & \alpha_i+\sum^n_{j=1}z_{ij}\beta_j+\bar{\alpha}\alpha'_i+\bar{\alpha}\sum_{j=1}^n z_{ij}\beta_j'\label{duis1}\\
\bar{du_i} & = & \alpha_i+\sum^n_{j=1}\bar{z}_{ij}\beta_j+\alpha\alpha'_i+\alpha\sum_{j=1}^n \bar{z}_{ij}\beta_j';
\end{eqnarray}
and for $i=n+1, \ldots, 2n$, we have
\begin{eqnarray}
du_i & = & \alpha_i+\sum^n_{j=1}z_{j, i-n}\beta_j+\alpha\alpha'_i+\alpha\sum_{j=1}^n z_{j,i-n}\beta_j'\\
\bar{du_i} & = & \alpha_i+\sum^n_{j=1}\bar{z}_{j, i-n}\beta_j+\bar{\alpha}\alpha'_i+\bar{\alpha}\sum_{j=1}^n \bar{z}_{j, i-n}\beta_j'\label{duis2}
\end{eqnarray}

Since $\nabla$ is a connection, 
$\nabla\left(a\cdot v\right) = a\nabla\left(v\right)+v\otimes da$,
for any $a\in \mathcal{O}_{\hn}$ and any $v\in H^1_{DR}$.  This, combined with Equations (\ref{duis1}) - (\ref{duis2}) shows that for $i=1, \ldots, n$,
\begin{eqnarray}
\nabla\left(du_i\right) & = & \sum_{j=1}^n\beta_j\otimes dz_{ij}+\bar{\alpha}\sum_{j=1}^n\beta'_j\otimes dz_{ij}\nonumber\\
 & = & \sum_{j=1}^n\left(\beta_j+\bar{\alpha}\beta'_j\right)\otimes dz_{ij}\label{duiexp1}\\
\nabla\left(du_{i+n}\right) & = & \sum_{j=1}^n\beta_j\otimes dz_{ji}+\alpha\sum_{j=1}^n\beta'_j\otimes dz_{ji}\nonumber\\
 & = & \sum_{j=1}^n\left(\beta_j+\alpha\beta'_j\right)\otimes dz_{ji}\label{duiexp2}
\end{eqnarray}

One can similarly compute $\nabla\left(\bar{du_i}\right)$.  For our purposes, however, we will only be interested in the application of $\nabla$ to the holomorphic differentials.  (This is because automorphic forms are associated with the {\it holomorphic} differentials.)

From Equations (\ref{duis1}) - (\ref{duis2}), we see that
\begin{eqnarray}
\left( \begin{array}{cc}
\beta_1+\alpha\beta'_1  \\
\vdots \\
\beta_n+\alpha\beta'_n \end{array} \right) & = & (z^t-\bar{z})^{-1} \left( \begin{array}{cc}
du_{n+1}-d\bar{u}_1  \\
\vdots \\
du_{2n}-d\bar{u}_n \end{array} \right)\label{vecdui1}\\
\nonumber\\
\nonumber\\
\nonumber\\
\left( \begin{array}{cc}
\beta_1+\bar{\alpha}\beta'_1 \\
\vdots \\
\beta_n+\bar{\alpha}\beta'_n \end{array} \right) & = & (z-z^*)^{-1} \left( \begin{array}{cc}
du_1-d\bar{u}_{n+1}  \\
\vdots \\
du_n-d\bar{u}_{2n} \end{array} \right)\label{vecdui2}
\end{eqnarray}

\begin{rmk}
We revisit \eqref{endnb} in the context of our example.  Consider an endomorphism $f:\A\rightarrow \A$ over $\hn$.  Then $f\left(L_z\right)\subset L_z$, and so $f\s$ maps the horizontal sections $\ai, \bi, \aip, \bip$ to horizontal sections.  (i.e. Forms with constant periods are mapped to forms with constant periods.)  Let $v$ be an element of $\dr$.  Then we can write $v = \sum f_i\gamma_i$ for some horizontal sections $\gamma_i$ and some sections $f_i\in\mathcal{O}_{\hn}$.  Since $f$ is a morphism over $\hn$, we see that $f\s\left(v\right) = \sum f_i f\s\gamma_i.$  So
\begin{align*}
\nabla\left(f\s\left(v\right)\right) = f\s\gamma_i\otimes df_i = \left(f\s\otimes \id\right)\nabla\left(\sum f_i\gamma_i\right) = \left(f\s\otimes \id\right)\nabla\left(v\right)
\end{align*}
\end{rmk}

This completes our example (for now) regarding the action of the Gauss-Manin connection.  This perspective will be useful again when we define the $C^\infty$-differential operators.

\subsubsection{Remark about some related connections}From the Gauss-Manin connection, we can construct connections on $(\dr)^{\otimes m}$, $\wedge \dr$, and $\sym \dr$ through the product rule.  For example, for any $v$ and $w$ in $H^1_{DR}$, we set
\begin{align}\label{prodrule}
\nabla\left(v\otimes w\right) = \sigma\left(\nabla\left(v\right)\otimes w\right) + v\otimes\nabla\left(w\right),
\end{align}
where $\sigma$\index{$\sigma$} is the canonical isomorphism switching the order of the last two components of the tensor product:
\begin{align}
\sigma: H^1_{DR}\otimes\Omega\otimes H^1_{DR}&\isomto H^1_{DR}\otimes H^1_{DR}\otimes \Omega\nonumber\\
v_1\otimes v_2\otimes v_3 &\mapsto v_1\otimes v_3\otimes v_2\nonumber
\end{align}
We similarly define $\nabla$ on higher tensor powers of $\dr$ inductively.

\subsection{The Kodaira-Spencer isomorphism}
We briefly review the Kodaira-Spencer isomorphism.  For more detailed treatments, see \cite{la, CF}.

From now on, $X$ is an abelian scheme.  
Let 
$\omega_X:=\pi_*\Omega_{X/S}$ and
$\omega_{X\dual} : = \pi_*\Omega_{X\dual/S},$
where $X\dual$ denotes the dual of $X$.  Hypercohomology gives a canonical exact sequence
\begin{eqnarray}\label{abex}
0\rightarrow \omega_X\hookrightarrow H^1_{DR}(X/S)\rightarrow R^1\pi_*(\mathcal{O}_X)
\end{eqnarray}
 
We then have canonical isomorphisms
$H^1_{DR}(X/S)/(\pi_*\Omega_{X/S})\isomto R^1\pi_*(\mathcal{O}_X)\isomto (\omega_{X\dual})\dual\nonumber.$
Define 
\begin{eqnarray}\label{tr1}
KS'\index{$KS'$}: \omega_X\rightarrow (\omega_{X\dual})\dual\otimes \Omega_{S/T}
\end{eqnarray}
to be the composition of canonical maps 
\begin{small}
\begin{eqnarray}\label{kstrace}
\omega\hookrightarrow H^1_{DR}(X/S)\stackrel{\nabla}{\rightarrow} H^1_{DR}(X/S)\otimes\Omega_{S/T}\twoheadrightarrow R^1\pi_*(\mathcal{O}_X)\otimes \Omega_{S/T}\isomto(\omega_{X\dual})\dual\otimes\Omega_{S/T}.
\end{eqnarray}
\end{small}
Tensoring each side with of \ref{tr1} with $\omega_{X\dual}$, we obtain a morphism
$KS\index{$KS$}: \omega_X\otimes \omega_{X\dual} \rightarrow \Omega_{S/T},$
making the Diagram (\ref{kstr}) commute.
\begin{align}\label{kstr}
\xymatrix{
 \omega_X\otimes\omega_{X\dual}\ar[rr]^{KS}\ar[dr]_{KS'\otimes \id} & & \Omega_{S/T}\\
 & \lp\omega_{X\dual}\rp\dual\otimes\Omega_{S/T}\otimes\omega_{X\dual} \ar[ur]_{f\otimes g\otimes h\mapsto f(h)\otimes g} & \\
 }
\end{align}

In Subsection \ref{cexa}, we explicitly describe the Kodaira-Spencer morphism in coordinates in an example over $\IC$.  In our example, we will be able to explicitly give the kernel of $KS$.  For more general cases, the kernel of $KS$ is provided in \cite{la}; we provide the relevant result from \cite{la} in Subsection \ref{lakern}.  However, while the abstract result from \cite{la} is useful, it is also important (for our particular situation) to keep in mind the example in coordinates that we work out in Subsection \ref{cexa}.

\subsubsection{Useful example over $\IC$}\label{cexa}
Like in Subsection \ref{gmex}, consider $\A$ over $\hn$ over $\IC$.  We first describe the polarization (for $z$ in $\hn$)
\begin{align*}
\lambda:=\lambda_z: A_z\rightarrow A_z\dual\nonumber
\end{align*}
following section 3.3 of \cite{shab} and section 4 of \cite{shar}.  Define $\langle , \rangle$\index{$\langle ,\rangle$} to be the non-degenerate symmetric $\IR$-bilinear pairing on $\IC^{2n}$ defined by
$\langle x, y\rangle  =  \sum_{i=1}^n x_i\bar{y}_i + \bar{x}_i y_i$
for all vectors $x = (x_i)$ and $y = (y_i)$ in $\IC^{2n}$.  Let $\zj\s, \ej\s, {\zjp}\s, {\ejp}\s$

\index{$\zj\s$, zj$\backslash$s}\index{$\ej\s$, ej$\backslash$s}\index{${\zjp}\s$, zjp$\backslash$s}\index{${\ejp}\s$, ejp$\backslash$s} denote the elements of $\IC^{2n}$ such that for any vector $v\in \IC^{2n}$
\begin{eqnarray}
\langle v, \zj\s\rangle & = & \bj(v)\nonumber\\
\langle v, {\zjp}\s\rangle & = & \bjp(v)\nonumber\\
\langle v, \ej\s\rangle & = & \aj(v)\nonumber\\
\langle v, {\ejp}\s\rangle & = & \ajp(v)\nonumber,
\end{eqnarray}
where $\bj, \aj, \bjp, \ajp$ are defined as in Equations (\ref{ab1}) through (\ref{ab2}).

Let $(, )$\index{$(,)$} be the pairing on $\IC^{2n}$ defined by
$(x, y) = \sum_{i=1}^n\bar{x}_i y_i$
for all vectors $x = (x_i)$ and $y = (y_i)$ in $\IC^{2n}$.  
Then we can write $\alpha_i, \beta_i, \alpha_i', \beta_i'$ as the sum of its $\IC$-anti-linear and $\IC$-linear pieces as
\begin{eqnarray}
\alpha_i(\bullet) & = & (\bullet, e_j^*)+(e_j^*, \bullet)\nonumber\\
\beta_i(\bullet) & = & (\bullet, z_j^*)+(z_j^*, \bullet)\nonumber\\
\alpha_i'(\bullet) & = & (\bullet, e_j'^*)+(e_j'^*, \bullet)\nonumber\\
\beta_i'(\bullet) & = & (\bullet, z_j'^*)+(z_j'^*, \bullet)\nonumber
\end{eqnarray}

Consider the Riemann form $E_z$\index{$E_z$} on $\IC^{2n}$ defined on the lattice $L_z$ by
$E_z(p_z(x), p_z(y)) :=  \tr_{\IC/\IR}(x\eta_n y^*)$
for all $x, y \in \OK^{2n}$, where $p_z(\bullet)$ is defined as in section \ref{reviewsym}.

Viewing $A_z^{\check{ }}$ as 
\begin{align}\label{du2}
\IC^{2n}/L_z^*, 
\end{align}
where $L_z^*$ is the lattice spanned by $z_i^*, e_i^*, z_i'^*, e_i'^*$, we have that $\lambda$ is the $\IC$-linear map defined by
$\langle \lambda(u), v\rangle = E_z(u, v)$
for all $u, v$ in $\IC^{2n}$.  In particular, we see that
$\lambda(e_i)  =  2z_i^*+ \tr_{\IC/\IR}(\bar{\alpha})z_i'^*$ and
$\lambda(e_i') =   \tr_{\IC/\IR}(\alpha)z_i^*+2\alpha\bar{\alpha}z_i'*.$
So for $u_j\in\IC$, we have that for $j=1, \ldots, n$,
\begin{eqnarray}
\lambda((0, \ldots, u_j, \ldots, 0)) & = & u_j\lambda\left(\frac{1}{\alpha-\bar{\alpha}}(\alpha e_j - e_j')\right)\nonumber\\
& = & \frac{u_j}{\alpha-\bar{\alpha}} \left(\left(2\alpha-\tr_{\IC/\IR}(\alpha)\right)z_j^*+\alpha\lp\tr_{\IC/\IR}(\bar{\alpha})-2\bar{\alpha}\rp z_j'^*\right)\nonumber\\
& = & u_j\lp z_j^*+\alpha z_j'^*\rp\nonumber
\end{eqnarray}
and similarly,
$\lambda((0, \ldots, u_{j+n}, \ldots, 0))  =  u_j\lambda\left(\frac{1}{\bar{\alpha}-\alpha}(\bar{\alpha} e_j - e_j')\right)
=  u_j(z_j^*+\bar{\alpha} z_j'^*).$

Now let $w_1, \ldots w_n, w_{n+1}, \ldots, w_{2n}$\index{$w_j$} be coordinates on $\IC^{2n}$ in terms of the vectors 
$z_1^*+\alpha z_1'^*, \ldots, z_n^*+\alpha z_n'^*,  z_{1}^*+\bar{\alpha} z_{1}'^*,\ldots, z_{n}^*+\bar{\alpha} z_{n}'^*,$
respectively.  Then for $1\leq j\leq n$,
$\lambda^*(dw_j)  =  du_j$.
\index{$dw_j$}
We are now in a position to look at the action of the Kodaira-Spencer morphism $KS$ on the basis $du_i\otimes dw_j$, $i, j =1, \ldots {2n}$ for $\omega_{A_z}\otimes\omega_{A_z\dual}$.  We will do this by tracing $du_i\otimes dw_j$ step-by-step through the composition of maps (\ref{kstrace}) and (\ref{kstr}).  For $i = 1, \ldots, n$ and $j = n+1, \ldots, 2n$, 
\begin{align}
du_i\otimes dw_j\stackrel{\mbox{(incl.)}}{\longmapsto} du_i\otimes dw_j & \stackrel{\nabla\otimes \id}{\longmapsto}
\lp\sum_{k=1}^n\lp\beta_k+\bar{\alpha}\beta'_k\rp\otimes dz_{ik}\rp\otimes dw_j\label{tr11}\\
& \stackrel{\mod\omega_{A_z}}{\longmapsto} \sum_{k=1}^n\lp\lp\bullet, z_k^*+\bar{\alpha}z_k'^* \rp \otimes dz_{ik}\rp\otimes dw_j\label{antiks}\\ 
& \mapsto \sum_{k=1}^n\lp\lp dw_k\rp\dual\otimes dz_{ik}\rp\otimes dw_j\mapsto dz_{i,j-n}.\label{antiks2}
\end{align}
(Note that in lines (\ref{du2}), (\ref{antiks}), and (\ref{antiks2}), we are implicitly associating the following via their canonical identifications: $\IC^{2n}$, the tangent space of $A_z\dual$, $\omega_{A_z\dual}\dual$, $\Hom_{\IC\mbox{-anti-lin}}(\omega_{A_z}, \IC)$, and $\bar{\omega}_{A_z}$.)

So from lines (\ref{tr11}) through (\ref{antiks2}), we see that
\begin{align}\label{ksdz}
KS(du_i\otimes dw_j) &= dz_{i,j-n}&\mbox{for } 1\leq i\leq n \mbox{ and } n+1\leq j\leq 2n.
\end{align}  Similarly, by tracing $du_i\otimes dw_j$ through the composition of maps  (\ref{kstrace}) and (\ref{kstr}), one finds that
\begin{align}
KS(du_i\otimes dw_j)& =  \left\{ 
\begin{array}{l l}
  dz_{j,i-n} & \quad \mbox{if $n+1\leq i\leq 2n$ and $1\leq j\leq n$}\\
  0 & \quad \mbox{if $1\leq i, j\leq n$}\\
  0 & \quad \mbox{if $n+1\leq i, j\leq 2n$} \end{array} \right. 
\end{align}

We thus find that 
$\IKS\index{$\IKS$, IKS}:=\ker (KS)$
is spanned by
\begin{align}\label{Ispan}
\left\{ du_i\otimes dw_j - du_j\otimes dw_i \rvert 1\leq i, j\leq 2n  \right\}\cup\left\{du_i\otimes dw_j\rvert 1\leq i, j\leq n\mbox{ or }n+1\leq i, j\leq 2n\right\}.
\end{align} 
This is a special case of the more general result given in Lemma \ref{lankern}.  In Section \ref{cinf}, the above description of $\IKS$ will be important in our consideration of the action of $\nabla$ on $\IKS$ (defined through the product rule).

\begin{lem}
KS induces an isomorphism (the ``Kodaira-Spencer isomorphism")
\begin{align}\nonumber
KS: \omega_{\univav}\otimes \omega_{\univav\dual}/\IKS\isomto \Omega_{\mathcal{H}/\IC}.
\end{align}
Associating $\Omega_{\mathcal{H}/\IC}$ with the complex vector space $\IC^n_n$ via
\begin{align}\nonumber
dz_{ij}\leftrightarrow e_{ij}:=\mbox{ the $n\times n$ matrix with $1$ in the $ij$-th position and zeroes everywhere else},
\end{align}
we have that for all $\gamma\in K^c\subset GL_{2n}(\IC)$, $h\in\omega_A\otimes\omega_{A\dual}$, and $g\in\Omega_{\mathcal{H}/\IC}$,
\begin{align}
KS\lp(\rho_{St}\otimes\rho_{St})(\gamma\otimes\gamma)(h)\rp & = \tau(\gamma)g.\label{repsagr}
\end{align}
\end{lem}
\begin{proof} The isomorphism follows directly from the above example.  Equation (\ref{repsagr}) also follows from the above example, combined with the definitions of $\rho_{St}$ and $\tau$ from Section \ref{reviewsym}.  
\end{proof}

For convenience, we often associate $\omega_A$ with $\omega_{A\dual}$ via the isomorphism
\begin{align}
\lambda^*&: \omega_{A\dual}\rightarrow\omega_A\nonumber\\
&dw_i\mapsto du_i\nonumber
\end{align}
coming from the polarization $\lambda$.  

\begin{lem}
$\nabla(\IKS)\mod\lp \Split(C^\infty)\otimes\Omega(C^\infty)\rp$
is contained in
$\IKS\otimes \Omega(C^\infty).$
\end{lem}
\begin{proof}
We prove this lemma by showing that 
\begin{align}
\nabla(v)\subset I\otimes \Omega\label{basshow}
\end{align}
for all the elements $v$ in the basis for $\IKS$ given in (\ref{Ispan}).  Using the product rule given in (\ref{prodrule}), we have that for all $x$ and $y$ in $\omci$,
\begin{align}\label{prule}
\nabla(x\otimes y-y\otimes x) = \sigma(\nabla(x)\otimes y) + x\otimes \nabla(y) - \sigma(\nabla(y)\otimes x) - y\otimes\nabla(x).
\end{align}
By (\ref{Ispan}), we know that for all $x$ and $y$ in $\omci$
$x\otimes y - y\otimes x$
lies in $\IKS$.
Therefore, for all $z$ and $w$ in $\omci$,
$z\otimes \nabla(w) - \sigma(\nabla(w)\otimes z)\mod\lp \splci\otimes\Omega(C^\infty)\rp$
lies in $\IKS\otimes\Omega$.  Consequently, from Equation (\ref{prule}), we see that
$\nabla(x\otimes y - y\otimes x)$
lies in $\IKS\otimes \Omega\mod\lp \splci\otimes\Omega(C^\infty)\rp$ for all $x$ and $y$ in $\omci$.  In particular, for all $i$ and $j$,
$\nabla(du_i\otimes du_j - du_j\otimes du_i)\mod\lp \splci\otimes\Omega(C^\infty)\rp$
lies in $\IKS\otimes \Omega.$  Now, we check that
$\nabla(du_i\otimes du_j)\mod\lp \splci\otimes\Omega(C^\infty)\rp$
lies in $\IKS\otimes\Omega(C^\infty)$ whenever $1\leq i, j\leq n$ or $n+1\leq i,j\leq 2n$.  From Equations (\ref{duiexp1}) through (\ref{vecdui2}), we see that for $1\leq i, j\leq n$,
$\nabla(du_i)\mod \Split(C^\infty)\otimes\Omega(C^\infty)$
is contained in the submodule of $\omci\otimes\Omega(C^\infty)$ generated by the set of elements
$\left\{du_k\otimes w\rvert 1\leq k\leq n\mbox{ and } w\in\Omega(C^\infty)\right\};$
and similarly,
$\nabla(du_{i+n})\mod \Split(C^\infty)\otimes\Omega(C^\infty)$
is contained in the submodule of $\omci\otimes\Omega(C^\infty)$ generated by the set of elements
$\left\{du_{k}\otimes w\rvert n+1\leq k\leq 2n\mbox{ and } w\in\Omega(C^\infty)\right\}.$
Therefore, we see that for $1\leq i, j\leq n$,
$\nabla(du_i\otimes du_j)\mod \Split(C^\infty)\otimes\Omega(C^\infty)$
is contained in the submodule of $\omci\otimes\omci\otimes \Omega(C^\infty)$ generated by
$\left\{du_k\otimes du_{l}\otimes w\rvert 1\leq k, l\leq n\mbox{ and } w\in\Omega(C^\infty)\right\},$
which is a submodule of $\IKS\otimes \Omega(C^\infty)$ (by \ref{Ispan}).  Similarly, we see that for  $1\leq i, j\leq n$,
$\nabla(du_{i+n}\otimes du_{j+n})\mod \Split(C^\infty)\otimes\Omega(C^\infty)$
is contained in the submodule of $\omci\otimes\omci\otimes \Omega(C^\infty)$ generated by
\begin{align*}
\left\{du_{k+n}\otimes du_{l+n}\otimes w\rvert n+1\leq k+n, l+n\leq 2n\mbox{ and } w\in\Omega(C^\infty)\right\},
\end{align*}
which is a also submodule of $\IKS\otimes \Omega(C^\infty)$ (by \ref{Ispan}).  Since we have now shown that (\ref{basshow}) holds for all $v$ in the basis for $\IKS$ given in (\ref{Ispan}), the proof of the lemma is complete.\end{proof}

\subsection{The kernel of the Kodaira-Spencer isomorphism}\label{lakern}

We try as much as possible in this section to be consistent with the notation of \cite{la}.  Let $S_0$ be the base scheme over which $\KSh$ is defined, and let $\peltup$  be the tuple associated to a morphism $S\rightarrow \KSh$.

\begin{lem}[\cite{la}, part of Proposition 2.3.4.2]\label{lankern} The kernel $\IKS$ of $KS: \omega_{A/S}\otimes\omega_{A\dual/S}\rightarrow\Omega_{S/S_0}$
contains the submodule $J_{KS}$ of $\omega_{A/S}\otimes\omega_{A\dual/S}$ generated by the set of elements
\begin{align}
\left\{\lambda\s(y)\otimes x-\lambda\s(x)\otimes y\rvert x, y\in \omega_{A\dual/S}\right\}\cup\nonumber\\
\left\{ (i(b)\s x)\otimes y - x\otimes((i(b\check))\s y)  \rvert x\in\omega{A/S}, y\in \omega_{A\dual/S}, b\in \OK \right\}\label{ksgen},
\end{align}
Furthermore, if $S\rightarrow\moduli$ is \'etale, then the map
$KS: \omega_{A/S}\otimes\omega_{A\dual/S}/J_{KS}\rightarrow\Omega_{S/S_0}$
is an isomorphism.
\end{lem}

Since $\lambda$ is a prime-to-$p$ polarization, the morphism
$\lambda\s: \omega_{A\dual/S}\rightarrow\omega_{A/S}$
is an isomorphism.

Therefore, we obtain the following Corollary of Lemma \ref{lankern}.
\begin{cor}\label{lankerncor}
Suppose $S\rightarrow\moduli$ is \'etale.  Then, the Kodaira-Spencer morphism $KS$ induces an isomorphism, which by abuse of notation we also denote $KS$:\begin{align}\label{KSwker}
KS: \sym^2(\omega_{A/S})/J_{KS}\isomto \Omega_{S/S_0}.
\end{align}
(Here, we associate $J_{KS}$ with its image in $\sym^2(\omega_{A/S})$.)
\end{cor}

\begin{rmk}
We shall mainly be applying Corollary \ref{lankerncor} in the case where $S = \moduli$ and the map $\moduli\rightarrow\moduli$ is the identity (so the abelian scheme associated to $S\rightarrow\moduli$ is the universal abelian scheme $\A$).  Since the identity map is \'etale, we can indeed apply Corollary \ref{lankerncor} in this situation.  

In the case where $S$ is an arbitrary scheme over $S_0$ and $S\rightarrow\moduli$ is not \'etale, $\IKS$ can be strictly larger than $\JKS$.  For example, consider the case where $\moduli =\hn$ over $S_0=\IC$, $S=\IC$, and $A$ is the abelian variety corresponding to a morphism $S:=Spec(\IC)\rightarrow \hn/\IC$.  In this case, $\Omega_{S/S_0} = 0$, but $\sym^2(\omega_{A/S})/\JKS\neq 0$.
\end{rmk}

As a direct consequence of Corollary \ref{lankerncor}, we have isomorphisms (which are crucial in our construction of the differential operators)
\begin{align}
\Omega_{\moduli/S_0}&\isomto Sym^2(\uo)/J_{KS}\label{ks1}\\
&\isomto\uo^+\otimes\uo^-\label{ks2}
\end{align}
Depending on the situation, we will work sometimes with isomorphism \eqref{ks1} and sometimes with \eqref{ks2}.

\section{Algebraic and analytic $q$-expansions}
\label{qexpnsection}
In this section, we briefly discuss algebraic $q$-expansions, which will be important in later proofs.

\subsection{Fourier expansions}\label{fouriersection}
This section closely follows Section 5 of \cite{shar}.  

For $c\in \IC$ and $X\in \IC^n_n$, let
$\e(c) = \exp(2\pi i c)$,
$\e^n(X) = \e(\tr(X))$, and
$S =S^n = \left\{\sigma\in K^n_n \vert \sigma\s = \sigma\right\}.$  Let $\Gamma\subset GU(\eta_n)$ be a congruence subgroup.  Then there exists a $\ZZ$-lattice $M$ in $S$ such that $\left( \begin{array}{cc}
1 & \sigma \\
0 & 1 \end{array} \right)$ is in $\Gamma$ for each $\sigma$ in $M$.  
Let 
$L' =\left\{h\in S\vert tr(hM)\subset \ZZ \right\}.$
Let $f$ be a holomorphic automorphic form with respect to $\Gamma$ that takes values in a vector-space $X$.  Then because $f(z+\sigma) = f(z)$ for each $\sigma\in M$, $f$ has a {\it Fourier expansion}, i.e. there exist elements $c(h)\in X$ such that
$f(z) = \sum_{h\in L'} c(h) \e^n(hz).$
We write the Fourier expansion of $f$ as
$f(z) = \sum_{h\in S} c(h)\e^n(hz).$
If $n>1$, then $c(h)\neq 0$ only if $h$ is nonnegative definite.

\subsection{The algebraic theory of $q$-expansions and the Mumford object}
This section should be viewed as a user's guide to algebraic $q$-expansions in the PEL moduli problem.  The situation for the $Sp(n)$ moduli problem is similar.  

The algebraic theory of $q$-expansions relies upon the existence of what we shall call ``Mumford objects" or ``Mumford abelian varieties."  Mumford ojects are the higher dimensional generalization of Tate elliptic curves.  Like Tate curves, Mumford abelian varieties arise naturally from a certain semiabelian scheme over toroidal compactifications of the moduli scheme $\moduli = Sh(2V)$ over $\lp\OK\rp_{(p)}$.  For each cusp of $\moduli$, there is a corresponding Mumford object, which lies over the compactification of $Sh(2V)$ at that cusp.  The details of the construction of toroidal compactifications of PEL Shimura varieties is given in \cite{la}.
For Hilbert modular forms, the corresponding toroidal compactifications were constructed in \cite{ra}.  For symplectic modular forms, this is discussed is \cite{CF}.  For details on the discussion in \cite{CF}, see \cite{la}.  Tate curves are often used explicitly in computations and described in detail in coordinates over $\IC$.  The current literature on Mumford objects and algebraic $q$-expansions does not provide a similarly explicit description.  A $q$-expansion principle analogous to the $q$-expansion principle for Tate curves is provided in \cite{la}, however.

\subsubsection{The Mumford object}\label{mumobject}

Let
$V = \mathcal{W}\oplus \mathcal{W}',$
where
$\mathcal{W}, \mathcal{W}' = \cmfield^n.$
Note that
$J = \left( \begin{array}{cc}
 0 &  -1_n\\
1_n & 0 \end{array}\right)$
induces a pairing $\Psi$ on $V$.  The pairing $\Psi$ induces an isomorphism
$\mathcal{W}'\isomto\Hom(\mathcal{W}, \cmfield).$
Define a pairing 
$\Psi': V\times V\rightarrow \IQ$
by
$(x, y)\mapsto \Psi'(x,y) = \tr\left(\delta_{\cmfield}^{-1}\Psi\left(x,y\right)\right)$
for each $x, y\in V$.  Let $L$ be the lattice inside $V$ defined by
$L =\OK^n\oplus\OK^n.$
Then the restriction
$\Psi': L\times L\rightarrow \ZZ$
of $\Psi'$ is a perfect pairing.  Write $W$ and $W'$ to denote $\mathcal{W}\cap L$ and $\mathcal{W}'\cap L$, respectively.

Let $U$ be an open compact subgroup of $GU(n,n)(\adeles_f) = G(\adeles_f)$.  Then
$Sh_G(U)  = G(\IQ)^+\backslash\hn\times G(\adeles_f)/U\supseteq \Gamma\backslash \hn,$
with
$\Gamma = G^+(\IQ)\cap U.$  Let $P$ be the stabilizer of $W'$ in $GU(n,n)$, where $GU(n,n)$ acts on $W'$ (viewed as row vectors) on the right.  Then each matrix in $P$ is of the form
$\left( \begin{array}{cc}
A & B\\
0 & C\end{array}\right).$
Let $N$ be the unipotent radical of $P$, and let
$H = N\cap \Gamma.$
Then $H$ is an upper-triangular, unipotent subgroup of $GU(n,n)(K)$.  It is a simple computation to show that $N$ is contained in the group of matrices of the form
$\left( \begin{array}{cc}
1_n & B\\
0 & 1_n\end{array}\right)$
with $B$ a Hermitian matrix with entries in $K$.  Thus, we can choose the lattice $M$ used to construct Fourier expansions in Subsection \ref{fouriersection} so that
$H= \left( \begin{array}{cc}
1_n & M\\
0 & 1_n\end{array}\right).$
Note that $H$ maps $W$ to $W'$.

Let $H\dual$ be the dual lattice of $H$.  That is, each element of $H\dual$ may be viewed as a $\ZZ$-linear map $H\rightarrow \ZZ$ given by
$h\mapsto \tr\left(gh\right)\subseteq \ZZ.$
for some (non-unique) matrix $g$.  A simple computation shows that we may associate $H\dual$ with the lattice $L'$ in $K^n_n$, where $L'$ is defined in terms of $M$ as above.

The data $L = W\oplus W'$ above is called the (zero-dimensional) ``cusp at infinity."  The zero-dimensional cusps are in one-to-one correspondence with the elements of $P(\IQ)\backslash G(\adeles_f)/K$.  There is not a canonical way to associate a lattice to each $g$ in $P(\IQ)\backslash G(\adeles_f)/K$, but a systematic way is the following.  Write
\begin{align*}
G(\adeles_f) &= \coprod_i G(\IQ)g_i K\\
g &= \gamma g_i k.
\end{align*}
Define the lattice at the cusp corresponding to $g$ to be
$L_g  = (L\otimes\hat{\ZZ})g_i\cap V, W_g  = W\gamma\cap L,
W'_g  = W'\gamma\cap L.$
Then $H$ is defined accordingly, corresponding to our new choice of cusp.  Note that the symmetric space is $\coprod_{\Gamma_i}\hn$.  Each $g_i$ tells us which $\Gamma_i$ and $\gamma$ says which Borel.  In the following discussion of Mumford objects and $q$-expansions, one can consider any of the cusps $g$, even though we write our discuss in terms of the notation for the cusp at infinity; the reader wishing to work with a different cusp $g$ should simply replace $L$ with $L_g$, $W$ with $W_g$, etc.  As explained in \cite{la}, for a toroidal compactification of $\moduli$, the completion along the boundary stratum for the zero-dimensional cusp $[g]$ lies over $\mathrm{Spf}R$, where $R$ is the ring

\begin{small}
\begin{align*}
\qexpringb^{\Gamma_g} = \left\{\sum_{h\in H\dual_{\geq 0}}a_hq^h\mid h\geq 0\mbox{, } a_h\in\lp\OK\rp_{(p)}\mbox{, and } a(\gamma h \gamma\s) \mbox{ for all } (\gamma, \gamma\s)\in \Gamma_g\right\}.
\end{align*}
\end{small}
There is a semiabelian scheme over the toroidal compactification of $\moduli$.  By passing to $\qexpringp$, we obtain an abelian variety $\mathcal{G}_H$ over $\moduli/\Spec(\qexpringp)$, which gives the Mumford object at the cusp $H$, a semi-abelian scheme lying over 
\begin{align*}
\qexpringp = \left\{\sum_{h\in H\dual}a_h q^h\mid a_h\in\lp\OK\rp_{(p)} \mbox{ and } a_h = 0 \mbox{ if } h<<0\right\}
\end{align*}
and denoted $\tateav.$

We now briefly discuss the analytic situation over $\IC$ to provide some context and motivation for our upcoming (algebraic) description of Mumford objects.  Recall that for an elliptic curve $E$, the analytic construction 
$E = \IC/\ZZ+\tau\ZZ\xrightarrow{\exp} \IC^\times/q(\ZZ),$
where $q$ is the map
\begin{align*}
q: \ZZ\rightarrow \IC^\times\\
n\mapsto e^{2\pi i n},
\end{align*}
gives a map from an elliptic curve to the complex points of the Tate curve.  We now do the analogue of this in our situation.

Setting
$X = \frac{Z+Z\s}{2}$ and 
$Y = \frac{-i\lp Z-Z\s\rp}{2},$
we have $Z = X+iY$, with $Y>0$.  Now we express $\hn$ in terms of $H$: $\hn = H\otimes \IR + \lp H\otimes \IR \rp_{>0} i\subset H\otimes \IC.$
(For a set $S$ of Hermitian matrices, we write $S_{>0}$ to denote the set of positive definite matrices in $S$.)  Let $\tau\in\hn\subset H\otimes \IC$.  Then we may express the abelian variety $A_\tau$ as
$W'\otimes\IC/L_{\tau},$
where $L_{\tau}$ is the $\ZZ$-lattice generated by $W'\otimes 1$ and $\tau(W)\subset W'\otimes\IC$.  So we have a commutative diagram of Lie groups (where the map $\underline{q}$ is defined implicitly through the $\exp$ map, and $\underline{q}(H\dual) = \exp(\tau(W))$)
\begin{footnotesize}
\begin{align*}
\xymatrix{
W'\otimes\IC/L_{\tau}\ar[dd]\ar[rr]_--{\exp}&&W'\otimes\IC^\times/\underline{q}(H\dual) = W'\otimes\mathbb{G}_m(\IC)/\underline{q_\tau}(H\dual)\ar[dd]\\
\\
\hn\subset H\otimes \IC\ar[rr]_--{\exp}&&H\otimes\IC^\times = \Spec\lp\qexpringp{\otimes}_{\OK}\IC\rp^{\an}
 }
\end{align*}
\end{footnotesize}

The quotient $W'\otimes \gm(\IC)/\underline{q}(H\dual)$ above is the set of complex points of an abelian variety.

We now describe the Mumford object more explicitly, analogous to the description of $\tate_{\mathfrak{a}, \mathfrak{b}}(q)$ in \cite{kaCM}.
\begin{rmk}For the reader trying to understand \cite{kaCM} in the context of our description of the general situation, we provide the following dictionary between Katz's notation and the notation we will use for the general situation.
\begin{align*}
\mathfrak{a}&\leftrightarrow W'\\
\mathfrak{b}&\leftrightarrow W\\
\mathfrak{a}^{-1}\mathfrak{b}^{-1}&\leftrightarrow H\\
\mathfrak{ab}&\leftrightarrow H\dual.
\end{align*}
\end{rmk}

We define a $\ZZ$-linear morphism
\begin{align}\label{uqmorph}
\uq: W\rightarrow W'\otimes\gm
\end{align}
from $W$ to the torus $W'\otimes \gm$ lying over $\OK((q, H\dual))$ to be the composition of morphisms
\begin{align*}
W\xrightarrow{w\mapsto \eval_w} \Hom_{\ZZ}(H, W')\isomto H\dual\otimes W \rightarrow W'\otimes \gm.
\end{align*}
(By $\eval_w$, we mean the map $h\mapsto h(w)$ in $\Hom_{\ZZ}(H, W')$.)

The ``Mumford abelian variety at the cusp $H$" is the algebraification of the rigid analytic quotient
\begin{align}
\uq(W)\backslash\lp(W')\dual\otimes\gm\rp.
\end{align}
We denote the Mumford abelian variety by $\tateav$ or $\mbox{Mum}_H(q)$.  The construction of the Mumford abelian variety is discussed in \cite{mumcf} and in \cite{la}.

The Mumford abelian variety $\tateav$ has a canonical PEL structure.  The canonical endomorphism
\begin{align*}
\iota_{can}: \OK\rightarrow End_{\qexpringp}(\tateav)
\end{align*}
is defined by
\begin{align*}
\alpha: L\rightarrow L\\
l\mapsto \alpha\cdot l
\end{align*}
for each element $\alpha$ of $\OK$.  The dual abelian variety is
$\tateav\dual = \mbox{Mum}_{{W'}\dual\oplus W\dual}(q),$
i.e. the algebraification of the rigid analytic quotient
$\uq({W'}\dual)\backslash\lp W\dual\otimes\gm\rp.$
The canonical isomorphism
$W'\isomto W\dual$
induced by the pairing $\Psi$ induces a canonical polarization
$\lambda_{can}: \uq(W)\backslash\lp W'\otimes\gm\rp\rightarrow \uq({W'}\dual)\backslash\lp W\dual\otimes\gm\rp$
of $\tateav$.  The natural exact sequence
$0\rightarrow W'\otimes \prod_l\varprojlim_n\mu_{l^n}\rightarrow \prod_l T_l(A)\rightarrow W\otimes\hat{\mathbb{Z}}\rightarrow 0$
induces a canonical level $\cpctopen$ structure
$\alpha_{can}: V\otimes \adeles_f\isomto \prod_l T_l(A)\otimes\IQ$
modulo the action of $\cpctopen$.  At times, we shall write $\tateav$ to mean the tuple $(\tateav, \lambda_{can}, \iota_{can}, \alpha_{can})$.

Observe that there is a canonical isomorphism
$\omega_L: \lie(\tateav)\isomto \lie(W'\otimes \gm) = W'\otimes \qexpringp$
Dualizing the morphism $\omega_L$, we obtain a canonical isomorphism
\begin{align}\label{omegaw}
\omega_{can}: {W'}\dual\otimes\qexpringp\isomto \uo,
\end{align}
which gives a canonical element of $\EE$.
There is a similar isomorphism on $\tateav\dual$:
\begin{align}\label{lb}
W\dual\otimes \qexpringp=\lie(\gm\otimes W\dual) \isomto \lie(\tateav\dual)
\end{align}

We now revisit the Kodaira-Spencer morphism in the context of $\tateav$.  Recall that for an abelian variety $A$ the Kodaira-Spencer isomorphism identifies derivations of $\OM$ with pairings of $\uo_A$ and $\lie(A\dual)$, i.e. with elements of $\Lie(A)\otimes \Lie(A\dual)$.  We provide a concise reminder of the Kodaira-Spencer isomorphism for an abelian variety $A$, reviewing precisely the details that we will need in our discussion of the situation for $\tateav$.  Recall the exact sequence
$0\rightarrow\uo\rightarrow\dr\rightarrow\lie(A\dual)\rightarrow 0$
and the Gauss-Manin connection (from which the Kodaira-Spencer morphism is constructed)
$\nabla: \dr\rightarrow\dr\otimes\Omega.$
Each
$D\in T_{\moduli/\OK} =\underline{\mbox{Der}}\lp\OM, \OM\rp$
defines a morphism
$\nabla(D): \dr\rightarrow\dr,$
which induces a morphism
$KS(D): \uo\rightarrow \lie(A\dual)$
defined to be the composition of maps
$\uo\hookrightarrow\dr\xrightarrow{\nabla(D)}\dr\twoheadrightarrow \lie(A\dual).$
The map $D\mapsto KS(D)$
defines the Kodaira-Spencer morphism
$T_{\moduli/\OK}\rightarrow\Hom_{\OM}\lp\uo, \lie(A\dual)\rp\cong \lie(A)\otimes \lie(A\dual).$
Dualizing gives 
\begin{align*}
\Omega\leftarrow\lie(A)\otimes\lie(A\dual).
\end{align*}
On $\tateav$, the Kodaira-Spencer map is
\begin{align*}
KS: \mbox{Der}\lp\qexpringp, \qexpringp\rp&\rightarrow \lie\lp\mbox{Mum}_{W\oplus W'}(q)\rp\otimes\lie \lp\mbox{Mum}_{{W'}\dual\oplus W\dual}(q)\rp\\&\cong W'\otimes W\dual\otimes\qexpringp
\end{align*}
Given $\gamma\in H\otimes_{\ZZ}\OK$, define
$D(\gamma)\in \mbox{Der}\lp \qexpringp, \qexpringp\rp$
by
\begin{align*}
D(\gamma)\lp \sum_{\alpha\in H\dual} a_{\alpha}q^{\alpha} \rp = \sum_{\alpha\in H\dual} \tr(\alpha\gamma)a_{\alpha}q^{\alpha}.
\end{align*}
Note that there is a natural $\ZZ$-linear morphism
$\phi_H: H\rightarrow W\dual\otimes_{\ZZ}W'.$
Then
$KS(D(\gamma)) = \phi_H(\gamma)\otimes 1$
in $W'\otimes W\dual\otimes \qexpringp$.

For $w\in {W'}\dual$, let $\omega(w)$ denote the image in $\uo$ of $w\otimes 1$ under the morphism \eqref{omegaw}.  For each $v\in W\dual$, let $l(v)$ denote the image of $v\otimes 1$ under the morphism \eqref{lb}.  Then for each $\gamma\in H$, 
$\nabla\left(D(\gamma)\right)\left(\omega(w)\right)\equiv l(w\cdot \gamma)\mod\uo.$
The notation $\omega$ and $l$ has been chosen to be similar to similar maps in \cite{kaCM}.  We further denote by
$\omega^{\pm}(w)$
the projection of $\omega(w)$ onto $\uo^{\pm}$.  We then have that the Kodaira-Spencer isomorphism is the map
$\omega^+(e_i)\otimes \omega^-(e_j)\mapsto (D(e_{ij}))\dual,$
with $e_i, e_j, e_{ij}$ standard basis vectors in $W$ and $H\otimes_{\ZZ}\OK = \left(\OK\right)^{n}_n$, respectively.

Note that by extending scalars, we may consider $\tateav$ over $R\otimes \qexpringp$ for each $\OK$-algebra $R$, and we can extend the above maps to the case of $\tateav$ over $R\otimes\qexpringp$.

\subsubsection{Algebraic $q$-expansions}
We now discuss the key features for our situation.  For a general, in-depth discussion of Fourier-Jacobi expansions, the reader is referred to \cite{la}.
\begin{defi}Let $f$ be an automorphic form of weight $(\rho, V)$ over $R$.  We define the {\bf $q$-expansion of $f$ at the cusp $H$} to be
$f\left((\tateav, \alpha_{can}, \iota_{can}, \lambda_{can})\otimes R, \omega_{can}\otimes R \right)$.
\end{defi}
As noted at the beginning of the section, the $q$-expansions of $f$ lie inside $V\otimes_RR\otimes\qexpringb$.  Furthermore, when working over $\IC$, the analytically defined Fourier coefficients of the function $f^{an}$ on $\hn$ (where we associate the function $f$ of abelian varieties with the function $f^{an}:\hn\rightarrow\IC$ as in Sections \ref{reviewsym} through \ref{agafapproach}) at the cusp $L$ are the same as the algebraically defined $q$-expansion coefficents at the cusp $L$.  That is, if 
$f\left((\tateav, \alpha_{can}, \iota_{can}, \lambda_{can})\otimes \IC, \omega_{can}\otimes \IC \right) = \sum_{h\in H\dual}c(h) q^h,$
then the $h$-Fourier cofficient of $f^{an}$ for each $h\in H\dual$ is also $c(h)$.

In Proposition 7.1.2.15 of \cite{la}, Lan proves the Fourier-Jacobi Principle for automorphic forms on PEL Shimura varieties.  The $q$-expansion principle for modular forms is a special case of this.
\begin{thm}[$q$-expansion Principle, special case of Proposition 7.1.2.15 of \cite{la}]  Let $f$ be an automorphic form on $U(n,n)$ over an $\OK$-algebra $R$ of weight $\rho$ with values in an $R$-module $R\otimes_{\OK} X$ for some $\OK$-module $X$.  
\begin{enumerate}
\item{If $f(\tateav) = 0$ at one cusp on each connected component of $\moduli$, then $f = 0$.}
\item{Let $R_0\hookrightarrow R$ be an $\OK$-subalgebra of $R$.  If $f(\tateav)\in \qexpringp \otimes_{\OK} R \otimes X$ actually lies in $\qexpringp \otimes_{\OK} R_0 \otimes X\hookrightarrow \qexpringp \otimes_{\OK} R \otimes X$ for one cusp in each component of $\moduli$, then there is a unique automorphic form of weight $\rho$ on $U(n,n)$ defined over $R_0$ which becomes $f$ after the extension of scalars $R_0\hookrightarrow R$.}
\end{enumerate}
\end{thm}

\section{$p$-adic automorphic forms and the Igusa tower}
\label{padicandIg}
In this section, we review $p$-adic automorphic forms, following the viewpoints of \cite{hida}, \cite{hi05} and \cite{SHL}.  We also introduce the results for $U(n,n)$ analogous to ones in sections 1.9 through 1.12 of \cite{kaCM}.

Let $R$ be an $\OK$-algebra that is separated for the $p$-adic topology, i.e. satisfying
$R\hookrightarrow \varprojlim_m R/p^mR.$
Let $R_0$ be the $p$-adic completion of $R$.  Let $K = K_p\times K^p\subset G(\adeles_f)$ be a compact open subgroup with $K_p\subseteq G(\IQ_p)$ a hyperspecial maximal compact and $K^p\subseteq G(\adeles_f^p)$.

Let $v$ be the prime of $\cmfield$ over $p$ determined by the embedding $\incl_p$.  Let $W = (\OK)_v$, and let $W_m = W/p^mW$ for all nonnegative integers $m$.  Fix a toroidal compactification $\KSh$ of $\ptopmoduli$ over $W$.  

  The theory of $p$-adic automorphic forms is independent of the choice of toroidal compactification. Let $\tilde{H}$ be a lift of a power of the Hasse invariant $H$ to $\KSh$.  For $m$ a positive integer, let
$M_m = \KSh\times_WW_m,$ and let $S_m = M_m\left[\frac{1}{\tilde{H}}\right]$
be the nonvanishing locus of $\tilde{H}$, i.e. the ordinary locus.  Let 
$S_0 = M\left[\frac{1}{\tilde{H}}\right],$
and let $S_{\infty}$ be the formal completion $\varprojlim_m S_m$ of $S_0$ along $S_1$.  Note that $S_m$ is independent of the choice of $\tilde{H}$ as long as $p$ is nilpotent in $W_m$ for all positive $m$.

For $m\geq 0$, let $\pmr$ be the rank $g$ $p$-adic \'etale sheaf
$\pmr = \univav[p^r]^{\mbox{\'et}} = \A[p^r]/\A[p^r]^0$
over $S_m$.  We define $T_{m,r}$ to be the finite \'etale $S_m$-scheme
\begin{align*}
\xymatrix{
T_{m,r} = Isom_{S_m}(\pmr, (\OK/p^r\OK)^n)\ar[dd]_{\pi_{m,r}}\\
\\
S_m
 }
\end{align*}
representing the functor
$(\pi: X\rightarrow S_m)\mapsto\left\{ \mbox{isomorphisms  } \Psi: \pmr \isomto (\OK/p^r\OK)^n \mbox{  over } X \right\}.$
Details about this scheme are given in \cite{hida} and \cite{hi05}; of particular interest to us will be the fact that $\pi_{m,r}$ is an {\it affine} morphism.  Define
$T_{m,\infty} = \varprojlim_r T_{m,r}$
and
$T_{\infty, \infty} = \varprojlim_m T_{m,\infty}.$
Note that the formal scheme $T_{\infty, \infty}$ is an \'etale cover of $S_{\infty}$. Note that $T_{m,r}$ classifies quintuples
$\underline{A} = \left(A, \lambda, \iota, \alpha^p, X[p^r]^{\et}\isomto\left(\OK/p^r\OK\right)^n\right),$
where $(A, \lambda, \iota, \alpha^p)$ is the abelian variety with prime-to-$p$ structure corresponding to a point of $\ptopmoduli$.  Note that 
$\left(\OK/p^r\OK\right)^n\cong \left(\ZZ/p^r\ZZ\right)^g.$
Therefore, the prime-to-$p$ polarization $\lambda$ and the isomorphism
$X[p^r]^{\et}\isomto\left(\OK/p^r\OK\right)^n$
induce isomorphisms
$A[p^r]^0\isomto A\dual[p^r]^0\isomto\mu_{p^r}^g,$
which induces an inclusion
$\alpha_p: \mu_{p^r}^g\hookrightarrow A.$

Let $\uo\mr$ be the pullback of $\pmr$ to $T\mr$, i.e.
$\uo\mr = \lp\pi\mr\s\pmr\rp\otimes\mathcal{O}_{T\mr}.$
For $r\geq m$, there's a universal isomorphism (i.e. the universal object over $T\mr$)
$\Psi_{univ}: \pi\mr\s\pmr\isomto\lp\OK/p^r\OK\rp^n,$
which induces an isomorphism
$\omega_{can} = \Psi_{univ}\otimes\id: \uo\mr\isomto \lp\OK/p^r\OK\rp^n\otimes_{\ZZ_p} \mathcal{O}_{T\mr}.$
For $r\geq m$, the sheaf $\uo\mr$ {\it is} just the pullback of the sheaf of differentials $\uo$ to $T\mr$.  Indeed, the pullback of $\uo$ to $T\mr$ is canonically identified with $\lie(\A\dual)\otimes W_m$; the isomorphism with $\uo\mr$ now follows from the canonical isomorphisms
\begin{align*}
\lie(\A\dual)\otimes W_m\isomto \lie(\A\dual[p^r]^0)\otimes W_m\isomto\lie(\mu_{p^r})^n\otimes W_r\isomto (\OK/p^r\OK)^n\otimes \mathcal{O}_{T_{m,r}}.
\end{align*}

Note that $\omcan$ induces isomorphisms $\omcan^+$ and $\omcan^-$ of $\uo^-$ and $\uo^+$ with ${\OK}_v^n\otimes\mathcal{O}_{T\mr}\cong \mathcal{O}^n_{T\mr}$ and ${\OK}_{\bar{v}}^n\otimes\mathcal{O}_{T\mr}\cong \mathcal{O}^n_{T\mr}$, respectively.
So for the $p$-adic situation (in contrast to the situation over $\IC$), $\omcan$ provides a canonical element of the sheaf $\EE$ introduced earlier.

A $p$-adic automorphic form of weight $\rho_-\otimes\rho_+$ is defined to be a global section of $\left(\mathcal{O}_{T_{\infty, \infty}}^n\right)^{\rho_-}\otimes\left(\mathcal{O}_{T_{\infty, \infty}}^n\right)^{\rho_+}$, where the action of $\OK$ on each copy of $\left(\mathcal{O}_{T_{\infty, \infty}}^n\right)$ is induced by the action on $\uo_-^{\rho_-}\otimes\uo_+^{\rho_+}$ (identifying the two sheaves via $\omcan$).  When we want to eliminate ambiguity about the identification, we shall write $\left(\mathcal{O}_{T_{\infty, \infty}}^-\right)^n$ or $\left(\mathcal{O}_{T_{\infty, \infty}}^+\right)^n$ to mean $\omcan\left(\uo^\mp\right)$, respectively.  We write $V(\rho, R_0)$ to denote the space of $p$-adic automorphic forms of weight $\rho$ over $R_0$.

Above, we have used notation similar to that for the Igusa tower in \cite{hida}, \cite{hi05}, and \cite{SHL}.  To emphasize the analogy with \cite{kaCM}, we shall sometimes use the notation
$\ompa^\pm = \uo^\pm_{T_{\infty, \infty}}$, 
$\OM(\padic) = \mathcal{O}_{T_{\infty, \infty}}$, and 
$\moduli(\padic) = T_{\infty, \infty}.$
We shall denote the space of $p$-adic automorphic forms of weight $\rho = \rho^-\otimes \rho^+$ over a $p$-adically complete and separated $\OK$-algebra $R_0$ by $\padicauto$\index{$\padicauto$, padicauto}.

\subsection{$q$-expansions of $p$-adic automorphic forms}
In this section, we discuss $q$-expansions of $p$-adic automorphic forms and the $q$-expansion principle for $p$-adic automorphic forms.  This is also covered in \cite{SHL}, which cites \cite{hida}.

By extending scalars from $\qexpringp$ to the $p$-adic completion\\
$R_0 \hat{\otimes}_{\OK}\qexpringp$ of $R_0\otimes_{\OK}\qexpringp$, we obtain the Mumford object $(\tateav, \lambda_{can}, \alpha_{can}, \iota_{can})$ over the $p$-adic ring ${ R_0 \hat{\otimes}_{\OK}}\qexpringp$.  (Note that by construction, the isomorphism $\omcan$ from the previous section, viewed over $\tateav$ is the same as the isomorphism $\omega_{can}$ in \eqref{omegaw}.

So we obtain a $q$-expansion homomorphism
$FJ$ from the space of $p$-adic automorphic forms with values in an $R_0$-module $X$ to ${ R_0 \hat{\otimes}_{\OK}}\qexpringp\otimes_{R_0}X$:
\begin{align*}
f\mapsto f\left(\left(\tateav, \lambda_{can}, \alpha_{can} = \alpha_{can}^p\times\left(\alpha_{can}\right)_p, \iota_{can}\right)_{{ R_0 \hat{\otimes}_{\OK}}\qexpringp}, \omega_{can}\right)
\end{align*}

\begin{defi}When $R_0$ has no $p$-torsion, we define the space of $p$-adic automorphic forms defined over $R_0\otimes \IQ$ to be $V(\rho, R_0)\otimes_{{\OK}_v}K_v$.
\end{defi}

The $q$-expansion homomorphism extends to a $q$-expansion homomorphism
\begin{align*}
FJ': V(\rho, R_0)\otimes_{{\OK}_v}K_v \rightarrow { R_0 \hat{\otimes}_{\OK}}\qexpringp\otimes K_v\otimes_{R_0}X.
\end{align*}
We now state the $q$-expansion principle for $p$-adic automorphic forms.  This is the analogue of Theorem 1.9.9 of \cite{kaCM} and Corollary 1.9.17 of \cite{kaCM}.
\begin{thm}[Theorem 2.3.3 of \cite{SHL}, which cites \cite{hida}]  The $q$-expansion homomorphisms have the following properties.
\begin{enumerate}
\item{The $q$-expansion homomorphisms $FJ$ and $FJ'$ are injective, and when $R_0$ has no $p$-torsion, the cokernel of $FJ$ and $FJ'$ has no $p$-torsion.}
\item{${FJ'}^{-1}(R_0\otimes X) = V(\rho, R_0)$.}
\end{enumerate}
\end{thm}

\subsubsection{Map from automorphic forms over a $p$-adic ring $R_0$ to $p$-adic automorphic forms over $R_0$}

\begin{thm}[((2.2.7) in \cite{SHL}), analogue of Theorem 1.10.15 of \cite{kaCM}]
The homomorphism
$f\mapsto \tilde{f}$
from the space of weight $\rho$ level $\alpha$ automorphic forms to the space of weight $\rho$ level $\alpha^p$ $p$-adic automorphic forms defined by
$\tilde{f}(X, \lambda, \iota, \alpha) = f(X, \lambda, \iota, \alpha, \omega_{can})$
preserves $q$-expansions.
\end{thm}

\subsection{Frobenius and the unit root splitting}

In this section, we give a splitting
\begin{align}\label{rootsp}
\dr = \uo\oplus\uroot
\end{align}
over $\ig$ analogous to the $\ci$-splitting
$\dr = \uo\oplus\bar{\uo}.$
The splitting \eqref{rootsp} will be indispensable for the construction of the $p$-adic differential operators.  Most of the material in the section is essentially covered in section 1.11 of \cite{kaCM}, with trivial generalizations.  However, we have provided details not given in \cite{kaCM}.

Let $X$ be an abelian variety of PEL type over an $\OK$-algebra $R$ in which $p$ is nilpotent, and suppose that each of the geometric fibers of $X/R$ is ordinary.  So there is an inclusion
$\alpha_p:\mu_p^g\hookrightarrow X.$
Let $\hat{X}$ denote the formal group of $X$, and let $H_{can}$ be the canonical subgroup of $X$, i.e. the kernel of multiplication of by $p$ in $\hat{X}$.  Then $H_{can} = \alpha_p\lp\mu_p^g\rp$.    Let
$X' = X/H_{can},$
and let
$\pi: X\rightarrow X'$
be the projection map.  When $p = 0$ in $R$, $X' = X^{(p)}$, where $X^{(p)}$ denotes the scheme over $R$ obtained from $X$ by extension of scalars $F_{abs}: R\rightarrow R$, and $\pi$ is the relative Frobenius morphism:

\begin{align}\nonumber
\xymatrix{
X\ar[ddr]_{\pi}\ar[ddrrr]^{F_{abs}}\ar[ddddr]& & &\\
&& &\\
&X^{(p)}\ar[rr]\ar[dd]& & X\ar[dd]\\
&&\square &\\
&R\ar[rr]_{F_{abs}} & &R
 }
\end{align}

Given a morphism
$\alpha_p:\ptors\hookrightarrow X,$
we define
$\alpha_p': \ptors\hookrightarrow X'$
to be the morphism that makes the following diagram commute
\begin{align*}
\xymatrix{
0\ar[r]&\mu_p^g\ar[r]\ar[dd]_{\cong}&\ptors\ar[r]^p\ar@{^{(}->}[dd]^{\alpha_p}&\ptors\ar[r]\ar@{.>}[dd]^{\alpha_p'}& 0\\
&&&&\\
0\ar[r]&H_{can}\ar[r]& X\ar[r] &X'\ar[r]&0
}
\end{align*}
Note that a prime-to-$p$ level structure $\alpha^p$ induces a prime-to-$p$ level structure ${\alpha^p}'$ on $X'$.
We let
\begin{align*}
{\iota^p}':{\OK}_{(p)}\hookrightarrow End(A')\otimes\ZZ_{(p)}
\end{align*}
be the embedding induced by
$\iota^p: {\OK}_{(p)}\hookrightarrow End(A)\otimes\ZZ_{(p)}.$
As Katz explains in Lemma 1.11.6 of \cite{kaCM}, if $(X, \lambda)$ is in $T_{\infty, r}$, then there is a unique polarization $\lambda'$ that reduces $\mod p$ to the polarization $\lambda^{(p)}$ on $X^{(p)}$.  We shall now also use $\pi$ to denote the morphism
$(X, \lambda, \iota , \alpha^p, \alpha_p)\mapsto (X', \lambda', \iota', {\alpha^p}', \alpha_p')$
induced by $\pi$.

Note that, by construction, $\pi$ is compatible with change in base
$R/p^mR\rightarrow R/p^{m-1}R$
induced by projection.  So the morphisms $\pi$ induce a morphism
\begin{align}\label{piu}
\A/W\rightarrow \A'/W
\end{align}
(where $W  = \Spec R$)
over the $p$-adic ring $R$.  Since this is not explicitly mentioned in \cite{kaCM} and could be somewhat confusing to the reader, we note that the map $\pi$ in \eqref{piu} is defined over $R$, not over $\ig$, though $\A$ and $\A'$ lie over $\ig$.  So there is a unique isomorphism
$F:\ig\rightarrow\ig$
such that $\A'$ is the fiber product
\begin{align*}
\xymatrix{
\A'\ar[r]\ar[d]&\A\ar[d]\\
\ig\ar[r]^F&\ig
}
\end{align*}

We now describe the action of $F$ on $q$-expansions, which we will use in the proof of Lemma \ref{keylemma}.
\begin{lem}\label{qgoesqp}
For any $q$-expansion homomorphism
$f\mapsto f(q),$
the action of $F$ on $f$ satisfies
$(Ff)(q) = f(q^p),$
and so,
if
$f(q) = \qexpn,$
then
$(Ff)(q) = \qexpnp.$
\end{lem}

\begin{lem}\label{qqp}
The abelian variety $\tateav'$ and the morphism
$\pi: \tateav\rightarrow\tateav'$
{\it a priori} defined over $(\OK)_v(q, L)$ are in fact defined over $\OK(q, L)$.
\end{lem}
\begin{proof}
Since $\tateav'$ is obtained from $\tateav$ by extension of scalars $q\mapsto q^p$, which is defined over $\OK(q, L)$, $\tateav = Mum_L(q^p)$ is defined over $\OK(q, L)$. It follows from the definition of $\pi$ that $\pi$ is the map making the following diagram commute (where the vertical maps are projection onto the quotient):
\begin{align*}
\xymatrix{
W'\otimes\mathbb{G}_m\ar[rrr]^{\times p}\ar[d] && &W'\otimes\mathbb{G}_m\ar[d]\\
W'\otimes\mathbb{G}_m\ar[rrr]^--{\pi} && &W'\otimes\mathbb{G}_m/p\cdot q(H\dual) = W'\otimes\mathbb{G}_m/q(p\cdot H\dual)
}
\end{align*}
\end{proof}

\begin{rmk}\label{frobcx}
Since $\pi$ and $\tateav'$ are defined over $\OK$, we can extend scalars and consider the map $\pi$ over $\IC$.  In this case, observe that $\pi$ corresponds to the map on lattices
\begin{align*}
p_z(L)&\rightarrow p_{pz}(L),\\
l&\mapsto pl
\end{align*}
i.e. the map
\begin{align*}
\IC/p_z(L)&\rightarrow\IC/p_{pz}(L)\\
x&\mapsto px.
\end{align*}
The morphism $F$ corresponds to the morphism
\begin{align*}
\hn\rightarrow \hn\\
z\mapsto pz.
\end{align*}
\end{rmk}

Define $Fr$ to be the morphism
$Fr = \pi\s: F\s(\dr)\rightarrow\dr.$
Note that $Fr$ defines a ($F$-linear) morphism of $\dr$.
The higher dimensional analogue of Lemma (A2.1) in \cite{ka2} is the following.
\begin{lem}\label{lemka2}
$\pi\s\lp F\s\uo\rp = p\uo.$
\end{lem}

As in \cite{kaCM}, we have the following powerful proposition, which is essential in the construction of the $p$-adic differential operators.
\begin{prop}\label{FrHodgemx}
There is a unique splitting
\begin{align}\label{usplit}
\dr = \uo\oplus\uroot
\end{align}
over $\OM$ such that $\pi\s F\s$ is an isomorphism on $\uroot$ when tensored with $\IQ$ and such that
$\nabla(\uroot)\subseteq \uroot\otimes\Omega$
\end{prop}
The splitting \eqref{usplit} is called the {\it unit root splitting} and $\uroot$ the unit root submodule of $\dr$, as in \cite{kaCM} and \cite{kad}.  \cite{kaCM} notes simply that the version of Proposition \ref{FrHodgemx} in \cite{kaCM} (Theorem 1.11.27) is explained in \cite{kad}.  If $\ig$ were affine (which it is not), then all but the uniqueness statement would follow immediately from \cite{kaCM}.  Although $\ig$ is not affine, one may prove that the statement extends to $\ig$ by proving on an affine cover and glueing.  Details are provided in \cite{E09}.

\subsection{Unit root splitting for the Mumford abelian variety}
  
We now discuss the unit root splitting over $\tateav$, which plays a key role in the proof of Theorem \ref{pdivder}.

\begin{lem}\label{keylemma}\label{unitsplit}(Analogue of \cite{kaCM} Key Lemma (1.12.7))  Upon extension of scalars to $(\OK)_v\otimes_{\OK}\qexpringp$, the elements $\nabla(D(\gamma))(\omega(w))$ lie in $\uroot\subseteq \dr$ for each $\gamma\in H$ and $w\in (W')\dual$.
\end{lem}
\begin{proof}
By Lemma \ref{qqp}, $F$, $\tateav'$, and 
$\pi: \tateav\rightarrow\tateav'$
are defined over $\OK(q, L)$.  By Lemmas \ref{lemka2} and \ref{FrHodgemx}, $\pi\s$ has the form 
$\left( \begin{array}{cc}
pA & 0\\
0 & D'
\end{array} \right)$
with respect to fixed bases for $\uo$ and $\uroot$ for some $g\times g$ matrixes $A$ and $D$ with entries in $\OM$ and $D$ invertible.  So it suffices to show that
\begin{align}\label{pifs}
\pi\s\lp F\s\lp\nabla\lp D(\gamma)\rp\lp \omega(w)\rp\rp\rp = \nabla\lp D(\gamma)\rp\lp \omega(w)\rp
\end{align}
for all $\gamma\in H$ and $w\in (W')\dual$.  So it is sufficient to extend scalars to $\IC$ and check \eqref{pifs} over $\IC$.

In our proof, we shall work with $L = \OK^{2n}$, i.e. the cusp at $\infty$, and we note that the proof at other cusps is similar.  (We choose $L = \OK^{2n}$ because working in the context of our explicit examples over $\IC$ - which all used this lattice - provides the most insight.)

Over $\IC$,
$\dr = \Hom_{\ZZ}\lp p_z(L), \IC \rp,$
and $\omega(w)$ is a $\IC$-linear combination of the elements $du_i$.
So we are now reduced to proving an assertion about maps of lattices.  By Remark \ref{frobcx},
$\lp\pi\s l\rp \gamma = l(p\gamma)$
for each $l\in F\s\dr$ and $\gamma\in L_z$.  By \eqref{duiexp1} and \eqref{duiexp2}, $\nabla(D(\gamma))(du_i)$ lies in the subspace of $\dr$ generated by elements of the form
\begin{align}\label{firstform}
\beta_j+\alpha\beta_j'
\end{align}
or
\begin{align}\label{barform}
\beta_j+\bar{\alpha}\beta_j'
\end{align}
$1\leq j\leq n$ (in the notation of \eqref{duiexp1} and \eqref{duiexp2}).  Therefore, it suffices to show that $\pi\s \circ F\s$ fixes each element of the form \eqref{firstform} and each element of the form \eqref{barform}.

By Remark \ref{frobcx} and the definition of $\beta_j$, we see that
$F\s\lp \beta_j\rp: p_{pz}(L)\rightarrow\IC$
is the $\ZZ$-linear map defined by
\begin{align*}
p\cdot z_j&\mapsto 1\\
pz_i&\mapsto 0, i\neq j\\
e_i, e_i', pz_i'&\mapsto 0, \mbox{for all } i.
\end{align*}
Similarly,
$F\s\lp\beta_j\rp: p_{pz}\lp L\rp\rightarrow\IC$
is the $\ZZ$-linear map defined by
\begin{align*}
p\cdot z_j'&\mapsto 1\\
p\cdot z_i'&\mapsto 0, i\neq j\\
e_i, e_i, p\cdot z_i&\mapsto 0, \mbox{for all } i.
\end{align*}
So
$\pi\s F\s\lp\beta_j+c\beta_j'\rp\lp l\rp  = F\s\lp\beta_j +c\beta_j'\rp\lp pl\rp
 = \lp \beta_j + c\beta_j'\rp (l)$
for each $l\in p_z(L)$ and each $c\in \OK$, in particular for $c = \alpha, \bar{\alpha}$.
Therefore,
$\pi\s F\s\nabla\lp D(\gamma)\rp\lp du_i\rp = \nabla\lp D(\gamma)\rp\lp du_i\rp,$
for $1\leq i\leq 2n$.
\end{proof}

\section{$\ci$-differential operators from the perspective of Shimura}
\label{ShimurasDiffOps}
We now review $\ci$-differential operators (acting on automorphic forms on unitary groups) from the perspective of \cite{shar}.  Examples of the material discussed here can be found in \cite{E09}.  In later sections, we reformulate Shimura's differential operators algebreo-geometrically, and then we construct and discuss a $p$-adic analogue of the $\ci$-differential operators.  In Proposition \ref{rho0430}, we show that the $\ci$-differential operators we construct algebreo-geometrically in Section \ref{cinf} are the same as Shimura's differential operators that we discuss in this section.

As in \cite{shar}, for each $z\in\hn$, let $\Xi(z) = (i(\bar{z}-{^tz}), i(z\s-z)).$  Let $T = \IC^n_n$.  Let $\{\ev\}$ be an $\IR$-rational basis of $T$ over $\IC$.  For $u\in T$, let $u_{\nu}$ be defined by
$u=\sum_\nu u_\nu\ev.$
Similarly, for $z\in\hn$, define $z_\nu\in\IC$ by
$z = \sum_\nu z_\nu\ev.$

Let $(\rho, V) = (\rho_-\otimes\rho_+, V_-\otimes V_+)$ be a finite-dimensional representation of $GL_n(\IC)\times GL_n(\IC)$.  Let $e$ be a positive integer.  For finite-dimensional vector spaces $X$ and $Y$, define $S_e(Y,X)$ to be the vector space of degree $e$ homogeneous polynomial maps of $Y$ into $X$, i.e. the space of maps $h$ from $Y$ to $X$ such that $h(a\cdot y) = a^e h(y)$

for each $a\in\IC$ and $y\in Y$.  We let $S_e(Y)$ denote $S_e(Y, \IC)$.  From here on, we identify $S_e(Y,X)$ with $S_e(Y)\otimes X$ via
$h(u)\otimes x\mapsto h(u)x,$
for each function $h$ in $S_e(Y) = S_e(Y, \IC)$ and $x$ in $X$.  Let $Ml_e(Y, X)$ denote the vector space of all $\IC$-multilinear maps
\begin{align*}
\underbrace{Y\times\cdots\times Y}_{e \mbox{ times}} \rightarrow X.
\end{align*}
An element of $Ml_e(Y, X)$ is called {\it symmetric} if 
$g\left(y_{\pi(1)}, \ldots, y_{\pi(e)}\right) = g(y_1, \ldots, y_e)$
for each permutation $\pi$ of $\left\{1, \ldots, e\right\}$.
As explained in Lemma 12.4 of \cite{shar}, for each $h$ in $S_e(Y, X)$, there is a unique symmetric element $h_*$ of $Ml_p(Y, X)$ such that
$h(y) = h_*(y, \ldots, y)$
for all $y$ in $Y$.  We shall associate $S_e(Y, X)$ with a subspace of $Ml_e(Y, X)$ in this way.
We define a representation $(\tau^e, Ml_e(T, \IC))$ of $GL_n(\IC)\times GL_n(\IC)$ as follows: Given $(a, b)\in GL_n(\IC)\times GL_n(\IC)$ and $h\in Ml_e(T, \IC)$,
$[\tau^e(a, b) h](u_1, \ldots, u_e) = h(^ta u_1 b, \ldots, ^ta u_e b).$
Thus, we obtain a representation $\rho^e\otimes\tau$ of $\glnc\times\glnc$ on $Ml_e(T, \IC)\otimes X = Ml_e(T, X)$ via
\begin{align*}
[\rho\otimes\tau^e(g)](h(u)\otimes x) = \tau^e(g)h\otimes \rho(g)x
\end{align*}
for each $g\in\glnc\times\glnc$, $h\in Ml_e(T, \IC)$, and $x\in X$.  We also write $\rho\otimes\tau^e$ to denote the restriction of this representation to $S_e(T, X)$.

For $f\in \ci(\hn, V)$, define operators
\begin{align}\label{defnD}
C, D: \ci(\hn, V)\rightarrow\ci(\hn, S_1(T, V))
\end{align}
by
\begin{eqnarray}\nonumber
(Df)(u) &=& \sum_{\nu} u_\nu\frac{\partial{f}}{\partial{z_\nu}},\nonumber\\
(Cf)(u) &=& (\tau^1(\Xi)Df)(u) = (Df)({^t\xi} u\eta),\nonumber
\end{eqnarray}
respectively. For $e>1$, we write $D^ef$ and $C^ef$ to denote $D(D^{e-1}f)$ and $C(C^{e-1}f)$, respectively.  The functions $D^ef$ and $C^ef$ have symmetric elements of $Ml_e(T, V)$ as their values, which allows us -- as explained in Section 12.1 of \cite{shar} -- to view them as elements of $\ci\left(\hn, S_e(T, V)\right).$  Therefore, the operators $C^e$ and $D^e$ can be viewed as maps
$\ci(\hn, V)\rightarrow\ci(\hn, S_e(T, V)).$

In general, the operators $C^e$ and $D^e$ do not map automorphic forms to automorphic forms.  They are, however useful for constructing a map from the space of automorphic forms of weight $\rho$ to the space of automorphic forms of weight $\rho\otimes \tau$.  Define
\begin{align}\nonumber
(D_\rho f)(u)& = \rho(\Xi)^{-1} D[\rho(\Xi)f](u)\nonumber\\
& = \lp\rho\otimes\tau\rp(\Xi)^{-1} C[\rho(\Xi)f]\nonumber
\end{align}
and, more generally,
\begin{align}\nonumber
(D_{\rho}^e f) = (\rho\otimes\tau^e)(\Xi)^{-1}C^e[\rho(\Xi)f].
\end{align}
The operator $D_{\rho}^e$ satisfies the following properties (\cite{shar}):
\begin{align}\nonumber
D_{\rho}^{e+1}&= D_{\rho\otimes\tau}D^e_{\rho} = D^e_{\rho\otimes\tau} D_{\rho}\nonumber\\
D_{\rho}^e(f\Vert_{\rho}\alpha)&=(D_{\rho}^e f)\Vert_{\rho\otimes\tau^e}\alpha,\label{autaut}
\end{align}
for $\alpha$ in $G$.  So $D_{\rho}^e$ maps automorphic forms of weight $\rho$ to automorphic forms of weight $\rho\otimes\tau^e$.

Let $Z$ be a $\glnc\times\glnc$-stable quotient of $S_e(T)$, and let $\phi_Z$ denote the projection of $S_e(T)\otimes X$ onto $Z\otimes X$.  Then the operator
$D_{\rho}^Z = \phi_Z D_{\rho}^e$
is a map from the space of automorphic forms of weight $\rho$ to the space of automorphic forms of weight $\rho\otimes\tau_Z$, where $\tau_Z$ denotes the restriction of $\tau$ to $Z$.

\section{Some purely algebraic differential operators}
\label{Kapuralg}
We introduce some key ingredients for the construction of both the $\ci$- and the $p$-adic-differential operators.
\subsection{Some algebraic differential operators}
Let $S$ be an $\OK$-scheme.  We assume throughout this section that
$S\rightarrow \moduli$  is an {\it \'etale} morphism.

Recall that by \eqref{gmonplus} and \eqref{gmonminus},
\begin{align}\label{kerinkern}
\nabla(H^{\pm}(A/S))\subseteq H^{\pm}(A/S)\otimes\Omega_{S/T}.
\end{align}
So the Gauss-Manin connection induces a connection (through the product rule (\ref{prodrule}) and the fact that (\ref{kerinkern}) holds)
$\nabla: T^\bullet (H^{\pm}(A/S))\rightarrow T^\bullet (H^{\pm}(A/S))\otimes \Omega_{S/T}.$
Let $\rho = \rho_+\otimes\rho_-$ be a quotient of $\rhost^{\otimes d_1}\otimes\rhost^{\otimes d_2}$ for some $d_1$ and $d_2$.  Applying the product rule (\ref{prodrule}) again, we get a connection
\begin{small}
\begin{align}\nonumber
\nabla: \dr(A/S)^\rho \otimes T^\bullet (H^+(A/S)\otimes H^-(A/S))\rightarrow \dr(A/S)^\rho\otimes T^\bullet (H^+(A/S)\otimes H^-(A/S))\otimes \Omega_{S/T}.
\end{align}
\end{small}

We define a differential operator 
\begin{small}
\begin{align}\nonumber
D^{\rho}_{A/S}: V_{A/S}:=\dr(A/S)^{\rho}\otimes T^\bullet (H^+(A/S)\otimes H^-(A/S))\rightarrow \dr(A/S)^{\rho}\otimes T^{\bullet+1}(H^+(A/S)\otimes H^-(A/S))
\end{align}
\end{small}
to be the composition of maps:
\begin{small}
\begin{align}\label{ddefdiag}
\xymatrix{
 V_{A/S}:=\dr(A/S)^{\rho}\otimes T^{\bullet} (H^+(A/S)\otimes H^-(A/S))\ar[r]^(.66){\nabla}\ar[drddddd]_{D^{\rho}_{A/S}}  &   V_{A/S}\otimes\Omega_{S/T}\ar[dd]^{id\otimes KS}\\
   &\\
   & V_{A/S}\otimes \uo^+(A/S)\otimes \uo^-(A/S)\ar[dd]^{\iota}\\
  & \\
      & V_{A/S}\otimes H^+(A/S))\otimes H^-(A/S)\ar[dd]^{=}\\
     & \\
    & \dr(A/S)^{\rho}\otimes T^{\bullet+1} (H^+(A/S)\otimes H^-(A/S))
 }
\end{align}
\end{small}
\begin{rmk}
Observe that we can similarly construct an algebraic differential operator 
\begin{small}
\begin{align*}
\tilde{D}^{\rho}_{A/S}: \dr(A/S)^\rho\otimes \sym^\bullet(H^+(A/S)\otimes H^-(A/S))\rightarrow\dr(A/S)^\rho\otimes \sym^{\bullet +1}(H^+(A/S)\otimes H^-(A/S))
\end{align*}
\end{small}
(essentially by replacing $T^\bullet$ with $\sym^\bullet$ in the definition of $D^\rho_{A/S}$).
\end{rmk}

In the case where $A = \A$ and $S= \moduli_{R_0}/R_0$, we define
$D^{\rho}:= D^{\rho}_{\A/\moduli_{R_0}}$.
We denote by $D$ the morphism
$D: T^{\bullet}(\dr(A/S))\rightarrow T^{\bullet+2}(\dr(A/S))$
whose restriction to $T^r(\dr(A/S))$ is $D^{\otimes r}$.  

We write $(D^{\rho}_{A/S})^d$ (resp. $(\tilde{D}^{\rho}_{A/S})^d$) to denote $D^{\rho}_{A/S}$ (resp. $\tilde{D}^{\rho}_{A/S}$) composed with itself $d$ times.

Now we give a formula for the action of $D$ in terms of the basis of invariant one-forms $\ai, \bi, \aip, \bip$ in $\dr(\ci)$ over $\IC$.  So that we can consider all representations of interest simultaneously, we consider the representation $\rho$ as a subrepresentation of the tensor algebra.  So we may view $D^{\rho}$ in terms of the restriction to $\dr(\ci)$ of the morphism $D$, which is a morphism
\begin{align*}
T^{\bullet}(\dr(\ci)) \mapsto T^{\bullet+2}(\dr(\ci))
\end{align*}
that is homogeneous of degree two in the sense that $D$ maps $T^r(\dr(\ci))$ to $T^{r+2}(\dr(\ci))$.  The sheaf $T(\dr(\ci))$ {\it is} the sheaf $\mathcal{R}$ of graded non-commutative $\mathcal{O}_{\moduli}(\ci)$-algebras generated by the horizontal sections $\ai, \bi, \aip, \bip$ (defined in Subsection \ref{gmex}) with no relations other than those in the commutative ring $\mathcal{O}_{\moduli}(\ci)$.  Lemma \ref{thexpl} gives the action of $D$ on $\mathcal{R}$ explicitly.
\begin{lem}\label{thexpl}
Viewed as a morphism on $\mathcal{R}$, the action of $D$ is defined for all sections $f$ of $\mathcal{O}_{\moduli}(\ci)$ and nonnegative integers $\kappa_i, \kappa_i', \lambda_i, \lambda_i'$ by
\begin{align}
&D\lp \lp\prod_{1\leq l\leq n}\alpha_l^{\kappa_l}\alpha_l'^{\kappa_l'}\beta_l^{\lambda_l}\beta_l'^{\lambda_l'}\rp f\rp =\nonumber\\
 &\prod_{1\leq l\leq n}\alpha_l^{\kappa_l}\alpha_l'^{\kappa_l'}\beta_l^{\lambda_l}\beta_l'^{\lambda_l'} \sum_{1\leq i, j\leq n}\frac{\partial f}{\partial z_{ij}}\cdot \lp Q_j\cdot P_i\rp,\label{twopimis}
\end{align}

with $P_i$ and $Q_j$ the elements of $\mathcal{R}$ defined by
\begin{align*}
P_i = \alpha_i+\sum^n_{k=1}z_{ik}\beta_k+\bar{\alpha}\alpha'_i+\bar{\alpha}\sum_{k=1}^n z_{ik}\beta_k' \end{align*}
and
\begin{align*}
Q_j = \alpha_i+\sum^n_{k=1}z_{kj}\beta_k+\alpha\alpha'_i+\alpha\sum_{k=1}^n z_{kj}\beta_k'.
\end{align*}
For all $v$ and $w$ in $\mathcal{R}$,
\begin{align}\label{Dcom}
D(v+w) = D(v) + D(w).
\end{align}
\end{lem}
\begin{proof}
Equation \eqref{Dcom} follows immediately from the definition of $D$.  Since  $\ai, \bi, \aip, \bip$ are horizontal,
\begin{align*}
\nabla\lp \lp\prod_{1\leq l\leq n}\alpha_l^{\kappa_l}\alpha_l'^{\kappa_l'}\beta_l^{\lambda_l}\beta_l'^{\lambda_l'}\rp f\rp = \prod_{1\leq l\leq n}\alpha_l^{\kappa_l}\alpha_l'^{\kappa_l'}\beta_l^{\lambda_l}\beta_l'^{\lambda_l'} \sum_{1\leq i, j\leq n}\frac{\partial f}{\partial z_{ij}}\cdot dZ_{ij}.
\end{align*}
Recall from \eqref{ksdz} that the injection
$\Omega\hookrightarrow T^2(\dr)$
defined by the Kodaira-Spencer isomorphism is given by
$dZ_{ij} \mapsto du_{j+n}\otimes du_i.$
By \eqref{duis1}--\eqref{duis2}, we see that in $\mathcal{R}$,
\begin{align*}
du_i &= P_i\\
du_{j+n} &= Q_j,
\end{align*}
for $1\leq i, j\leq n$.  So now the lemma follows from the definition of $D$. \end{proof}
Note that since 
$\nabla(H^{\pm})\subseteq H^{\pm}\otimes \Omega,$
restriction of $D_{\rho}$ to $(H^\pm)^{\rho_\pm}$ gives maps
\begin{align*}
(H^{\pm})^{\rho_{\pm}}\rightarrow (H^\pm)^{\rho_{\pm}\otimes\rhost}\otimes (H^{\mp})^{\rhost}.
\end{align*}

So through the product rule, $D$ induces morphisms
\begin{align*}
(\dr(A/S)^+)^{\rho_+}\otimes (\dr(A/S)^-)^{\rho_-}\longrightarrow (\dr(A/S)^+)^{\rho_+\otimes \rhost}\otimes (\dr(A/S)^-)^{\rho_-\otimes\rhost}.
\end{align*}

\subsection{Some algebraically defined maps on automorphic forms}\label{agdefmaps}

The maps defined in this section will be used in the proofs of Theorems \ref{cialgthm} and \ref{paalgthm}, algebraicity theorems about the differential operators defined in Sections \ref{cioperatordefinition} and \ref{padicoperatordefinition}.  Our construction is completely analogous to the one in \cite{kaCM}, and we follow \cite{kaCM} closely.  Our construction here is a general vector-valued construction that generalizes the scalar-valued one in \cite{kaCM} to our higher-dimensional setting.

We work over an $\OK$-algebra $R_0$.  Let $R$ be an $R_0$-algebra, and let $x$ be an $R$-valued point of the moduli scheme $\moduli_{R_0}$ over $R_0$, corresponding to a morphism 
$\Spec(R)\rightarrow\moduli$
over $R_0$.  Let $\underline{X}$ denote the associated abelian variety $X$ with the associated PEL structure.  Let $\lambda$ be an element of $\EE_{X/R}$.  

Suppose that we are given an $R$-sub-module 
$\Split(X/R)$
 in $\dr(X/R)$ such that the natural map (induced by the inclusions)
\begin{eqnarray}\label{splitting}
\omega_{X/R}\oplus \Split(X/R){\rightarrow} H^1_{DR}(X/R)
\end{eqnarray}
is an isomorphism and such that
$\dr(X/R)^{\pm} \subseteq \Omega_{X/R}^{\pm}\oplus\Split(X/R).$
(For example, when $R = \IC$ and we work in the $\ci$-category, the Hodge decomposition gives us a splitting in which we can take $\Split(X/R)$ to be the sheaf of anti-holomorphic one-forms.)

As earlier, we let $\agauto_{\rho}(\level)(R_0)$ denote the space of automorphic forms over $R_0$ of weight 
\begin{align*}
(\rho=\rho^-\otimes\rho^+, V = V^-\otimes V^+)
\end{align*}
 and level $\level$.  Let $e$ and $d$ be positive integers.  In this section, we define an $R_0$-linear map
\begin{eqnarray}
\partial(\rho, e, x,\lambda, \Split(X/R), d): \agauto_{\rho\otimes\tau^e}(R_0)\rightarrow V^-\otimes V^+\otimes (R^n\otimes R^n)^{\otimes e+d}.\nonumber
\end{eqnarray}

We identify each automorphic form $f$ in $\agauto_{\rho\otimes\tau^e}(R_0)$ with the corresponding global section of $(\uo^-\otimes\uo^+)^{\rho}\otimes(\Omega_{\moduli/R_0}^{\otimes e})$, as in Section \ref{agafapproach}.  The canonical inclusion
$\uo(X/R)^{\pm}\hookrightarrow \dr(X/R)^{\pm}$
and the Kodaira-Spencer map \eqref{ks2} induce inclusions
\begin{align}\label{roin}
(\uo_{X/R}^-\otimes\uo_{X/R}^+)^{\rho}\otimes(\Omega_{\moduli/R_0}^{\otimes e})\hookrightarrow (\dr^-\otimes\dr^+)^{\rho_-\otimes\rho_+}\otimes (\dr^+\otimes\dr^-)^{\otimes e}.
\end{align}
Associate $f$ with its image in $(\dr\otimes\dr)^{\rho}\otimes ({\dr}^+\otimes{\dr}^-)^{\otimes e}$ via the inclusion (\ref{roin}).  Then for each integer $d$, $(D^{\rho})^d(f)$ is a global section of $(\dr^+\otimes\dr^-)^{\rho}\otimes ({\dr}^+\otimes{\dr}^-)^{\otimes e+d}.$  Thus,

\begin{small}
\begin{align}\nonumber
((D^{\rho})^d(f))(x)\in (\dr^-(X/R)\otimes\dr^+(X/R))^{\rho_-\otimes\rho_+}\otimes\lp (\dr^+(X/R)\otimes\dr^-(X/R)\rp^{\otimes d+e}.
\end{align}
\end{small}

The choice of $\lambda$ gives isomorphisms
$\lambda^{\pm}:\omega_{X/R}^{\pm}\isomto R^{n},$
which induce isomorphisms
\begin{align}\nonumber
(\omega_{X/R}^-\otimes\omega_{X/R}^+)^{\rho_-\otimes\rho_+}\otimes (\uo_{X/R}^-\otimes\uo_{X/R}^+)^{\otimes e+d}\isomto V^-\otimes V^+\otimes (R^n\otimes R^n)^{\otimes e+d}
\end{align}

The splitting (\ref{splitting}) gives projections
\begin{align}\label{projec}
\dr(X/R)^{\pm}\rightarrow \omega_{X/R}^{\pm},
\end{align}
which induce projections
$(\omega^{\pm}\oplus \Split(X/R))^{\rho_{\pm}}\rightarrow (\omega^{\pm}_{X/R})^{\rho_{\pm}}.$

The projection (\ref{projec}) also induces a projection
\begin{align*}
(\dr(X/R)^-\otimes \dr(X/R)^+)^{\otimes d+e}\rightarrow (\omega_{X/R}^-\otimes\omega_{X/R}^+)^{\tau^{\otimes (d+e)}}
\end{align*}
and a projection
$(\omega_{X/R}^-\oplus \Split(X/R))^{\rho_-}\otimes(\omega_{X/R}^+\oplus \Split(X/R))^{\rho_+}\rightarrow (\omega_{X/R}^-)^{\rho_-}\otimes(\omega_{X/R}^+)^{\rho_+}.$

We now define the $R_0$-linear map 
\begin{eqnarray}
\partial(\rho, e, x,\lambda, \Split(X/R), d): \agauto_{\rho\otimes\tau^e}(R_0)\rightarrow V^-\otimes V^+\otimes (R^n\otimes R^n)^{\otimes e+d},\nonumber
\end{eqnarray}
as follows.  We define $\partial(\rho, e, x,\lambda, \Split(X/R), d)(f)$ to be the image of $(D^{\rho^+\otimes \rho^-})^d(f)\in (\dr^+)^{\rho^+\otimes \rhost^d}\otimes (\dr^-)^{\rho^-\otimes \rhost^d}$ under the composition of morphisms given by the diagonal map in the commutative diagram \eqref{acom}:

\begin{small}
\begin{align}\label{acom}
\xymatrix{
(\dr^+)^{\rho^+\otimes \rhost^d}\otimes (\dr^-)^{\rho^-\otimes \rhost^d}\ar@{.>}[rrdddd]\ar[rr]^-{g\mapsto g(x)} & &(\omega_{X/R}^-\oplus \Split(X/R))^{\rho_-\otimes\rhost^d}\otimes(\omega_{X/R}^+\oplus \Split(X/R))^{\rho_+\otimes\rhost^d}\ar[dd]^{\mod \Split(X/R)}  \\
 & & \\
& & \mformsxrd\ar[dd]^{\lambda}\\
& &\\
& & V^-\otimes V^+\otimes (R^n\otimes R^n)^{\otimes e+d}
 }
\end{align}
\end{small}

For each  
$R$-submodule $Z$ that is a $GL_n(R)\otimes GL_n(R)$-stable quotient of $\mformsxrd$, define $\phi_Z$\index{$\phi_Z$} to be the projection of $\mformsxrd$ onto $Z$.  Identify $Z$ with $\lambda(Z)$.  Then we define 
$\partial(\rho, e, x,\lambda, \Split(X/R), d)^Z = \phi_Z\circ \partial(\rho, e, x,\lambda, \Split(X/R), d).$

\section{The $\ci$-differential operators}
\label{cinf}
\subsection{Construction of the $\ci$ differential operators}\label{cioperatordefinition}

Later, we will define $p$-adic differential operators through a construction similar to the one in this section.

Let 
\begin{eqnarray}\label{cinfinitysplitting}
H^1_{DR}(\ci) = \omci\oplus \splci
\end{eqnarray}
be the canonical splitting of the Hodge filtration corresponding to the holomorphic and anti-holomorphic one-forms.  Here, $\Split(C^\infty)$ is the sheaf $\overline{\uo(\ci)}$ of anti-holomorphic one forms.  Note that for each derivation $D\in Der(\OM^{\ci}, \OM^{\ci})$,
$\nabla(D)(\overline{\uo(\ci)})\subset\overline{\uo(\ci)}.$

Since 
$\dr(\ci)^{\pm}\subseteq\omci^{\pm}\oplus\splci,$
the splitting \eqref{cinfinitysplitting} induces projections
$\dr(\ci)^{\pm}\rightarrow\omci^{\pm},$
which induces a projection
\begin{align}\label{proji}
\dr(\ci)^\rho\otimes T^\bullet(\dr^+(\ci)\otimes\dr^-(\ci))&\rightarrow \omci^\rho\otimes T^\bullet(\uo^+(\ci)\otimes \uo^-(\ci))\\
&\isomto \omci^\rho\otimes T^\bullet(\Omega(\ci)).\nonumber
\end{align}

As usual, we associate $\uo^\pm$ with its image in $(\dr)^{\pm}$ under the inclusion coming from hypercohomology
\begin{align}\label{uointodr}
\uo\hookrightarrow\dr.
\end{align}
As in (\ref{roin}),  the inclusion \eqref{uointodr} and the Kodaira-Spencer isomorphism \eqref{ks2} induce inclusions
\begin{align}\label{uor}
\mformsnost\otimes(\Omega_{\moduli/R_0}^{\otimes e})\hookrightarrow (\dr)^{\rmrp}\otimes (\dr^-\otimes\dr^+).
\end{align}
Restricting $D^\rho$ to the image of \eqref{uor}, we get a map

\begin{small}
\begin{align}\nonumber
(D^\rho)^d\rvert_{\mformsnost\otimes(\Omega_{\moduli/R_0}^{\otimes e})}: \mformsnost\otimes(\Omega_{\moduli/R_0}^{\otimes e})\rightarrow \dr(\ci)^\rmrp\otimes (\dr^+\otimes \dr^-)^{\otimes e+d}.
\end{align}
\end{small}

We define the $C^\infty$-differential operator $\ciop$ to be the map
\begin{align}\nonumber
\ciop: \mformsnost\otimes(\Omega(\ci)^{\otimes \bullet})\rightarrow  \mformsnost\otimes(\Omega(\ci))^{\otimes \bullet+d})\end{align}
that is the composition of maps in the following commutative diagram:

\begin{small}
\begin{align}\nonumber
\xymatrix{
\mformsnost(\ci)\otimes(\Omega(\ci)^{\otimes \bullet})\ar@{.>}[rrrdddd]_{\ciop}\ar[rrr]^{(D^\rho)^d}  &  & &\dr(\ci)^\rho\otimes (\dr^+\otimes \dr^-)^{\bullet+d}\ar[dd]^{\mod \splci}\\
& &&\\
 & &&\mformsnost(\ci)\otimes (\uo^+\otimes\uo^-)^{\otimes \bullet+d}\ar[dd]^{\cong}\\
 & &&\\
&& & \mformsnost(\ci)\otimes (\Omega(\ci))^{\bullet+d}
 }
\end{align}
\end{small}

\begin{rmk}\label{adhocrmk}
For the reader who worries that defining the differential operators on $\mformsnost$ (instead of $\EEVr$) is too {\it ad hoc} or non-canonical, we note that the differential operators can equivalently be defined as morphisms from $\EEVr$ to $\EE_{V, (\rho_+\otimes\rhost)\otimes(\rho_-\otimes\rhost)}$.  Indeed, the composition of maps

\begin{footnotesize}
\begin{align}\nonumber
\xymatrix{
\EEVr\ar@{.>}[rrrr]\ar[dd]_{(\lambda, v)\mapsto \lambda(v)}&&&& \EE_{(V_-\otimes \IC^n)\otimes(V_+\otimes\IC^n), (\rho_-\otimes\rhost)\otimes(\rho_+\otimes\rho_+)}\\
&&&&\\
\mformsnost\ar[rrr]_{\ciop}&&&\mformsnost\otimes\Omega\ar[r]_{KS}&\mformsst\ar[uu]_{v\mapsto (\lambda, \lambda^{-1}(v))}
 }
\end{align}
\end{footnotesize}

gives an equivalent expression of our differential operator as an operator from  $\EEVr$ to $\EE_{(V\otimes(R^n)\otimes(R^n)), (\rho_+\otimes\rhost)\otimes(\rho_-\otimes\rhost)}$.
\end{rmk}

Let $Z$ be a $GL_n(\IC)\times GL_n(\IC)$-stable quotient of 
$\mformsnost(\ci)\otimes (\Omega_{\A/\moduli}(\ci))^{\bullet+d},$
 and let $\phi_Z$ be the projection of $\mformsnost(\ci)\otimes (\Omega_{\A/\moduli}(\ci))^{\bullet+d}$ onto $Z$.  We define the differential operator $\ciopo^Z$ by 
$\phi_Z\circ \ciop.$

\subsection{Algebraicity theorem for $\ci$-differential operators}
The following algebraicity theorems (Theorems \ref{cialgthm} and \ref{cialgthm2}) are important for our intended applications.  The statement of Theorem \ref{cialgthm} and the idea of the proof are essentially the same as what is done in Section 2.4 of \cite{kaCM}; the new parts are our generalizations from Katz's special scalar-valued case to the arbitrary (often vector-valued) case and to the case of projection onto subrepresentations.

Throughout this section, fix an $\OK$-algebra $R$ with an inclusion
$\ir: R\hookrightarrow\IC.$

In the special case $R=\bar{\IQ}$, the statement of Theorem \ref{cialgthm} is essentially the same as Theorem 14.9 (2) of \cite{shar}.  However, the methods or the proof of Theorem \ref{cialgthm} are different from the proof in \cite{shar}; the proof we present is similar to what is done in Section 2.4 of \cite{kaCM}.  

We associate each automorphic form $f$ in $\agauto_{\rho\otimes\tau^e}(R)$ with its image in $\agauto_{\rho\otimes\tau^e}(\IC)$ via the extension of scalars induced by $\ir$.  We also associate each automorphic form in $\agauto_{\rho\otimes\tau^e}(\IC)$ with the corresponding holomorphic section of $\mformsnost(\ci)\otimes\underline{\Omega}(\ci)^{\otimes e}$ on $\moduli(\ci)$.

Let $x=\underline{X}$ be an $R$-valued point of $\moduli_R$, and let $\lambda$ be an element of $\EE_{X/R}$.  Suppose that over $R$ there is a splitting
$\splxr\oplus\omxr\isomto\drxr.$
Then we have an inclusion
\begin{align}\label{incah}
\Split(X/R)\hookrightarrow H^1(X^{an}_{\IC}, \IC)
\end{align}
coming from the composition of maps
\begin{align}\nonumber
\xymatrix{
\splxr\ar@{^{(}->}[rrrr]^{\mbox{ext. of scalars via }\ir}\ar@{.>}[rrrrrdd]&&&& \splxr\otimes\IC\ar@{^{(}->}[r]^{\eqref{incah}\otimes \id}&\dr(X/R)\otimes\IC\ar[dd]^{\cong}\\
&&&&&\\
&&&&& H^1(X^{an}_{\IC}, \IC)
 }
\end{align}

We say that the the pair $(x, \splxr)$ satisfies the condition \pstrict   if the following holds:

\begin{align}
\mbox{The image of $\splxr$ under the inclusion } &\mbox{\eqref{incah} is the antiholomorphic subspace}\nonumber\\ H^{0,1}&\subset H^1({X^{an}}_{\IC}, \IC),\nonumber\\
 \mbox{ i.e. }\splxr\otimes\IC& = \splci(x_{\IC}).\tag{\strict}
\end{align}

Note that this condition is essentially the same as condition (2.4.2) in \cite{kaCM}.

For our intended applications, the only points that will interest us are certain ordinary CM points.  We shall see later that for each such ordinary CM point, there is indeed a splitting satisfying condition \pstrict.

Note that in general, we only know that the values of $\ciop(f)$ at points $(x, \lambda)$ lie in a $\IC$-vector space.  (This is because even if $f\in \agauto_{\rho}(R)$, we only know that $\ciop(f)$ is a $\ci$-function and nothing about where its values at arbitrary points lie.)  We see in theorem \ref{cialgthm}, however, that we can say much more about the values of $\ciop(f)$ at points satisfying \pstrict.

\begin{thm}\label{cialgthm}
Suppose that $(x, \splxr)$ is an $R$-valued point of $\moduli_R$ that satisfies condition \pstrict.  Let $f$ be an automorphic form in $\agauto_{\rho\otimes\tau^e}(R)$ with values in an $R$-module $V$.  Then 
\begin{align}
(\ciop f)(x, \lambda)_{\IC} = \ir (\partial(\rho, e, x, \lambda, \splxr, d)f).\label{surpeq}
\end{align}
Therefore,
$(\ciop f)(x, \lambda)_{\IC}\in V\otimes_R \left(R^n\otimes R^n\right)^{\otimes d}.$
\end{thm}
The proof we provide is similar to Katz's proof of Theorem 2.4.5 in \cite{kaCM}.
\begin{proof}
As it was defined in Section \ref{agdefmaps}, $\ir (\partial(\rho, e, x, \lambda, \splxr, d)f)$ lies in 
$V\otimes_R \left(R^n\otimes R^n\right)^{\otimes d}.$
So to prove the theorem, it suffices to prove that Equation \eqref{surpeq} holds.  

By the extension of scalars from $R$ to $\IC$ given by $\ir$, we associate the automorphic form $f$ with its image $f_{\IC}$ in $\agauto_{\rho\otimes\tau^e}(\IC)$, the $R$-valued point $x$ with $R$-basis $\lambda$ with $x$ with basis $\lambda_{\IC}$, and $\splxr$ with its image in $\splxr\otimes\IC$.  Then, we see that
$\ir (\partial(\rho, e, x_{\IC}, \lambda_{\IC}, \splxr_{\IC}, d)f_{\IC}) = \ir(\partial(\rho, e, x, \lambda, \splxr, d)f)$
is $V\otimes_R \left(R^n\otimes R^n\right)^{\otimes d}$-valued.  So it suffices to show that Equation \eqref{surpeq} holds in the case $R=\IC$, which we will now do.

Associate $f$ with its image in $(\dr)^\rho\otimes(H^+\otimes H^-)^{\otimes e}$.  Then $(\partial(\rho, e, x, \lambda, \splxr, d)f)(x)$  is obtained by applying $(D^\rho)^d$ to $f$ and composing with the maps along the right side  of the commutative diagram \eqref{algcomdiag}.  (Commutativity in \eqref{algcomdiag} follows from the hypothesis that $(x, \splxr)$ satisfies condition \pstrict.)  Similarly, $(\ciop f)(x)$ is obtained by applying $(D^\rho)^d$ to $f$ and composing with the maps along the left side of the commutative diagram \eqref{algcomdiag}.  So \eqref{surpeq} holds for all $x$ with a splitting satisfying \pstrict.

\begin{footnotesize}
\begin{align}\label{algcomdiag}
\xymatrix{
(\dr)^\rho\otimes(\dr^+\otimes\dr^-)^{\otimes e+d}\ar[ddr]^{\mbox{     take fiber at } x}\ar[dd]_{\mod\splci}&\\
&\\
\mformsnost(\ci)\otimes(\uo^+(\ci)\otimes\uo^-(\ci))^{\otimes e+d}\ar[ddr]_{\mbox{     take fiber at } x}&(\dr(X/R))^\rho\otimes((\dr^+\otimes\dr^-)(X/R))^{\otimes e+d}\ar[dd]^{\mod \splxr}\\
&\\
&((\uo^+\otimes\uo^-)(X/R))^\rho\otimes((\uo^-\otimes\uo^+)(X/R))^{\otimes e+d}\ar[dd]^{\cong \mbox{ from choice of } \lambda}\\
&\\
&(V\otimes_R\IC)\otimes_{\IC}(\IC^n\otimes \IC^n)^{\otimes d}
 }
\end{align}
\end{footnotesize}

\end{proof}

We also obtain the following generalization of Theorem \ref{cialgthm}:
 
\begin{thm}\label{cialgthm2}
Suppose that $(x, \splxr)$ is an $R$-valued point of $\moduli_R$ that satisfies condition \pstrict.  Let $f$ be an automorphic form in $\agauto_{\rho\otimes\tau^e}(R)$.  Let $Z$ be a $GL_n(R)\times GL_n(R)$-stable $R$-quotient of $\uo_R^\rho\otimes (\Omega_{\A/\moduli_R})^{\bullet+d}$, and let $\phi_Z$ be the projection of $\omci^\rho\otimes (\Omega_{\A/\moduli}(\ci))^{\bullet+d}$ onto $Z\otimes_R\IC$.  Let $Z_x$ be the fiber of $Z$ at $x$.  Then 
$(\ciop^Z f)(\x,\lambda)_{\IC} = \ir (\partial(\rho, e, x, \lambda, \splxr, d)^{Z_X}f).$
Therefore,
$(\ciop^Z f)(x,\lambda)_{\IC}$
actually takes values in the $R$-module $\lambda(Z)$.
\end{thm}

The proof is similar to the proof of Theorem \ref{cialgthm}, except that we also compose with the projection onto $Z$.  Note that like in Remark \ref{adhocrmk}, restriction to $Z$ yields a canonical map
$\EE_{V, \rho}\rightarrow \EE_Z.$

\subsection{Some properties of the $\ci$-differential operators}
In this section, we give some fundamental properties of the $\ci$-differential operators $\ciop$.

We denote by $(\ciopo)^d$ the composition of $\ciop$ with itself $d$ times.
\begin{thm}\label{impprop}
The differential operators satisfy the following properties.
\begin{enumerate}
\item{For each positive integer $d$, \begin{align}\label{afciopo}
\ciop = (\ciopo)^d
\end{align}
}
\item{
Associating the space of automorphic forms of weight $\rho\otimes \tau^f$ with $\mformsnost(\ci)\otimes\Omega(\ci)^{\otimes f}$ via the natural isomorphism induced by the Kodiara-Spencer isomorphism, we have that $\ciop$ is the same as $\partial(\rho\otimes\tau^f, \ci, d)$, i.e. for all positive integers $f\leq e$,
\begin{align}
\ciop&\rvert_{\mformsnost(\ci)\otimes\Omega_{\A/\moduli}(\ci)^{\otimes e}}\nonumber\\
&= \partial(\rho\otimes\tau^f, \ci, d)\vert_{\mformsnost(\ci)\otimes\Omega_{\A/\moduli}(\ci)^{\otimes e}} \label{rhotorhotci}
\end{align}
}
\item{For all positive integers $d\geq 2$,
\begin{align}
\partial(\rho\otimes\tau^{d-1}, \ci, 1)\circ\cdots\circ\partial(\rho\otimes\tau, \ci, 1)\circ\ciopo = \ciop\label{likeshci}
\end{align}
}
\end{enumerate}
\end{thm}
\begin{proof}
Recall that $\splci$ is horizontal with respect to $\nabla$, i.e. 
$\nabla(\splci)\subseteq \splci\otimes\Omega_{\A/\moduli}(\ci).$  
So from the definition of $D^{\rho}$, we see that
$D^{\rho}(\splci)\subseteq \splci,$
and hence, for all positive integers $d$,
$(D^{\rho})^d(\splci) \subseteq \splci.$  So
$(D^{\rho})^d\circ(\mbox{``$\mod\splci$"}) = (D^{\rho}\circ(\mbox{``$\mod\splci$"}))^d,$
where $(D^{\rho}\circ(\mbox{``$\mod\splci$"}))^d$ denotes $D^{\rho}\circ(\mbox{``$\mod\splci$"})$ composed with itself $d$ times.
Therefore, it follows directly from the definition of $\ciop$ that
$\ciop = (\ciopo)^d.$
So \eqref{afciopo} holds.

Equation \eqref{rhotorhotci} follows directly from the definition of $\tau$, our earlier explicit description of the Kodaira-Spencer isomorphism, and the definition of the map $\ciop$ for any representation $\rho$.

Now, $\ciopo$ has image in $\omci^{\rho}\otimes\sum_{e = 1}^{\infty}\Omega^{\otimes e}(\ci)$.  So \eqref{likeshci} is a corollary of \eqref{afciopo} and \eqref{rhotorhotci}.
\end{proof}

\subsection{Comparison with Shimura's $\ci$-differential operators}\label{rho04301}
\begin{prop}\label{rho0430}  Let $D_{\rho}$ be Shimura's differential operator discussed in the introduction (cf. section 12.1 of \cite{shar}).  
Let $f: \hn\rightarrow V = V_-\otimes V_+$ be a $\ci$-function.  Let $\lambda\in\EE$.  Then
\begin{align}\label{frho}
\ciopo\left(\lambda(f)\right)= \lambda(D_{\rho} f).
\end{align}
\end{prop}

(In Equation \eqref{frho}, $\lambda$ refers to the induced map from $V_{\pm}$ to $\uo_{\pm}^{\rho_{\pm}}$.)

\begin{proof}
Define
$v^{\pm}_{\lambda_{\pm}} = \lambda^{-1}(du^{\pm}_{\lambda_{\pm}})$
for each tuple $\lambda_{\pm}$.
Writing $f$ in terms of the basis $v_{\lambda_{\pm}}$ for $V^{\pm}$, we have
\begin{align}\label{sumoverl}
f = \sum_{\lambda_-, \lambda_+} f_{\lambda_-,\lambda_+}\lambda(v^-_{\lambda_-}\otimes v^+_{\lambda_+}),
\end{align}
for some $\ci$ complex valued functions $f_{\lambda_-, \lambda_+}$.  The sum in Equation \eqref{sumoverl} is, as usual, over all tuples $\lambda_{\pm}$ so that $v_{\lambda_{\pm}}$ is in $V^{\pm}$.  Note that
\begin{align}
\ciopo(\lambda(f(z))) &= \ciopo\left(\sum_{\lambda_-,\lambda_+}f_{\lambda_-, \lambda_+}\lambda(v^-_{\lambda_-}\otimes v^+_{\lambda_+})\right)\label{regbar1}\\
&= \ciopo\left(\sum_{\lambda_-,\lambda_+}f_{\lambda_-, \lambda_+}(du^--d\bar{u}^-)_{\lambda_-}\otimes (du^+-d\bar{u}^+)_{\lambda_+}\right).\label{regbar2}
\end{align}
We get from \eqref{regbar1} to \eqref{regbar2} by recalling that $\ciopo(\bar{\uo})=0$, since $\bar{\uo}$ is holomorphically horizontal with respect to $\nabla$.

Applying Equations \eqref{vecdui1} and \eqref{vecdui2} to Equation \eqref{regbar2}, we see that
\begin{align*}
\ciopo&(\lambda(f(z)))= \ciopo\left( \rho(\Xi(z)) \left(\sum_{\lambda_-,\lambda_+}f_{\lambda_-, \lambda_+}(\beta+\bar{\alpha}\beta)_{\lambda_-}\otimes (\beta+\alpha\beta)_{\lambda_+}\right)\right).
\end{align*}

Now recall that the sections $\beta_i+\bar{\alpha}\beta_i$ and $\beta_i+\alpha\beta_i$ are horizontal for the Gauss-Manin connection.  So their image under $\ciopo$ is zero.

Therefore 
\begin{small}
\begin{align*}
\ciopo(\lambda(f(z))) = D\left( \rho(\Xi(z)) \left(\sum_{\lambda_-,\lambda_+}f_{\lambda_-, \lambda_+}(\beta+\bar{\alpha}\beta)_{\lambda_-}\otimes (\beta+\alpha\beta)_{\lambda_+}\right)\right)\mod\splci,
\end{align*}
\end{small}
where $D$ is the map \eqref{defnD}.  Applying \eqref{vecdui1} and \eqref{vecdui2} again, we obtain
\begin{small}
\begin{align*}
\ciopo&(\lambda(f(z)))\\
&=\rho^{-1}(\Xi(z)) D\left( \rho\left(\Xi(z)\right) \left(\sum_{\lambda_-,\lambda_+}f_{\lambda_-, \lambda_+}(du^--d\bar{u}^-)_{\lambda_-}\otimes (du^+-d\bar{u}^+)_{\lambda_+}\right)\right)\mod\splci\\
&= \rho^{-1}(\Xi(z)) D\left( \rho\left(\Xi(z)\right) \left(\sum_{\lambda_-,\lambda_+}f_{\lambda_-, \lambda_+}(du^-)_{\lambda_-}\otimes (du^+)_{\lambda_+}\right)\right)\\
& = \rho^{-1}(\Xi(z)) D\left( \rho\left(\Xi(z)\right) \lambda(f)\right).
\end{align*}
\end{small}\end{proof}

Let $Z$ be a $GL_n(\IC)\times GL_n(\IC)$-stable quotient of 
$\omci^\rho\otimes (\Omega_{\A/\moduli}(\ci))^{\bullet+d},$
 and let
$\mathcal{Z} = H^0(\hn, Z).$
Let $\phi_{\mathcal{Z}}$ denote projection onto $\mathcal{Z}$.  Recall that Shimura defines a differential operator $D_{\rho}^\mathcal{Z}$ (see e.g. \cite{shar}) by
$D_{\rho}^\mathcal{Z} = \phi_{\mathcal{Z}}D^d.$
So as a corollary of Proposition \ref{rho0430}, we obtain the following
\begin{cor}
$D_{\rho}^\mathcal{Z} = \ciopo^Z.$
\end{cor}

\setcounter{equation}{0}

\section{The $p$-adic differential operators}
\label{kapa}
In this section, we construct $p$-adic differential operators that act on $p$-adic automorphic forms, and we discuss basic properties of these operators.  We follow the arguments from \cite{kaCM} closely.  Rather than the notation from \cite{kaCM}, however, we use the notation from \cite{SHL} and \cite{hida}, since this is what we will use in our applications.

Throughout this section, let $R$ be an $\OK$-algebra that is separated for the $p$-adic topology, and let $R_0$ be the $p$-adic completion of $R$.

\subsection{Analogue of the differential operators $\ciop$}\label{padicoperatordefinition}

We define $p$-adic differential operators
\begin{align*}
\paop: \mformspa\rightarrow\mformspa\otimes \Omega(\padic)
\end{align*}
in the same way as the $\ci$-differential operators $\ciop$, except that we replace $\splci$ and $\Omega(\ci)$ with $\splpa$ and $\Omega(\padic)$, respectively.  Replacing all occurrences of $\ci$ with $\padic$ in the proof of Theorem \ref{impprop}, we see that all the properties of the $\ci$-operators given in Theorem \ref{impprop} also hold for the $p$-adic operators.
One can construct operators $\paop^Z$ similarly to how one constructs the operators $\ciopo^Z$.  The operators $\paop^Z$ have properties similar to those of $\ciopo^Z$.

\subsection{$p$-adic arithmeticity result}

In this section, following \cite{kaCM}, we give a $p$-adic analogue of the algebraicity theorem \ref{cialgthm}.  

Let $x=\underline{X}$ be an $R$-valued point of $\moduli_R$ with an element $\lambda$ of $\EE_{x/R}$.  Suppose there is a splitting over $R$
\begin{align}\nonumber
\splxr\oplus\omxr\isomto\drxr.
\end{align}
Then, similarly to the $\ci$ case, we have an inclusion
\begin{align}\label{incah2}
\splxr\hookrightarrow \dr(\padic)(X/R_0)
\end{align}

We introduce a $p$-adic analogue of the condition \pstrict.  We say that the the pair $(x, \splxr)$ satisfies the condition \pstrictp if the following holds:
\begin{align}
\mbox{The image of $\splxr$ under the inclusion } &\mbox{\eqref{incah2} is the unit root subspace}\nonumber\\ \uroot(X/R_0)&\subset \dr(\padic)(X/R_0),\nonumber\\
 \mbox{ i.e. }\splxr\otimes R_0& = \uroot (X/R_0).\tag{\ddag}
\end{align}

Let $f$ be an automorphic form of weight $\rho$ over $R$, and associate it with a corresponding element of $\uo^{\rho}_R$.  Then by the extension of scalars $R\hookrightarrow R_0$, we can view $f$ as a section $f(\padic)$ of $\ompa^{\rho}$. 

We now give a $p$-adic analogue of the algebraicity theorem \ref{cialgthm}.

\begin{thm}\label{paalgthm}
Suppose that $(x, \splxr)$ is an $R$-valued point of $\moduli_R$ that satisfies condition \pstrictp, and let $\lambda$ be an element of $\EE_{x/R}$.  Let $f$ be an automorphic form in $\agauto_{\rho\otimes\tau^e}(R)$ with values in the $R$-module $V$.  Then 
$(\paop f(\padic))(x,\lambda)_{R_0} = \ir (\partial(\rho, e, x, \lambda, \splxr, d)f).$
Therefore,
$(\paop f(\padic))(x,\lambda)_{R_0}$
lies in the $R$-module $V\otimes\left(R^n\otimes R^n\right)^{\otimes d}$.
\end{thm}

The proof of Theorem \ref{paalgthm} is similar to the proof of Theorem \ref{cialgthm}; simply replace ``$\ci$" with ``$\padic$."  We obtain a similar arithmeticity result for the operators $\paop^Z$.  We note that the only points that matter in our intended applications are certain ordinary CM points.  We shall see soon that for each such ordinary CM point there is a splitting that satisfies condition \pstrictp.

\subsection{Differential operators on $p$-adic automorphic forms}\label{padicformulae}

In this section, we construct a morphism $\theta$ that acts on the space of $p$-adic automorphic forms.  The operator $\theta$ is a vector-valued analogue of Ramanujan's operator $q\frac{d}{dq}$.  

Recall that $p$-adic automorphic forms are actually certain sections of modules of the form $\Oigg\otimes_R V$ with $(\rho, V)$ a representation of $GL_n\times GL_n$.  Thus, all $p$-adic automorphic forms appear as certain sections of the tensor product of total tensor algebras
$T(\Oigg)\otimes T(\Oigg),$
which we associate with the tensor product of the free algebras on $n$ letters
$\noncomalg = \Oig\otimes_R\freealgrt\otimes_R\freealgrt,$
via
\begin{align}\label{eitoti}
e_i\otimes e_j\mapsto T_i\otimes T_j,
\end{align}
where $e_i$ denotes the $i$-th standard basis element of $\Oigg$.
This viewpoint allows us to consider $p$-adic automorphic forms as subsections of the algebra $\noncomalg$.  In the $p$-adic modular forms case, this algebra is simply the ring of $p$-adic modular forms.

We use this viewpoint in this section, not only because it conveniently allows us to consider modular forms of all different weights at once, but also because this viewpoint is important for applications involving construction of families of $p$-adic automorphic forms of different weights.

While not in general a derivation of $\noncomalg$ over $R$, the morphism $\theta$ that we construct later in this section extends to an $R$-derivation of the commutative subalgebra
$\Oig\otimes_R R[T_1, \ldots, T_n]\otimes R[T_1, \ldots T_n].$

Note that composition of the canonical isomorphism
$\omcan: \Oigg\otimes \Oigg\isomto \uo^-(\padic)\otimes \uo^+(\padic)$
with \eqref{eitoti} induces an isomorphism 
\begin{align}
\noncomalg \isomto T\left( \uo^-\left(\padic\right)\right)\otimes T\left(\uo^+\left(\padic\right)\right),
\end{align}
which we shall also denote by $\omcan$.

We are now in a position to define the map $\theta$ and state some of its fundamental properties.

\begin{thm}\label{pdivder}
There exists a morphism $\theta$ of $\noncomalg$ such that the following hold:
\begin{enumerate}
\item{The following diagram commutes:
\begin{small}
\begin{align*}
\xymatrix{
\noncomalg\ar@{.>}[rrrr]^{\theta}\ar[dd]_{\cong\omcan}&&&& \noncomalg\ar[dd]^{\cong\omcan}\\
&&&&\\
T(\ompa^-)\otimes T(\ompa^+)\ar[rrrr]^{\paopo}&&&&T(\ompa^-)\otimes T(\ompa^+)}
\end{align*}
\end{small}}

\item{The morphism $\theta$ is ``homogeneous of degree $1\otimes 1$", i.e. if $x_1$ and $x_2$ are homogeneous elements of $\Oig\otimes_R \freealgrt$ of degrees $d_1$ and $d_2$, then $\theta(x_1\otimes x_2)$ is homogeneous of degree $(d_1+1)\otimes (d_2+1)$.}

\item{On the commutative subalgebra $\Oig[T_1, \ldots, T_n]\otimes_{\Oig} \Oig[T_1, \ldots, T_n]$ of $\noncomalg$, $\theta$ is a derivation.}

\item{If $f(q) =\sum_{h = \left(h_{ij}\right)\in H\dual_{\geq 0}} \lp c(h) q^h\rp$ at the cusp at infinity $(L, H),$ with $c(h)\in \freealgrt\otimes\freealgrt$,
then
\begin{align}\label{qexpaction}
(\theta(f))(q) = \sum_{h = \left(h_{ij}\right)\in H\dual_{\geq 0}} \lp d(h)q^h\rp,
\end{align}
where
$d(h) = \sum_{i,j}h_{ij}c(h)\cdot(T_j\otimes T_i).$}
\end{enumerate}
\end{thm}

\begin{proof}
Define the morphism $\theta$ by
$\theta = \omcan^{-1}\circ \paopo\circ \omcan.$
Recall that for any positive integers $d$ and $e$, $\paopo$ maps each element of $T^d(\ompa^-)\otimes T^e(\ompa^+)$ to an element of $T^{d+1}(\ompa^-)\otimes T^{e+1}(\ompa^+)$.  So $\theta$ maps homogenous elements of degree $d\otimes e$ in $\noncomalg$ to homogeneous elements of degree $(d+1)\otimes (e+1)$ in $\noncomalg$.  It follows from the definition of $\paopo$ that $\theta$ is a $R$-derivation of the commutative subalgebra 
$\Oig\otimes_R R[T_1, \ldots, T_n]\otimes_R R[T_1, \ldots T_n]$
of $\noncomalg$.

Now, we will examine the action of $\theta$ over the Mumford object $\tateav$.  By Lemma \ref{keylemma},
the elements $\nabla(D(\gamma))(\omega(w))$ lie in $\uroot\subseteq \dr$ for each $\gamma\in H$ and $w\in W$.  Since $\nabla(D(\gamma))$ is an $R$-derivation and $\uroot$ an $R$-module, we in fact have that
\begin{align}
\nabla\left(D\left(\gamma\right)\right)\left(\omega^{\pm}\left(w\right)\right)\in\uroot
\end{align} for each $w\in W$.  

Note that
$\paopo \left(\omega^{\pm}\lp e_{i_1}\rp\otimes\cdots\otimes\omega^{\pm}\lp e_{i_r}\rp\right) = 0$
for each positive integer $r$ and $1\leq i_1, \ldots, i_r\leq n$.  Note that
$\paopo\left(f \cdot \omega^{\pm}\lp e_{i_1}\rp\otimes\cdots\otimes\omega^{\pm}\lp e_{i_r}\rp\right) = \omega^{\pm} \lp e_{i_1}\rp\otimes\cdots\otimes\omega^{\pm}\lp e_{i_r}\rp\cdot Df,$
for each section $f$ of $\Oig$.
If the value of $f$ at $\tateav$ is
$f(q) = f(\tateav) = \sum_{h = \left(h_{ij}\right)\in H\dual_{\geq 0}}a(h)q^h,$
then for the standard basis elements $e_{kl}\in H$,
\begin{align*}
\left(D\left(e_{kl}\right)\right)\left(f\left(q\right)\right)=\sum_{h = \left(h_{ij}\right)\in H\dual_{\geq 0}}a(h)\cdot \tr(e_{kl}h)q^h =  \sum_{h = \left(h_{ij}\right)\in H\dual_{\geq 0}}a(h)h_{lk}q^h.
\end{align*}
So over the Mumford object, we have
\begin{small}
\begin{align*}
\nabla\left(D\left(e_{kl}\right)\right)&\left(f(q)\cdot\left(\omega^{\pm}\lp e_{i_1}\rp\otimes\cdots\otimes\omega^{\pm}\lp e_{i_r}\rp\right)\right)\\
& = D\left(e_{kl}\right)(f(q))\cdot \left(\omega^{\pm}\lp e_{i_1}\rp\otimes\cdots\otimes\omega^{\pm}\lp e_{i_r}\rp\right) \cdot KS(D(e_{kl})\mod\splpa\\
& = \sum_{h = \left(h_{ij}\right)\in H\dual_{\geq 0}}a(h)h_{lk}q^h\cdot \omega^+(e_k)\otimes\omega^-( e_l)\mod\splpa.
\end{align*}
\end{small}
Therefore,
$\nabla f(q) =  \sum_{h = \left(h_{ij}\right)\in H\dual_{\geq 0}}\sum_{1\leq l, k, \leq n}a(h)h_{lk}q^h\cdot \omega^{\pm}(e_k)\otimes \omega^{\pm}(e_l)\mod\splpa,$
so
\begin{align*}
\paopo f (q) =  \sum_{h = \left(h_{ij}\right)\in H\dual_{\geq 0}}\sum_{1\leq l, k, \leq n}a(h)h_{lk}q^h\cdot \omega^{\pm}(e_k)\otimes \omega^{\pm}(e_l).
\end{align*}
Thus, $\theta$ acts on $q$-expansions of automorphic forms as in Equation \eqref{qexpaction}.
\end{proof}

\begin{rmk}
The operators $\theta$ are a vector-valued generalization of Ramanujan's operator $\theta = q \frac{d}{dq}$ and Katz's analogous operator for Hilbert modular forms in \cite{kaCM}.
\end{rmk}

\begin{defi}
For each subrepresentation $Z$, define 
$\theta^Z: (\Oig)^{\rho}\rightarrow(\Oig)^Z$
by 
$\theta^Z: = \phi_Z\circ\theta^d\vert_{(\Oig)^{\rho}}.$
\end{defi}
From the definition of $\theta$, we see that
$\theta^Z = \omcan^{-1}\circ\paop^Z\circ \omcan.$

We note that, in practice, the above discussion of $\theta$ can often be simplified according to the properties of the particular representation with which one works.  For example, in our intended applications, we will only be interested in representations of the form $\rho_-\otimes \rho_+$ with $\rho_{\pm} = \det^{k_{\pm}}\otimes Sym^{l_{\pm}}$.  In this case, we will be able to restrict our discussion to the commutative subalgebra 
$\Oig\otimes_R R[T_1, \ldots, T_n]\otimes_R R[T_1, \ldots, T_n]$
of $\noncomalg$, on which $\theta$ will be a derivation.

Now we compare the values of $\theta$ to the values of $\paop$.
\begin{lem}\label{lereln}
Let $\rvaluedpt$ be an $R$-valued point of $\moduli(\padic)$ that satisfies condition \pstrictp, and let $\omega+$ and $\omega^-$ be elements of $\EAR^+$ and $\EAR^-$ respectively.  Let $c = (c_{ij})\in (GL_n\times GL_n)(R_0)$ satisfy
\begin{align}\label{ctrans}
\omega^{\pm}  = c^{\pm}\cdot\omega_{can}^{\pm}(\rvaluedpt).
\end{align}
Let $f$ be an automorphic form of weight $(\rho, V)$ over $R$, and let
$\tilde{f} = \omega_{can}^{-1}(f)\in \uo.$
Then
\begin{align}\nonumber
\paop(\tilde{f})(\rvaluedpt, \omega)= ((\rho\otimes\tau^d)(c^{-1}))(\theta^d f)(\rvaluedpt).
\end{align}
\end{lem}

\begin{proof}
We have
$\paop(\tilde{f}) = \omega_{can}\circ\theta^d (f).$
So by \eqref{ctrans},
\begin{align*}
\paop(\tilde{f})(\rvaluedpt, \omega)  = (\rho\otimes \tau^d)(c^{-1})\paop(\tilde{f})(\rvaluedpt, \omega_{can}(\rvaluedpt))
= ((\rho\otimes\tau^d)(c^{-1}))(\theta^d f)(\rvaluedpt)
\end{align*}
\end{proof}

The same method also gives a similar result when we restrict to subrepresentations $Z$ of $\rho\otimes\tau^d$:
\begin{cor}\label{correln}
With hypotheses as in Lemma \ref{lereln},
$\paop^Z(f)(\rvaluedpt, \omega)= ((\rho\otimes\tau^d)\vert_Z(c^{-1}))(\theta^Z \tilde{f})(\rvaluedpt).$
\end{cor}

As a corollary of Theorem \ref{paalgthm} and Lemma \ref{lereln}, we obtain the following theorem. \begin{thm}
Let $\rvaluedpt$ be an $R$-valued point of the moduli scheme $\moduli(\padic)$, and let $f$ be an automorphic form over $R$ of weight $(\rho, V)$, and let $Z$ be a subrepresentation of $\rho\otimes\tau^d$.  
Then
\begin{align*}
((\rho\otimes\tau^d)\vert_Z(c^{-1}))(\theta^Z \tilde{f})(\rvaluedpt) = \ir (\partial(\rho, e, \rvaluedpt, \omega, \Split(A/R), d)f).
\end{align*}
Therefore,
$((\rho\otimes\tau^d)\vert_Z(c^{-1}))(\theta^Z \tilde{f})(\rvaluedpt)$
lies in the $R$-submodule $Z = R\otimes_R Z$ of $R_0\otimes_R Z$.
\end{thm}

\section{Splitting of $\dr$ for CM abelian varieties}
\label{cmsection}
We now discuss conditions under which an abelian variety $A/R$ has a splitting $\dr(A/R) = \uo \oplus M$ over $R$
that simultaneously satisfies both condition \pstrict and \pstrictp.

Let $E\times E'$ be a product of CM algebras $E$ and $E'$, with
$E  = L_1\times\cdots\times L_{\me}$
and $E' = L_1'\times \cdots\times L_{\mep}'$
products of CM fields $L_i$, $L_i'$ such that each field $L_i$, $L_i'$ is a totally real extension of the CM field $\cmfield$ fixed in Section \ref{notation}.  Let 
$\gothsig = \gothsig_E\times \gothsig_{E'}$
be a CM type for $E\times E'$, where
$\gothsig_E  = \gothsig_{L_1}\times \cdots\times\gothsig_{L_{\me}}$
and $\gothsig_E'  = \gothsig_{L_1'}\times\cdots\times\gothsig_{{L'}_{\mep}}$
are $CM$ types for $E$ and $E'$.
\begin{defi}
We shall say the CM type $\gothsig$ {\bf is compatible with $\cmtype$} if the following two conditions are both met:
\begin{enumerate}
\item{For each $\sigma$ in $\gothsig_E$, $\sigma\lvert_\cmfield$ is in $\cmtype$.}
\item{For each $\sigma$ in $\gothsig_{E'}$, $\sigma\rvert_\cmfield$ is in $\bar{\cmtype}$.}
\end{enumerate}
\end{defi}

Suppose $(E\times E', \gothsig)$ is a CM type compatible with $(\cmfield, \cmtype)$.  So $
E = L_1\times\cdots\times L_{\me}$ and
$E' = L_1'\times \cdots\times L_{\mep}'$,
with each $L_i$ and each $L_i'$ a totally real extension of $K$, i.e. $L_i$ (resp. $L_i'$) is of the form $F_i\otimes K$ (resp. $F_i'\otimes K$) for some totally real fields $F_i$.  We use the following notation:
\begin{align*}
\scripto_E & = \scripto_{L_1}\times\cdots\times\scripto_{L_{\me}}\\
\scripto_{E'} & = \scripto_{L_1'}\times\cdots\times\scripto_{L_{\mep}}\\
\scripto_{E\times E'} &= \scripto_E\times\scripto_{E'}\\
\scripto &= \scripto_{F_1}\times\cdots\times\scripto_{F_{\me}}\times\scripto_{F_1'}\times\cdots\times\scripto_{F_{\mep}}.
\end{align*}

Let $R$ be a $\ZZ_{(p)}$-subalgebra of $\bar{\IQ}$ containing each $\scripto_{L_i}$ and each $\scripto_{L_i'}$.  For each CM type $\left(L_i, \gothsig_{L_i}\right)$, there is a natural ring isomorphism
\begin{align*}
\scripto_{L_i}\otimes R &\isomto \scripto_{F_i}\otimes R\times \scripto_{F_i}\otimes R\\
a\otimes r &\mapsto \phi_{\gothsig_i}(a\otimes r)\times\phi_{\gothsig_i}(\bar{a}\otimes r),
\end{align*}
where
\begin{align*}
\phi_{\gothsig_i}: \scripto_{L_i}\otimes R\rightarrow \scripto_{F_i}\otimes R\cong R^{\gothsig_i}\\
a\otimes r\mapsto \prod_{\sigma\in \gothsig_i} \sigma(a)r.
\end{align*}
is the projection 
A similar isomorphism holds for each $(L_i', \gothsig_{L_i'})$.  So there is a corresponding ring homomorphism
\begin{align*}
\scripto_{E\times E'}\otimes R\rightarrow \scripto\otimes R\times \scripto\times R\\
a\otimes r\mapsto (\phi_{\gothsig}(a)r, \phi_{\gothsig}(\bar{a})),
\end{align*}
and for any $\scripto_{E\times E'}\otimes R$-module $M$, there is a corresponding $\scripto\otimes R$-decomposition
$M\cong M(\gothsig)\oplus M(\gothsig),$
with
$M(\gothsig)  = \left\{m\in M\rvert a\cdot m = \phi_{\gothsig}(a)m \mbox{ for all } a\in \scripto_{E\times E'}\right\}$, 
$M(\bar{\gothsig})  = \left\{m\in M\rvert a\cdot m = \phi_{\gothsig}(\bar{a})m \mbox{ for all } a\in \scripto_{E\times E'} \right\}$.
If $M$ is an invertible $\scripto_{E\times E'}\otimes R$-module, then so are $M(\gothsig)$ and $M(\bar{\gothsig})$ as $\scripto\otimes R$-modules.

\begin{prop}\label{cmsplit}[Analogue of Key Lemma 5.1.27 in \cite{kaCM}] Let $(\gothsig, E\times E')$ be a CM type compatible with $(\cmtype, \cmfield)$, and let $R$ be as above.  If $A$ is an ordinary CM abelian variety of PEL type over $R$ with complex multiplication by $(\scripto_{E\times E'}, \gothsig)$, then
$\uo_{A/R} = \dr(\gothsig),$
and the splitting
$\dr(A/R) = \dr(\gothsig)\oplus \dr(\bar{\gothsig})$
simultaneously satisfies both \pstrict and \pstrictp.
\end{prop}
\begin{proof}
Let $H = \dr(A/R)$.  Since $(A, \gothsig)$ is a CM abelian variety over $R$,
$H = H(\gothsig)\oplus H(\bar{\gothsig})$
is an invertible $\scripto_{E\times E'}\otimes R$-module.  So $H(\gothsig)$ and $H(\bar{\gothsig})$ are invertible $\scripto\otimes R$-modules.  The action of $\scripto_{E\times E'}$ on $\Lie(A\dual/R)$ is through $a\mapsto \phi_{\bar{\gothsig}}$.  So in the exact sequence
$0\rightarrow \uo\rightarrow H(\gothsig)\oplus H(\bar{\gothsig})\rightarrow \Lie(A\dual/R)\rightarrow 0,$
$H(\gothsig)$ maps to $0$ in $\Lie(A/R)$.  So $H(\gothsig)$ is contained in $\uo$.  Since $A$ is ordinary, $\uo$ is an invertible $\scripto\otimes R$-module.  So $H(\gothsig) = \uo$.  The rest of the proof now follows exactly as in \cite{kaCM} Key Lemma 5.1.27.
\end{proof}

Let $U(1)^n$ denote the subgroup
\begin{align*}
\underbrace{U(1)\times \cdots\times U(1)}_{n \mbox{ times}}
\end{align*}
of $U(n)$.
Consider the natural embedding
$Sh\lp U(1)^n\times U(1)^n\rp\hookrightarrow Sh(U(n)\times U(n))\hookrightarrow Sh(U(n,n)).$
The points of $Sh\lp U(1)^n\times U(1)^n\rp$ parametrize abelian varieties isogenous to a CM abelian variety of the form 
\begin{align*}
\underbrace{A\times \cdots \times A}_{2n \mbox{ copies of } A}
\end{align*}
(where each copy of $A$ is one-dimensional) with CM type
\begin{align*}
\underbrace{\left(\cmfield, \cmtype\right)\times \cdots\times \left( \cmfield, \cmtype\right)}_{n \mbox{ times}}\times \underbrace{\left(\cmfield, \bar{\cmtype}\right)\times\cdots \times \left(\cmfield, \bar{\cmtype}\right)}_{n \mbox{ times}}.
\end{align*}
An abelian variety parametrized by a point of $Sh(U(n)\times U(n))$ is isogenous to an abelian variety parametrized by $Sh\lp U(1)^n\times U(1)^n\rp$.  So points of $Sh(U(n)\times U(n))$ parametrize CM abelian varieties compatible with $\cmtype$.  Since each of the abelian varieties in $Sh(U(n)\times U(n))$ is of CM type compatible with $\cmtype$, we obtain:

\begin{cor}\label{pullco}
Each abelian variety parametrized by $Sh(U(n)\times U(n))$ has a splitting simultaneously satisfying \pstrict{ and }\pstrictp.
\end{cor}

Corollary \ref{pullco} is crucial to our applications involving the pullback method to construct $L$-functions, which only requires evaluating automorphic forms at points of $U(n)\times U(n)$.

\begin{rmk}
Note that the proof of Proposition \ref{cmsplit} shows that there are also other abelian varieties which have a splitting simultaneously satisfying both conditions {\pstrict } and \pstrictp.  For example, suppose $A$ is a CM abelian variety with CM by a CM field $L = F\otimes K$ in which $p$ splits completely in $F$, where $\deg F = 2n = \dim A$.  Let
$\gothsig = \left\{\sigma_1, \ldots, \sigma_{2n}\right\}$
be a CM type for $L$ such that
$\sigma_i\lvert_\cmfield \in \cmtype$
for $1\leq i\leq n$ and
$\sigma_i\lvert_\cmfield \in \bar{\cmtype}$
for $n+1\leq i\leq 2n$.  Then the proof of Proposition \ref{cmsplit} shows that $\dr(\gothsig)\oplus\dr(\bar{\gothsig})$ gives a splitting of $\dr$ satisfying conditions \pstrict and \pstrictp{ simultaneously}.  Thus, there are also abelian varieties (over $R$) in $Sh(U(2V))$ not in $Sh(U(n)\times U(n))$ that have a splitting (over $R$) simultaneously satisfying \pstrict{ and} \pstrict.  For all our intended applications (i.e. construction of certain $p$-adic $L$-functions using the doubling method), however, only abelian varieties of the type in Proposition \ref{cmsplit} will be relevant.
\end{rmk}

\bibliography{eischen}

\end{document}